\documentclass[12pt,a4paper]{article}

\usepackage[dvips]{graphicx}
\usepackage{epsfig}
\usepackage{graphics,color}
\usepackage{enumerate}

\usepackage{amssymb,amsmath,amsthm,amscd,latexsym}

\usepackage[utf8]{inputenc}

\newtheorem{theo}{Theorem}[section]
\newtheorem{corol}[theo]{Corollary}
\newtheorem{prop}[theo]{Proposition}

\newtheorem{lem}[theo]{Lemma}

\theoremstyle{definition}
\newtheorem*{defi}{Definition}

\newtheorem*{claim}{Claim}

\newtheorem*{question}{Question}

\newtheorem*{propA}{Proposition A}
\newtheorem*{propB}{Proposition B}
\newtheorem*{propC}{Proposition C}

\newtheorem{etape}{Step}

\newenvironment{rem}{\bigskip\emph{Remark. }}{}
\newenvironment{rems}{\bigskip\emph{Remarks. }\begin{itemize}}{\end{itemize}}

\def\Rr{\mathbf{R}}
\let\RR\Rr

\def\Zz{\mathbf{Z}}

\let\NN\Nn

\def\cals{\mathcal{S}}
\def\calb{\mathcal{B}}
\def\calM{\mathcal{M}}

\let\bord\partial
\let\bydef\emph
\let\epsi\varepsilon

\def\var{mani\-fold}

\def\sing{\mathrm{sing}}
\def\reg{\mathrm{reg}}
\def\Ric{\mathrm{Ric}}
\def\Rm{\mathrm{Rm}}
\def\vol{\mathrm{vol}}

\def\Rmin{R_\mathrm{min}}
\def\Rmax{R_\mathrm{max}}
\def\cyl{\mathrm{cyl}}
\def\Int{\mathrm{Int}}
\def\bmax{b_\mathrm{max}}

\def\diam{\mathop{\rm diam}\nolimits}

\def\Ric{\mathop{\rm Ric}\nolimits}
\def\Rm{\mathop{\rm Rm}\nolimits}

\def\constst{\mathrm{const_{st}}}

\begin{document}

\title{Ricci flow on open 3-manifolds and positive scalar curvature}
\author{Laurent Bessi\`eres, G\'erard Besson and Sylvain Maillot}

\maketitle

\begin{abstract}
We show that an orientable $3$-dimensional manifold $M$ admits a complete riemannian metric of bounded geometry and uniformly positive scalar curvature if and only if there exists a finite collection $\mathcal F$ of spherical space-forms such that $M$ is a (possibly infinite) connected sum where each summand is diffeomorphic to $S^2\times S^1$ or to some member of $\mathcal F$. This result generalises G.~Perelman's classification theorem for compact $3$-manifolds of positive scalar curvature. The main tool is a variant of Perelman's surgery construction for Ricci flow.
\end{abstract}

\section{Introduction}
Thanks to G.~Perelman's proof~\cite{Per1,Per2,Per3} of W.~Thurston's Geometrisation Conjecture, the topological structure of compact $3$-manifolds is now well understood in terms of the canonical geometric decomposition. The first step of this decomposition, which goes back to H.~Kneser~\cite{Kneser}, consists in splitting such a manifold as a connected sum of \bydef{prime} $3$-manifolds, i.e.~$3$-manifolds which are not nontrivial connected sums themselves.

It has been known since early work of J.~H.~C.~Whitehead~\cite{whitehead:unity} that the topology of \emph{open} $3$-manifolds is much more complicated. Directly relevant to the present paper are counterexamples of P.~Scott~\cite{st:exotic} and the third author~\cite{maillot:examples} which show that Kneser's theorem fails to generalise to open manifolds, even if one allows infinite connected sums.

Of course, we need to explain what we mean by a possibly infinite connected sum. If $\mathcal X$ is a class of $3$-manifolds, we will say that a $3$-manifold $M$ is a \bydef{connected sum} of members of $\mathcal X$ if there exists a locally finite graph $G$ and a map $v\mapsto X_v$ which associates to each vertex of $G$ a copy of some manifold in $\mathcal X$, such that by removing from each $X_v$ as many $3$-balls as vertices incident to $v$ and gluing the thus punctured $X_v$'s to each other along the edges of $G$, one obtains a $3$-manifold diffeomorphic to $M$. This is equivalent to the requirement that $M$ should contain a locally finite collection of pairwise disjoint embedded $2$-spheres $\cals$ such that the operation of cutting $M$ along $\cals$ and capping-off $3$-balls yields a disjoint union of $3$-manifolds which are diffeomorphic to members of $\mathcal X$.\footnote{See below for the precise definition of \emph{capping-off}.}

Note that restricting this definition to finite graphs and compact manifolds yields a slightly nonstandard definition of a connected sum. In the usual definition of a finite connected sum, one has a \emph{tree} rather than a graph. It is well-known, however, that the graph of a finite connected sum (in the sense of the previous paragraph) can be made into a tree at the expense of adding extra $S^2\times S^1$ factors. The more general definition we have chosen for this paper seems more natural in view of the surgery theory for Ricci flow. It can be shown that the two definitions are equivalent even when the graph is infinite;
however, having a tree rather than a graph is only important for issues of \emph{uniqueness}, which will not be tackled here. 

The above-quoted articles~\cite{st:exotic,maillot:examples} provide examples of badly behaved open $3$-manifolds, which are not connected sums of prime $3$-manifolds.
From the point of view of Riemannian geometry, it is natural to look for sufficient conditions for a riemannian metric on an open $3$-manifold $M$ that rule out such exotic behaviour. One such condition was given by the third author in the paper~\cite{maillot:spherical}. Here we shall consider riemannian manifolds of positive scalar curvature. This class of manifolds has been extensively studied since the seminal work of A.~Lichnerowicz, M.~Gromov, B.~Lawson, R.~Schoen, S.-T.~Yau and others (see e.g.~the survey articles~\cite{gromov:sign,rosenberg:report}.)

Let $(M,g)$ be a riemannian manifold.  We denote by $\Rmin(g)$ the infimum of the scalar curvature of $g$.
We say that $g$ has
\bydef{uniformly positive scalar curvature} if $\Rmin(g)>0$. Of course, if $M$ is compact, then this amounts to insisting that $g$ should have positive scalar curvature at each point of $M$. 

A $3$-manifold is \bydef{spherical} if it admits a metric of positive constant sectional curvature. M.~Gromov and B.~Lawson~\cite{gl:spin} have shown that any compact, orientable $3$-manifold which is a connected sum of spherical manifolds and copies of $S^2\times S^1$ carries a metric of positive scalar curvature. Perelman~\cite{Per2}, completing pioneering work of Schoen-Yau~\cite{schoenyau} and Gromov-Lawson~\cite{GromovLawson}, proved the converse.

In this paper, we are mostly interested in the noncompact case.
 We say that a riemannian metric $g$ on $M$ has \bydef{bounded geometry} if it has bounded sectional curvature
and injectivity radius bounded away from zero. It follows from the Gromov-Lawson construction that if $M$ is a (possibly infinite) connected sum of spherical manifolds and copies of $S^2\times S^1$ such that there are finitely many summands up to diffeomorphism, then $M$ admits a complete metric  of bounded geometry and uniformly positive scalar curvature. We show that the converse holds, generalising  Perelman's theorem:
\begin{theo}\label{thm:positive scalar general}
Let $M$ be a connected, orientable $3$-manifold which carries a complete riemannian metric of bounded
geometry and uniformly positive scalar curvature. Then there is a finite collection $\mathcal F$ of spherical manifolds such that $M$ is a connected sum of copies of $S^2\times S^1$ or members of $\mathcal F$.
\end{theo}

In fact, the collection $\mathcal F$ depends only on bounds on the geometry and a lower bound for the scalar curvature (cf.~Corollary~\ref{corol:finiteness}.)

Our main tool is R.~Hamilton's Ricci flow. Let us give a brief review of the analytic theory of Ricci flow on complete manifolds.
The basic short time existence result is due to W.-X.~Shi~\cite{shi:complete2}: if $M$ is a $3$-manifold and $g_0$ is a complete riemannian metric on $M$ which has bounded sectional curvature, then there exists $\epsi>0$ and a Ricci flow $g(\cdot)$ defined on $[0,\epsi)$ such that $g(0)=g_0$, and for each $t$, $g(t)$ is also complete of bounded sectional curvature. 

For brevity, we say that a Ricci flow $g(\cdot)$ has a given property $\mathcal{P}$ if for each time $t$, the riemannian metric $g(t)$ has property $\mathcal{P}$. Hence the solutions constructed by Shi are complete Ricci flows with bounded sectional curvature. This seems to be a natural setting for the analytical theory of Ricci flow.\footnote{However, there have been attempts to generalise the theory beyond this framework, see e.g.~\cite{xu:existence,simon:noncollapsed}.}

Uniqueness of complete Ricci flows with bounded sectional curvature is due to B.-L.~Chen and X.-P.~Zhu~\cite{Che-Zhu2}.

We shall provide a variant of Perelman's surgery construction for Ricci flow, which has the advantage of being suitable for generalisations to open manifolds.
Perelman's construction can be summarised as follows. Let $M$ be a closed, orientable $3$-manifold. Start with an arbitrary metric $g_0$ on $M$. Consider a maximal Ricci flow solution $\{g(t)\}_{t\in [0,T_{max})}$  with initial condition $g_0$. If $T_{max}=+\infty$, there is nothing to do. Otherwise, one analyses the behaviour of $g(t)$ as $t$ goes to $T_{max}$ and finds an open subset $\Omega\subset M$ where a limiting metric can be obtained. If $\Omega$ is empty, then the construction stops. Otherwise the ends of $\Omega$ have a special geometry: they are so-called $\epsi$-horns. Removing neighbourhoods of those ends and capping-off $3$-balls with nearly standard geometry, one obtains a new closed, possibly disconnected riemannian $3$-manifold. Then one restarts Ricci flow from this new metric and iterates the construction. In order to prove that finitely many surgeries occur in any compact time interval, Perelman makes crucial use of the finiteness of the volume of the various riemannian manifolds involved.

When trying to generalise this construction to open manifolds, one encounters several difficulties. First, the above-mentioned volume argument breaks down. Second,  a singularity with $\Omega=M$ could occur, i.e.~there may exist a complete Ricci flow with bounded sectional curvature defined on some interval $[0,T)$ and maximal \emph{among  complete Ricci flows with bounded sectional curvature}, such that when $t$ tends to $T$, $g(t)$ converges to, say, a metric of unbounded curvature $\bar g$. Then it is not known whether Ricci flow with initial condition $\bar g$ exists at all, and even if it does, the usual tools like the maximum principle may no longer be available. One can imagine, for example, an infinite sequence of spheres of the same radius glued together by necks  whose curvature is going to infinity. In this situation $(M,\bar g)$ would have no horns to do surgery on.

In order to avoid those difficulties, we shall perform surgery \emph{before} a singularity appears. To this end, we introduce a new parameter $\Theta$, which determines when surgery must be done (namely when the supremum $\Rmax$ of the scalar curvature reaches $\Theta$.) We do surgery on  tubes rather than horns. Furthermore, we replace the volume argument for non-accumulation of surgeries by a \emph{curvature} argument: the key point is that at each surgery time, $\Rmax$ drops by a definite factor (which for convenience we choose equal to $1/2$.) This, together with an estimate on the rate of curvature blow-up, is sufficient to bound from below the elapsed time between two consecutive surgeries.

The idea of doing surgery before singularity time is not new: it was  introduced by R.~Hamilton in his paper~\cite{hamilton:isotropic} on $4$-manifolds of positive isotropic curvature.  Our construction should also be compared with that of G.~Huisken and C.~Sinestrari~\cite{hs:surgeries} for Mean Curvature Flow, where in particular there is a similar argument for non-accumulation of surgeries. Needless to say, we rely heavily on Perelman's work, in particular the notions of $\kappa$-noncollapsing and canonical neighbourhoods.

Our construction should have other applications. In fact, it has already been adapted by H.~Huang~\cite{huang:isotropic} to complete $4$-dimensional manifolds of positive isotropic curvature, using work of B.-L.~Chen, S.-H.~Tang and X.-P.~Zhu~\cite{ctz:complete} in the compact case.

Remaining informal for the moment, we provisionally define a \bydef{surgical solution} as a sequence of Ricci flow solutions $\{(M_i,g_i(t))\}_{t \in [t_i,t_{i+1}]}$, with \linebreak 
$0=t_0< \cdots < t_i< \cdots\leqslant +\infty$ is discrete in $\Rr$, such that $M_{i+1}$ is obtained from $M_i$ by splitting along a locally finite collection of pairwise disjoint 
embedded $2$-spheres, capping-off $3$-balls and removing components which are 
spherical or diffeomorphic to $\Rr^3$, $S^2\times S^1$, $S^2\times\Rr$, $RP^3\# RP^3$  or a punctured $RP^3$. 
If $M_{i+1}$ is nonempty, we further require that  $\Rmin(g_{i+1}) \geqslant \Rmin(g_i)$ at time $t_{i+1}$. The formal definition of surgical solutions will be given in Section~\ref{sec:surgical}.

The components that are removed at time $t_{i+1}$ are said to \emph{disappear}. If all 
components disappear, that is if $M_{i+1}=\emptyset$, we shall say that the surgical 
solution becomes \bydef{extinct} at time $t_{i+1}$. In that case, it is straightforward to reconstruct the topology of the original manifold $M_0$ as a connected sum of the disappearing components (cf.~Proposition~\ref{prop:topo} below.) Since $\Rr^3$, $S^2\times\Rr$ and punctured $RP^3$'s are themselves connected sums of spherical manifolds (in fact infinite copies of $S^3$ and $RP^3$,) the upshot is that $M_0$ is a connected sum of spherical manifolds and copies of $S^2\times S^1$.

A simplified version of our main technical result follows. 
\begin{theo}\label{thm:existence surg}
Let $M$ be an orientable $3$-manifold. Let $g_0$ be a complete riemannian metric on $M$
which has bounded geometry. Then there exists a complete
surgical solution of bounded geometry defined on $[0,+\infty)$, with initial condition $(M,g_0)$.
\end{theo}

When in addition we assume that $g_0$ has uniformly positive scalar curvature, we get (from the maximum principle and the condition that surgeries do not decrease $\Rmin$) an \emph{a priori} lower bound for $\Rmin$ which goes to infinity in finite time. This implies that surgical solutions given by Theorem~\ref{thm:existence surg} are automatically extinct. As a consequence, any $3$-manifold satisfying the hypotheses of Theorem~\ref{thm:positive scalar general} is a connected sum of spherical manifolds and copies of $S^2\times S^1$. However, we also need to prove finiteness of the summands up to diffeomorphism. Below we state a more precise result, which will suffice for our needs.

We say that a riemannian metric $g_1$ is $\epsi$-\bydef{homothetic} to some riemannian metric $g_2$ if there 
exists $\lambda>0$ such that $\lambda g_1$ is $\epsi$-close to $g_2$ in 
the $C^{[\epsi^{-1}]+1}$-topology. 
A riemannian metric which is $\epsi$-homothetic to a round metric (i.e.~a metric of constant sectional curvature $1$) is said to be $\epsi$-\bydef{round}.

\begin{theo}\label{thm:existence surg quant}
For all $\rho_0,T>0$ there exist $Q,\rho>0$ such that
if $(M_0,g_0)$ is a complete riemannian orientable $3$-manifold which has sectional curvature bounded in absolute value by $1$ and injectivity radius greater than or equal to $\rho_0$, then there exists a complete
surgical solution defined on $[0,T]$, with initial condition $(M_0,g_0)$, sectional curvature bounded in absolute value by $Q$ and injectivity radius greater than or equal to $\rho$ and such that all spherical disappearing components have scalar curvature at least $1$, and are $10^{-3}$-round or diffeomorphic to $S^3$ or $RP^3$.
\end{theo}

Let us explain why this stronger conclusion implies that there are only finitely many disappearing components up to diffeomorphism. By definition, nonspherical disappearing components belong to a finite number of diffeomorphism classes. Now by the Bonnet-Myers theorem, $10^{-3}$-round  components with scalar curvature at least $1$ have diameter bounded above by some universal constant. Putting this together with the bounds on sectional curvatures and injectivity radius, the assertion then follows from Cheeger's finiteness theorem.

Remark that there is an apparent discrepancy between Theorems~\ref{thm:existence surg} and~\ref{thm:existence surg quant} in that in the former, the surgical solution is defined on $[0,+\infty)$ whereas in the latter it is only defined on a compact interval. However, Theorem~\ref{thm:existence surg}  can be formally deduced from Theorem~\ref{thm:existence surg quant} via iteration and rescalings (cf.~remark at the end of Section~\ref{sec:surgical}.)

Throughout the paper, we use the following convention:

\bigskip
{\em All $3$-manifolds considered here are orientable.}
\bigskip

Here is a concise description of the content of the paper: in Section~\ref{sec:surgical}, we give some definitions, in particular the formal definition of surgical solutions, and show how to deduce Theorem~\ref{thm:positive scalar general} from Theorem~\ref{thm:existence surg quant}. The remainder of the article (except the last section) is devoted to the proof of Theorem~\ref{thm:existence surg quant}.  In Section~\ref{sec:metric surgery}, we discuss Hamilton-Ivey curvature pinching, the standard solution, and prove the Metric Surgery Theorem, which allows to perform surgery. In Section~\ref{sec:cn}, we recall some definitions and results on $\kappa$-noncollapsing, $\kappa$-solutions, and canonical neighbourhoods, and fix some constants that will appear throughout the rest of the proof.

In Section~\ref{sec:rdelta}, we introduce the important notion of $(r,\delta,\kappa)$-surgical solutions. These are special surgical solutions satisfying various estimates, and with surgery performed in a special way, according to the construction of Section~\ref{sec:metric surgery}. We state an existence theorem for those solutions, Theorem~\ref{thm:existence 1 precis},
which implies Theorem~\ref{thm:existence surg quant}. Then we reduce Theorem~\ref{thm:existence 1 precis} to three propositions, called A, B, and C. Sections~\ref{sec:coupure} through~\ref{sec:proofC} are devoted to the proofs of Propositions A, B, C, together with some technical results that are needed in these proofs.

Section~\ref{sec:general} deals with generalisations of Theorem~\ref{thm:existence surg quant}. One of them is an equivariant version, Theorem~\ref{thm:existence surg equi}, which implies a classification of $3$-manifolds admitting metrics of uniformly positive scalar curvature whose universal cover has bounded geometry. We note that equivariant Ricci flow with surgery in the case of finite group actions on closed $3$-manifolds has been studied by J.~Dinkelbach and B.~Leeb~\cite{dl:equi}. We follow in part their discussion; however, things are much simpler in our case, since we are mainly interested in the case of \emph{free} actions. We also give a version of Theorem~\ref{thm:existence 1 precis} with extra information on the long time behaviour. This may be useful for later applications. Finally, we review some global and local compactness results for Ricci flows in two appendices.

\paragraph{Acknowledgements} The authors wish to thank the Agence Nationale de la Recherche for its support under the programs F.O.G.~(ANR-07-BLAN-0251-01) and GROUPES (ANR-07-BLAN-0141). The third author thanks the Institut de Recherche Mathématique Avancée, Strasbourg who appointed him while this work was done. 

We warmly thank Joan Porti for numerous fruitful exchanges. The idea of generalising Perelman's work to open manifolds was prompted by conversations with O.~Biquard and T.~Delzant. We also thank Ch. Boehm, J. Dinkelbach, B. Leeb, T. Schick, B. Wilking and H. Weiss.

\section{Surgical solutions}\label{sec:surgical}
Let $M$ be a possibly noncompact, possibly disconnected (orientable) $3$-manifold. 

\subsection{Definitions}

\begin{defi}
Let $I\subset\Rr$ be an interval. An \bydef{evolving Riemannian manifold} is a pair
$\{(M(t),g(t))\}_{t\in I}$ where for each $t$, $M(t)$ is a (possibly empty, possibly disconnected) manifold and $g(t)$ a
riemannian metric on $M(t)$. We say that it is \bydef{piecewise $\mathcal C^1$-smooth}
 if there exists  $J\subset I$, which is discrete as a subset of $\Rr$, such that the following conditions are satisfied:

\begin{enumerate}
\item  On each
connected component of $I\setminus J$, $t\mapsto M(t)$ is constant, and $t\mapsto g(t)$ is
$\mathcal C^1$-smooth.
\item For each $t_0\in J$, $M(t_{0})=M(t)$ for any $t<t_{0}$ sufficiently close to $t_{0}$ and $t\to g(t)$ is left continuous at $t_{0}$. 

\item For each $t_0\in J\setminus \{\sup I\}$, $t \to (M(t),g(t))$ has a right limit at $t_{0}$, denoted $(M_{+}(t_{0}),g_{+}(t_{0}))$
\end{enumerate}
\end{defi}

A time $t\in I$ is \bydef{regular} if $t$ has a neighbourhood in $I$ where $M(\cdot)$ is constant
and $g(\cdot)$ is $\mathcal C^1$-smooth. Otherwise it is \bydef{singular}.

\begin{defi}
A piecewise $\mathcal{C}^1$-smooth evolving Riemannian $3$-manifold $\{(M(t),g(t))\}_{t\in I}$ is a
\bydef{surgical solution} of the Ricci Flow equation
\begin{equation}\label{eq:fdr}
\frac{dg}{dt} = -2\Ric_{g(t)}
\end{equation}

if the following statements hold:
\begin{enumerate}
\item Equation~(\ref{eq:fdr}) is satisfied at all regular times;
\item  For each singular time $t$ we have $\Rmin(g_+(t))\ge \Rmin(g(t))$;
\item For each singular time $t$ there is a locally finite collection $\mathcal S$ of disjoint embedded $2$-spheres in $M(t)$ and a manifold $M'$ such that
\begin{enumerate}
\item \label{item:capoff} $M'$ is obtained from $M(t)\setminus\mathcal S$ by capping-off $3$-balls;
\item $M_+(t)$ is a union of connected components of $M'$ and 
$g(t)=g_+(t)$ on $M(t) \cap M_+(t)$; 
\item Each component of $M'\setminus M_+(t)$ is spherical, or diffeomorphic to $\Rr^3$,
$S^2\times S^1$, $S^2\times\Rr$, $RP^3\# RP^3$  or a punctured $RP^3$.
\end{enumerate}
\end{enumerate}
\end{defi}

A component of $M'\setminus M_+(t)$ is said to \bydef{disappear at time $t$}. 

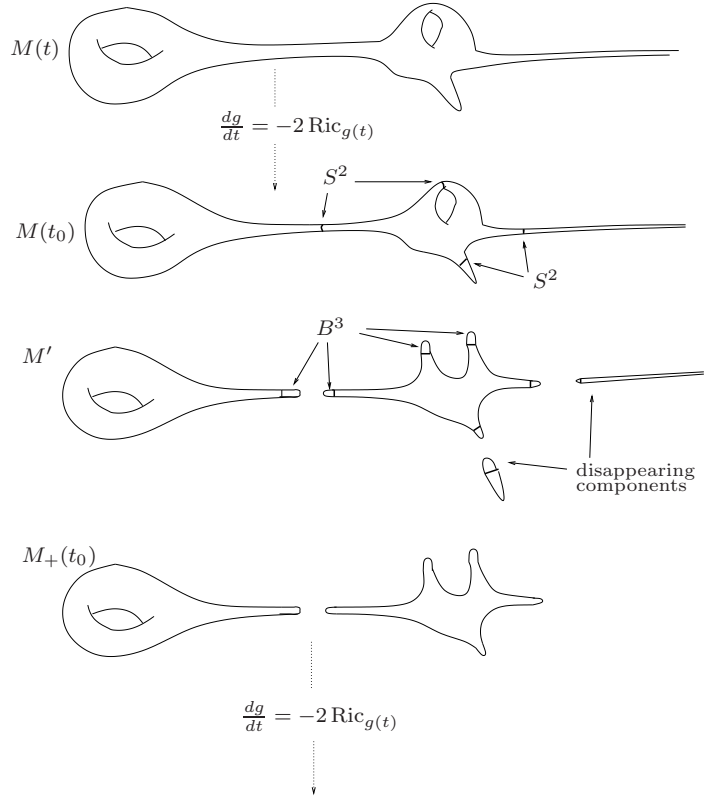
\begin{center}
 \input{surgical_solution.pstex_t}
\end{center}

An evolving riemannian manifold $\{(M(t),g(t))\}_{t\in I}$ is
\bydef{complete} (resp.~\bydef{has bounded geometry}) if for each $t\in I$ such that $M(t)\ne \emptyset$, the 
riemannian manifold $(M(t),g(t))$ is complete  (resp.~\bydef{has bounded geometry}).

\subsection{Deduction of Theorem~\ref{thm:positive scalar general} from Theorem~\ref{thm:existence surg quant}}

The purpose of this subsection is to explain how to deduce Theorem~\ref{thm:positive scalar general} from Theorem~\ref{thm:existence surg quant}. For this, we need a result about the evolution of $\Rmin$ in a surgical solution. For convenience, we take the convention that $\Rmin(t)$ is $+\infty$ if $M(t)$ is empty.

\begin{prop}\label{prop:R explose}
Let $(M(\cdot),g(\cdot))$ be a complete $3$-dimensional surgical solution with bounded sectional curvature
defined on an interval $[0,T)$. Assume that $\Rmin(0)\ge 0$. Then
$$\Rmin(t) \ge \frac{\Rmin(0)}{1-2t \Rmin(0)/3}.$$
\end{prop}

\begin{proof}
Follows from the evolution equation for scalar curvature, the maximum principle for complete Ricci flows of bounded curvature~\cite[Corollary 7.45]{cln} and the assumption that the minimum of scalar curvature is nondecreasing at singular times of surgical solutions.
\end{proof}

\begin{corol}\label{corol:extinction}
For every $R_0>0$ there exists $T=T(R_0)$ such that the following holds. Let  $(M(\cdot),g(\cdot))$ be a complete $3$-dimensional surgical solution defined on $[0,T]$, with bounded sectional curvature, and such that $\Rmin(0)\ge R_0$. Then  $(M(\cdot),g(\cdot))$ is extinct.
\end{corol}

We now recall the definition of \bydef{capping-off $3$-balls} to a $3$-manifold.

\begin{defi}
 Let $M,M'$ be $3$-manifolds. Let $\cals$ be a locally finite collection of embedded $2$-spheres in $M$. One says that $M'$ is obtained from $M \backslash \cals$ by \bydef{capping-off $3$-balls} if there exists a collection $\calb$ of 
$3$-balls such that $M'$ is the disjoint union 
$$M' = M \setminus \cals \bigsqcup \calb,$$
where each $S \in \cals$ has a tubular neighbourhood $V\subset M$ such that 
$V \setminus S$ has two connected components $V_-,V_+$ and there exists
 $B_-,B_+ \in \calb$ such that $V_- \bigsqcup B_-$ and $V_+ \bigsqcup B_+$ are 
$3$-balls in $M'$. Conversely, each $B \in \calb$ is included in such a $3$-ball 
of $M'$.  
\end{defi}

Note that it is implicit in the above definition that there is an orientation preserving  diffeomorphism, say $\phi_- : \partial B_- \to S\subset \partial V_-$, such that identifying $\partial B_-$ to the corresponding boundary of $V_-$ one obtains a $3$-ball. From \cite{smale:diff}, the differentiable structure of $M'$ does not 
depend of the above diffeomorphisms. Moreover, if $M'$ and $M''$ are obtained 
from $M \setminus \cals$ by capping off $3$-balls, one can choose the diffeomorphism 
from $M'$ to $M''$ to be the identity on $M \cap M' = M \cap M''$.

We shall need the following topological lemma:
\begin{prop}\label{prop:topo}
Let $\mathcal X$ be a class of closed $3$-manifolds.
Let $M$ be a $3$-manifold. Suppose that there exists a finite sequence of $3$-manifolds
\linebreak $M_0,M_1,\ldots, M_p$ such that $M_0=M$, $M_p=\emptyset$, and for each $i$, $M_i$ is
obtained from $M_{i-1}$ by splitting along a locally finite collection of pairwise disjoint, embedded $2$-spheres, capping off
$3$-balls, and removing some components which are connected
sums of members of $\mathcal X$.
Then each component of $M$ is a connected sum of members of $\mathcal X$.
 \end{prop}
 
\begin{proof}
We prove the result by induction on $p$. The case $p=1$ is immediate from the definition of a connected sum.

Supposing that the proposition is true for some $p$, we consider a sequence \linebreak $M_0,M_1,\ldots, M_{p+1}$ such that $M_0=M$, $M_{p+1}=\emptyset$, and for each $i$, $M_i$ is
obtained from $M_{i-1}$ by splitting along a locally finite collection of $2$-spheres, capping off
$3$-balls, and removing some components which are connected
sums of members of $\mathcal X$. By the induction hypothesis, $M_1$ is a connected sum of members of $\mathcal X$.

Let $\mathcal{S}$ be the collection of $2$-spheres involved in the process of turning $M_0$ into $M_1$. Let $\mathcal{B}$ be the collection of capped-off $3$-balls. Let $\mathcal{S}'$ be the collection of $2$-spheres involved in the connected sum decomposition of $M_1$. If $\mathcal{B}\cap\mathcal{S}'$ is empty, then the spheres of $\mathcal{S}'$ actually live in $M_0$, and the union of $\mathcal{S}$ and $\mathcal{S}'$ splits $M_0$ into prime summands homeomorphic to members of $X$. This observation reduces our proof to the following claim:

\begin{claim}
 $\mathcal{S}'$ can be made disjoint from $\mathcal{B}$ by an ambient isotopy.
\end{claim}

Let us prove the claim.  For each component $B_i$ of $\mathcal B$, we fix a $3$-ball $B'_i$ contained in the interior of $B_i$ and disjoint from $\mathcal{S}'$, and a collar neighbourhood $U_i$ of $\bord B_i$ in $M_1\setminus B_i$.
 Since $\{B_i\}$ is locally finite, we may ensure that the $U_i$'s are pairwise disjoint. Choose an ambient isotopy of $M_1$ which takes $B'_i$ onto $B_i$ and $B_i$ onto $B_i\cup U_i$ for each $i$. Then after this ambient isotopy, $\mathcal{S'}$ is still locally finite, and is now disjoint from $\mathcal{B}$.
\end{proof}

To see why these results imply Theorem~\ref{thm:positive scalar general}, take a $3$-manifold
$M$ and a complete metric $g_0$ of bounded geometry and uniformly positive
scalar curvature on $M$. By rescaling if necessary we can assume that the bound on the curvature is $1$. From the positive lower bound on $\Rmin(g_0)$ we get an \emph{a priori} upper bound $T$ for the extinction time of a surgical solution, using Corollary~\ref{corol:extinction}. Applying Theorem~\ref{thm:existence surg quant}, we get  numbers $Q,\rho$ and a surgical solution
$(M(\cdot),g(\cdot))$ with initial condition $(M,g_0)$ defined on $[0,T]$ satisfying the two additional conditions. This solution is extinct, and as we have already explained in the introduction, the disappearing components are connected sums of spherical manifolds and copies of $S^2\times S^1$, the summands belonging to some finite collection which depends only on the bounds on the geometry.
Let $\mathcal X$ be the collection of prime factors of the disappearing components. Let $0=t_0<t_1<t_2<\cdots < t_p=T$ be a set of regular times of
$(M(\cdot),g(\cdot))$ such that there is exactly one singular time between each pair
of consecutive $t_i$'s. The conclusion of  Theorem~\ref{thm:positive scalar general} now follows from Proposition~\ref{prop:topo} applied with $M_{i}=M(t_{i})$.

\subsection{More definitions}

\paragraph{Notation}
Let $n\ge 2$ be an integer and $(M,g)$ be a riemannian $n$-manifold.

For any $x\in M$, we denote by $\Rm(x):\Lambda^2T_{x}M \rightarrow \Lambda^2T_{x}M$ the curvature operator, and   $|\Rm (x)|$ its norm, which is also the maximum of the absolute values of the sectional curvatures at $x$. We let $R(x)$ denote the scalar curvature of $x$. The infimum (resp.~supremum) of the scalar curvature of $g$ on $M$ is denoted by $\Rmin(g)$ (resp.~$\Rmax(g)$).

We  write $d:M\times M\to [0,\infty)$ for the distance function associated to $g$. For $r>0$ we denote by $B(x,r)$ the open ball of radius $r$ around $x$. 
Finally, if $x,y$ are points of $M$, we denote by $[x y]$ a geodesic segment connecting $x$ to $y$. This is a (common) abuse of notation, since such a segment is not unique in general.

For closeness of metrics we adopt the conventions of \cite[Section 0.6.1]{B3MP}.

Let $\{(M(t),g(t))\}_{t\in I}$ be a surgical solution and $t\in I$. If $t$ is 
singular, one sets $M_\reg(t) :=M(t)\cap M_+(t)$ and denotes by $M_\sing(t)$ its complement, i.e.~$M_\sing(t):=M(t) \setminus M_\reg(t)=M(t)\setminus M_+(t)$. 
If $t$ is regular, then $M_\reg(t) = M(t)$ and $M_\sing(t)=\emptyset$.

At a singular time, connected components of $M_\sing(t)$ belong to three types:
\begin{enumerate}
\item components of  $M(t)$ which are disappearing components of $M'$,
\item  closures of components of $M(t) \setminus \cals$ which give, after being capped-off, disappearing components of $M'$, and
\item embedded $2$-spheres of $\cals$.
\end{enumerate}
In particular, the boundary of $M_\sing(t)$ is contained in $\cals$.

 We say that a pair $(x,t)\in
M\times I$ is \bydef{singular} if $x\in M_\sing(t)$; otherwise we call $(x,t)$ \bydef{regular}.

\begin{defi}
Let $t_0$ be a time, $[a,b]$ be an interval containing $t_0$ and $X$ be a subset of $M(t_0)$ such that for every $t\in [a,b)$, we have $X\subset M_\reg(t)$. Then the set $X\times [a,b]$ is
\bydef{unscathed.} Otherwise, we say that $X$ is \bydef{scathed}.
\end{defi}

\begin{rems} 
\item[1)] In the definition of `unscathed', we allow the final time slice to contain singular points, i.e.~we may have $X \cap M_\sing(b) \ne \emptyset$. The point 
is that if $X \times [a,b]$ is unscathed, then $t \mapsto g(t)$ evolves smoothly 
by the Ricci flow equation on all of $X \times [a,b]$.\\
\item[2)] Assume that $X \times [a,b]$ is scathed. Then there is $t\in[a,b)$ and 
$x \in X$ such that $x \notin M_\reg(t)$. Assume that $t$ is closest to 
$t_0$ with this property. If $t >t_0$ then $x \in M_\sing(t)$ and disappears at time 
$t$ unless $x \in \partial M_\sing(t)$ or if the component of $M_\sing(t)$ which 
contains $x$ is a sphere $ S \in \cals$.
If $t<t_0$, then $x \in M_+(t) \setminus M_\reg(t)$ is in one of the $3$-balls that are added at time $t$.

\end{rems}

\paragraph{Notation}
For $t\in I$ and $x\in M(t)$ we use the notation $\Rm(x,t)$,
$R(x,t)$ to denote the curvature operator and the scalar curvature respectively. For brevity we set $\Rmin(t):=\Rmin(g(t))$ and $\Rmax(t):=\Rmax(g(t))$.

We use $d_t(\cdot,\cdot)$ for the distance function associated to $g(t)$. The ball of radius $\rho$ around $x$ for $g(t)$ is denoted by $B(x,t,\rho)$.

For the definition of closeness of evolving riemannian manifolds, we refer to \cite[ Section 0.6.2.]{B3MP}.

\begin{defi}
Let  $t_0\in I$ and $Q>0$. The 
\bydef{parabolic rescaling} with factor $Q$ at time $t_0$ is the evolving manifold $\{(\bar M(t),\bar g(t))\}$ where
$\bar M(t)= M(t_0+t/Q)$, and
$$\bar g(t) = Q\,g(t_0+\frac{t}{Q}).$$
\end{defi}

Finally we remark that Theorem \ref{thm:existence surg} follows by iteration of Theorem \ref{thm:existence surg quant} via parabolic rescalings. Hence in the sequel, we focus on proving Theorem \ref{thm:existence surg quant}.

\section{Metric surgery}\label{sec:metric surgery}

\subsection{Curvature pinched toward positive}\label{sec:courbure 3}
Let $(M,g)$ be a $3$-manifold and $x\in M$ be a point. We denote by $\lambda(x)\ge \mu(x) \ge \nu(x)$ the eigenvalues of the curvature operator $\Rm (x)$. By our definition, all sectional curvatures lie in the interval $[\nu(x),\lambda(x)]$. Moreover, $\lambda(x)$ (resp.~$\nu(x)$) is the maximal (resp.~minimal) sectional curvature at $x$. If $C$ is a real number, we sometimes write $\Rm (x)\ge C$ instead of $\nu(x)\ge C$. Likewise, $\Rm(x) \le C$ means $\lambda(x)\le C$.

It follows that the eigenvalues of the Ricci tensor are equal to $\lambda+\mu$, $\lambda+\nu$, and $\mu+\nu$; as a consequence, the scalar curvature $R(x)$ is equal to $2(\lambda(x)+ \mu(x)+ \nu(x))$.

For evolving metrics, we use the notation $\lambda(x,t)$, $\mu(x,t)$, and $\nu(x,t)$, and correspondingly write $\Rm (x,t)\ge C$ for $\nu(x,t)\ge C$, and $\Rm(x,t) \le C$ for $\lambda(x,t)\le C$.

Let $\phi$ be a nonnegative function. A metric $g$ on $M$ has $\phi$-\bydef{almost nonnegative curvature} if  $\Rm \ge -\phi(R)$.

Now we consider a familly of positive functions $(\phi_{t})_{t\geqslant 0}$ defined as follows. Set  $s_{t}:=\frac{e^2}{1+t}$ and define
$\phi_{t} : [-2s_{t},+\infty) \longrightarrow [s_{t},+\infty)$
as the reciprocal of the function $s \mapsto 2s(\ln(s) + \ln(1+t) -3)$.

Following \cite{Mor-Tia}, we use the following definition. 
\begin{defi}
Let $I\subset [0,\infty)$ be an interval, $t_0\in I$ and $\{g(t)\}_{t\in I}$ be an evolving metric on $M$. We say that $g(\cdot)$ has \bydef{curvature
pinched toward positive at time $t_0$ }  if for all $x \in M$ we have
\begin{gather}
R(x,t_0)  \geqslant -\frac{6}{4t_0+1}, \label{eq:pinching 1} \\
\Rm (x,t_0) \geqslant -\phi_{t_0}(R(x,t_0)). \label{eq:pinching 2}
\end{gather}
We say that $g(\cdot)$ has  \bydef{curvature
pinched toward positive} if it has curvature
pinched toward positive at each $t \in I$. 
\end{defi}

Remark that if $|\Rm(g(0))| \le 1$, then $g(\cdot)$ has curvature
pinched toward positive at time $0$. Next we state a result due to Hamilton and Ivey in the compact case. For a proof of the general case, see~\cite[Section 5.1]{Cho-Kno2}.
\begin{prop}[Hamilton-Ivey pinching estimate]\label{prop:ham ivey}
Let $a,b$ be two real numbers such that $0\le a<b$. Let $(M,\{g(t)\}_{t\in [a,b]})$ be a complete Ricci flow with bounded curvature. If $g(\cdot)$ has  curvature pinched toward positive at time $a$, 
then $\{g(t)\}_{t\in [a,b]}$ has curvature pinched toward positive.
\end{prop}

The following easy lemmas will be useful.  
\begin{lem} \label{lem:phi decreasing}
\begin{enumerate}[i)]
\item  $\phi_{t}(s)=\frac{\phi_{0}((1+t)s)}{1+t}$.
\item   $\phi_t(s)\over s$ decreases to $0$ as $s$ tends to $+\infty$.
\item   ${\phi_0( s)\over s}={1\over 4}$ if $s = 4e^5$.
\end{enumerate}
\end{lem}

\begin{proof}\cite[Lemma 2.4.6]{B3MP}
\end{proof}
 
The main purpose of Property ii) is to ensure that limits of suitably rescaled evolving metrics 
with curvature pinched toward positive have nonnegative curvature operator 
(see Proposition \ref{prop:limite positive}). In the sequel we set $\bar s:=4e^5$.

\begin{lem}[Pinching Lemma]\label{lem:pinching}
 Assume that $g(\cdot)$ has \bydef{curvature
pinched toward positive} and let  $t \geqslant 0$, $ r>0$ be such that $(1+t)r^{-2}  \geqslant \bar s $. 
If  $R(x,t)\leqslant r^{-2}$ then \ $ |\Rm (x,t)|\leqslant r^{-2}\,.$
\end{lem}
\begin{proof} \cite[Lemma 2.4.7]{B3MP}
\end{proof}

\subsection{The standard solution}\label{sub:standard solution}

We recall the definition we used in \cite{B3MP} as initial condition of the standard solution. The functions $f,u$ below are chosen in \cite[Section 5.1]{B3MP}.

%

\begin{defi}
Let $d\theta^2$ denote the round metric of scalar curvature $1$ on $S^2$.

The \bydef{initial condition of the standard solution} is the riemannian manifold $\cals_0=(\textbf{R}^3,\bar g_0)$, where the metric $\bar g_0$ is given in polar coordinates by:
$$\bar g_0=e^{-2f(r)}g_{u}\,,$$
where
$$ g_{u} = dr^2+u(r)^2d\theta^2\,.$$

We also define $\cals_u:=(\Rr^3,g_u)$. The origin of $\Rr^3$, which is also the centre of spherical symmetry, will  be denoted by $p_0$.
\end{defi}

In particular, $(B(0,5),\bar g_{0})$ has positive sectional curvatures (see \cite[Lemma 5.1.2]{B3MP}), and the complement of $B(0,5)$ is isometric to $S^2 \times [0,+\infty)$.

Ricci flow with initial condition $\cals_0$ has a maximal solution defined
on $[0,1)$, which is unique among complete flows of bounded sectional curvature~\cite{Per2}. This solution is called the \bydef{standard solution}.

The \bydef{standard} $\epsi$-\bydef{neck}
is the riemannian product $S^2\times (-\epsi^{-1},\epsi^{-1})$,
where the $S^2$ factor is round of scalar curvature $1$. Its metric is denoted by $g_\cyl$. We fix a basepoint $*$ in $S^2\times \{0\}$.

\begin{defi}
Let $(M,g)$ be a riemannian $3$-\var\ and $x$ be a point of $M$.
A neighbourhood $N\subset M$ of $x$ is called an $\epsi$-\bydef{neck centred at} $x$ if $(N,g,x)$ is $\epsi$-homothetic to $(S^2\times (-\epsi^{-1},\epsi^{-1}),g_\cyl,*)$.
\end{defi}

If $N$ is an $\epsi$-neck and $\psi:N_\epsi \to N$ is a \bydef{parametrisation}, i.e.~a diffeomorphism such that some rescaling of $\psi^*(g)$ is $\epsi$-close to $g_\cyl$, then the sphere $\psi(S^2\times\{0\})$ is called a \bydef{middle sphere} of $N$.

\begin{defi} Let $\delta,\delta'$ be positive numbers. Let $g$ be a riemannian metric on $M$. Let $(U,V,p,y)$ be a $4$-tuple such that $U$ is an open subset of $M$, $V$ is a compact 
subset of $U$, $p \in \Int V$, $y\in \partial V$. Then $(U,V,p,y)$ is called a 
 \bydef{marked $(\delta,\delta')$-almost standard cap} if there exists a $\delta'$-isometry  $\psi : B(p_{0},5+\delta^{-1}) \rightarrow (U,R(y)g)$, 
sending $B(p_{0},5)$ to $\Int V$ and  $p_{0}$ to $p$. One calls $V$ the \emph{core} and $p$ the tip.
\end{defi}

\subsection{The metric surgery theorem}

\begin{theo}[Metric surgery]\label{thm:chirurgie metrique}
There exist $\delta_0>0$ and a function $\delta' : (0,\delta_{0}] \ni \delta
\mapsto \delta'(\delta) \in (0,\epsi_0/10]$ tending to zero as $\delta\to 0$,
with the following property:

Let $\phi$ be a nondecreasing, positive function; let $\delta\le \delta_0$; let $(M,g)$ be a riemannian $3$-manifold with $\phi$-almost
nonnegative curvature, and  $\{N_i\}$ be a locally finite collection of pairwise disjoint $\delta$-necks in $M$. Let $M'$ be a manifold obtained by cutting $M$ along the middle spheres of the $N_i$'s and capping off $3$-balls.

Then there exists a riemannian metric $g_+$ on $M'$ such that:
\begin{enumerate}
\item  $g_+=g$ on $M'\cap M$;
\item For each component $B$ of $M'\setminus M$, there exist $p \in \Int B$ and $y\in \partial B$ such that $(N' \cup B,B,p,y)$ is a marked $(\delta,\delta'(\delta))$-almost standard cap
with respect to $g_+$, where $N'$ is the `half' of $N$ adjacent to $B$ in $M'$; 
\item $g_+$ has $\phi$-almost nonnegative curvature.
\end{enumerate}
\end{theo}

\rem \ In the application of the above theorem, $M_+$ will be a submanifold of $M'$.

\begin{proof}
On $M'\cap M$ we set $g_{+}:=g$. On the added $3$-balls  we  define $g_{+}$ as follows.  
Let $N\subset M$ be one of the $\delta$-necks of the collection, and let $S$ be its middle sphere. By definition there exists a diffeomorphism
$ \psi: S^2\times (-\delta^{-1},\delta^{-1})\longrightarrow N$
and a real number $\lambda >0$ such that 
$ \vert\vert \psi^{\ast}\lambda g - g_{\cyl} \vert\vert < \delta$
in the $C^{[\delta^{-1}]+1}$-norm (see \cite{B3MP} for the details). 
Note that for each $y\in N$ we have that $\lambda /R(y)$ is $\delta'$-close to $1$ for some universal  $\delta'(\delta)$ tending to zero with $\delta$. 

Define $N_{+} := \psi (S^2\times [0,\delta^{-1}))$, i.e.~$N_+$ is the right half of the neck. Let $\Sigma \subset M'\setminus M $ be the $3$-ball that is capped-off  to it and 
$\Phi : \partial \Sigma \longrightarrow \partial N_{+}$ be the corresponding diffeomorphism. Our goal is to define $g_{+}$ on 
$\Sigma$ in such a way that $(N_{+} \cup_{\Phi} \Sigma,g_{+})$ is a $(\delta, \delta'(\delta ))$-almost standard cap with $\phi$-almost nonnegative curvature.

Let us introduce more notation. For $0\leq r_1 \leq r_2$, we let
$C[r_{1},r_{2}]$ denote the annular region of $\Rr^3$ defined by
the inequations $r_1\le r\le r_2$ in polar coordinates.
Observe that for all $3\le r_1<r_2$, the restriction of $g_{u}$
to $C[r_1,r_2]$ is isometric to the cylinder $S^2\times[r_{1},r_{2}]$
with scalar curvature $1$. We consider $B:=B(0,5) \subset \cals_{u}$. 

Set  $V_{-}:= \psi(S^2\times (-2,0])$ and $V_{+}:= \psi(S^2\times [0,2))$. Restrict $\psi$ on $S^2 \times (-2,\delta^{-1})$ to $V_{-}\cup N_{+}$, where 
$S^2\times  (-2,\delta^{-1})$ is now considered as the annulus $C(3,5+\delta^{-1}) \subset \cals_{u}$.  

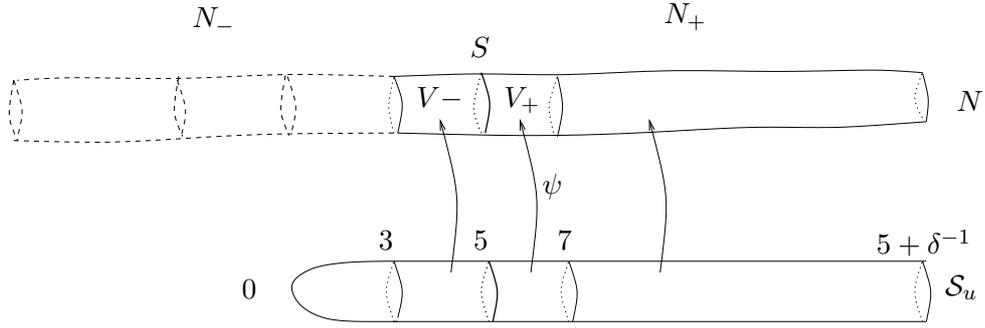
\begin{center}
\input{surgery1.pstex_t}
\end{center}

Let $\bar g:=\psi^{\ast}(\lambda g)$ be the pulled-back rescaled metric on $C(3,5+\delta^{-1})$. Note that $\vert\vert \bar g - g_{u}\vert\vert < \delta$ on this set and that 
$\bar g$ has $\phi$-almost nonnegative curvature. 

On $B(0,5+\delta^{-1})$ we define in polar coordinates 
$$ \bar g_{+} := e^{-2f}\left(\chi g_{u} + (1-\chi)\bar g \right)=\chi \bar g_{0} + (1-\chi)e^{-2f}\bar g$$
where 
 $\chi :  [0,5+\delta^{-1}] \rightarrow [0,1]$ is a  smooth function satisfying
 $$\left\{\begin{array}{ll}
\chi\equiv 1 & \textrm{ on } [0,3]\\
\chi'<0 & \textrm{ on } (3,4)\\
\chi\equiv 0 & \textrm{ on } [4,5+\delta^{-1}].
\end{array}\right.$$
Note that 
$$\left\{\begin{array}{ll}
\bar g_{+} = \bar g_{0} & \textrm{ on } B(0,3)\\
\bar g_{+} = e^{-2f}\bar g & \textrm{ on } C[4,5]\\
\bar g_{+}  = \bar g & \textrm{ on } C[5,5+\delta^{-1}).
\end{array}\right.$$
Finally set 

$$\left\{\begin{array}{ll}
 g_{+} := (\psi^{-1})^{\ast}(\lambda^{-1} \bar g)=g & \textrm{ on } N_{+}\\
 g_{+}  := \lambda^{-1}\bar g_{+} & \textrm{ on } B(0,5).
\end{array}\right.
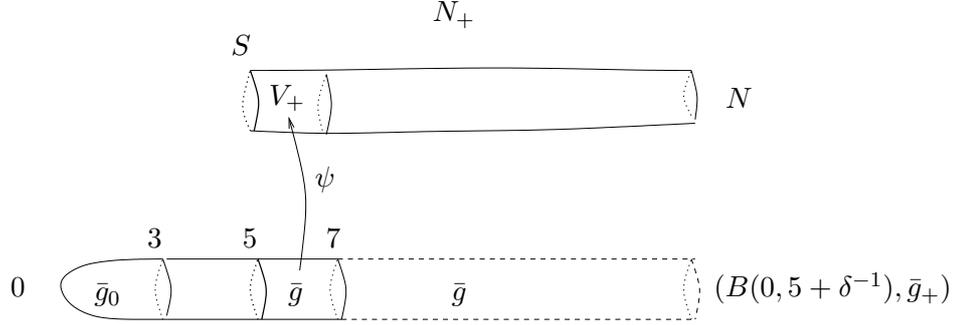
$$

\begin{center}
\input{surgery2.pstex_t} 
\end{center}

Let $p$ be the origin and $y$ be an arbitrary point of $\partial B$.
There remains to show that $((N_{+} \cup_{\psi_{\vert \partial B}} B,B,p,y),g_{+})$ is a $(\delta,\delta'(\delta ))$-almost standard cap, and has $\phi$-almost nonnegative curvature. 
It suffices clearly to consider $g_{+}$ on $B$, or $\bar g_{+}$ on $B(0,5)$. This is tackled by the following proposition (\cite[Proposition 5.2.2]{B3MP}), applied 
to $\bar g$ with the rescaled pinching function $s \mapsto \lambda^{-1} \phi(\lambda s)$:
\begin{prop}\label{prop:recollement}
There exists $\delta_1 >0$ and a function $\delta' : (0,\delta_{1}] \longrightarrow 
(0,\frac{\epsi_0}{10}]$ with limit zero at zero, having the following property: let $\phi$ be a nondecreasing positive function,  $0<\delta\leq \delta_1$ and $\bar g$ be a metric on $C(3,5) \subset \Rr^3$, with $\phi$-almost nonnegative curvature, such that $||\bar g-g_{u}||_{\mathcal C^{[{1\over\delta} ]+1}}<\delta$ on $C(3,5)$. Then
the metric
$$\bar g_+ = e^{-2f}\left(\chi g_{u} + (1-\chi)\bar g \right)$$  
has $\phi$-almost nonnegative curvature, and is $\delta'(\delta )$-close to  $\bar g_{0}$ on $B(0,5)$. 
\end{prop}

Setting $\delta_{0}:=\delta_{1}$ completes the proof of Theorem~\ref{thm:chirurgie metrique}.
\end{proof}

\section{$\kappa$-noncollapsing and canonical neighbourhoods}\label{sec:cn}

\subsection{$\kappa$-noncollapsing}
Let $\{(M(t),g(t))\}_{t\in I}$ be an evolving riemannian manifold. 
We say that a pair $(x,t)$ is a \bydef{point in spacetime} if $t\in I$ and $x\in M(t)$. For convenience we denote by $\calM$ the set of all such points. A (backward) \bydef{parabolic
neighbourhood} of a point $(x,t)$  in spacetime is a set of the form
$$P(x,t,r,-\Delta t) := \{(x',t') \in \calM \mid x'\in B(x,t,r),t' \in [t-\Delta t,t]\}.$$
In particular, the set $P(x,t,r,-r^2)$ is called a \bydef{parabolic ball}
of radius $r$. 

A parabolic neighbourhood $P(x,t,r,-\Delta t)$ is \bydef{unscathed} if 
$B(x,t,r) \times [t-\Delta t,t]$ is unscathed. In this case 
$P(x,t,r,-\Delta t)= B(x,t,r) \times [t-\Delta t,t]$. 

\begin{defi}
Fix $\kappa,r>0$. We say that  $(M(\cdot),g(\cdot))$ is $\kappa$-\bydef{collapsed} at $(x,t)$ on the scale $r$ if for all $(x',t')\in P(x,t,r,-r^2)$ one has $|\Rm (x',t') | \le r^{-2}$, and $\vol B(x,t,r) < \kappa r^n$.
Otherwise,  $(M(\cdot),g(\cdot))$ is $\kappa$-\bydef{noncollapsed} at $(x,t)$ on the scale $r$.

We say that $(M(\cdot),g(\cdot))$ is $\kappa$-\bydef{noncollapsed} on the scale $r$ if it is $\kappa$-non\-col\-lap\-sed on this scale at every point of $\calM$. 
\end{defi}

\subsection{Canonical neighbourhoods}

\begin{defi}
Let $(M,g)$ be a riemannian $3$-\var\ and $x$ be a point of $M$.
We say that $U$ is an $\epsi$-\bydef{cap centred at} $x$ if $U$ is the union of two sets $V,W$ such that $x\in \Int V$, $V$ is
diffeomorphic to $B^3$ or $RP^3\setminus B^3$,  $\overline{W} \cap V=\bord V$, and $W$ is an $\epsi$-neck. A subset $V$ as
above is called a \bydef{core} of $U$.
\end{defi}

Let $\epsi>0$, $C>>\epsi^{-1}$,  $\{(M(t),g(t))\}_{t \in I}$ be an evolving riemannian manifold and $(x_0,t_0)$ be a point in spacetime. \\

We call  \bydef{cylindrical flow} the manifold $S^2\times \Rr$
together with the product flow on $(-\infty,0]$, where the first factor is round, normalised so that the scalar curvature at time $0$ is identically $1$. We denote this evolving metric by $g_\cyl (t)$.

\begin{defi}
 An open subset $N\subset M(t_0)$ is called a \bydef{strong $\epsi$-neck}\footnote{We use `strong neck' to denote a subset of $M(t_0)$, rather than a subset of spacetime as other authors do.} 
 centred at $(x_0,t_0)$ if there is a number $Q>0$ such that $(N,\{g(t)\}_{t\in [t_0-Q^{-1},t_0]},x_0)$ is unscathed, and  after parabolic rescaling with factor $Q$ at time $t_0$, $\epsi$-close to $(S^2\times (-\epsi^{-1},\epsi^{-1}),\{g_\cyl (t)\}_{t\in [-1,0]},*)$.
\end{defi} 

\begin{rem} 
Let $Q>0$, and consider the flow $(S^2\times\Rr,Q g_\cyl (tQ^{-1}))$ restricted to $(-Q,0]$. Then for every $x\in S^2\times \Rr$ and every $\epsi>0$,  the point $(x,0)$ is centre of a strong $\epsi$-neck.\\
\end{rem}

We recall that \bydef{$\epsi$-round} means $\epsi$-homothetic to a metric of positive constant sectional curvature.

\begin{defi} \label{def:VC}
Let $U \subset M(t)$ be an open subset and $x\in U$ such that $R(x):=R(x,t)>0$. One says that $U$ is an 
\bydef{$(\epsi,C)$-canonical neighbourhood} centred at $(x,t)$ if: 
\begin{enumerate}
\item[A.] $U$ is of a strong $\epsi$-neck centred at $(x,t)$, or
\item[B.] $U$ is an $\epsi$-cap centred at $x$ for $g(t)$, or
\item[C.] $(U,g(t))$ has sectional curvatures $>C^{-1}R(x)$ and is diffeomorphic to $S^3$ or $RP^3$, or
\item[D.] $(U,g(t))$ is $\epsi$-round,
\end{enumerate}
and if moreover, the following estimates hold in cases A, B, C for $(U,g(t))$: 
There exists $r\in (C^{-1} R(x)^{-\frac12},C R(x)^{-\frac12})$ such that:
\begin{enumerate}
\item $\overline{B(x,r)} \subset U \subset B(x,2r)$; 
\item The scalar curvature function restricted to $U$ has values in a compact
subinterval of $(C^{-1} R(x),C  R(x))$;
\item If $B(y,r) \subset U$ and if $|\Rm| \leq r^{-2}$ on $B(y,r)$ then
\begin{equation}
C^{-1} < \frac{\vol B(y,r)}{r^3} ~; \label{eq:vol}
\end{equation}
\item \begin{equation}
|\nabla R(x)|< C R(x)^{\frac32}~, \label{eq:nabla R}
\end{equation}
\item 
\begin{equation}
|\Delta R(x) + 2|\Ric (x)|^2|< C R(x)^2 ~, \label{eq:Delta R}
\end{equation}
\item \begin{equation}
       |\nabla \Rm(x)|< C |\Rm(x)|^{\frac32}~, \label{eq:nabla Rm}
      \end{equation}
\end{enumerate}
\end{defi}

\begin{rems}
\item In case D, Estimates (i)-(vi) hold except maybe (iii) (consider e.g.~lens spaces.)
\item (i) implies that $\diam\ U$ is at most $4 r$, which in turn is bounded above by a function of $C$ and $R(x)$.
\item (iii) implies that $\vol\ U$ is bounded from below by $C^{-1}R(x)^{-3/2}$.  
\item Estimate (v) implies the following scale-invariant bound on the time-derivative of $R$ (at a regular time):
\begin{equation}
\vert \frac{\partial R}{\partial t} (x,t) \vert < C R(x,t)^2. \label{eq:dR dt}
\end{equation}
\item  We call  \bydef{($\epsi$,C)}-cap any $\epsi$-cap of $(M,g)$ which satisfies (i)-(vi),
\item In cases C and D, $U$ is diffeomorphic to a spherical manifold.
\item Cases C and D are not mutually exclusive.
\item Being the centre of an $(\epsi,C)$-canonical neighbourhood is an open property in spacetime: if $U \subset M(t)$ is unscathed on $(t-\alpha, t+\alpha)$ for some $\alpha>0$, then there exists a neighbourhood $\Omega$ of $(x,t)$ 
such that any $(x',t') \in \Omega$ is centre of an $(\epsi,C)$-canonical neighbourhood. In case A, one can use the same set $N=U$ and factor $Q$, but change the parametrisation so that the basepoint 
$*$ is sent to $x$ rather than $x_0$. Case B is similar.  Cases C, D are obvious. 
\item The same argument shows that being the centre of an $(\epsi,C)$-canonical neighbourhood is also an open property with respect to a change of metric in the $C^{[{\epsi}^{-1}]}$-topology. 
\end{rems}

\subsection{Fixing the constants}\label{sub:constantes}

In order to fix the constants, we recall some results of Perelman on $\kappa$-solutions and the standard solution.

\begin{theo}\label{thm:kappa VC}
For all $\epsi>0$ there exists $C_\mathrm{sol}=C_\mathrm{sol}(\epsi)$ such that if \linebreak $(M,\{g(t)\}_{t\in (-\infty,0]})$ is a $3$-dimensional $\kappa$-solution, then every  $(x,t)\in M\times (-\infty,0]$ is centre of an $(\epsi,C_\mathrm{sol})$-canonical neighbourhood.
\end{theo}

\begin{prop}\label{prop:standard kappa}
There exists $\kappa_\mathrm{st}>0$ such
that the standard solution is $\kappa_\mathrm{st}$-noncollapsed on all scales.
\end{prop}

\begin{prop}\label{prop:standard VC}
For every $\epsi>0$ there exists $C_\mathrm{st}(\epsi)>0$ such that
if $(x,t)$ is a point in the standard solution such that $t>3/4$ or $x\not \in B(p_0,0,\epsi^{-1})$, then $(x,t)$ has an $(\epsi,C_\mathrm{st})$-canonical neighbourhood. Moreover there is an estimate 
$ \Rmin(t) \geqslant \constst(1-t)^{-1}$ for some constant $\constst>0$.
\end{prop}

Next we recall two technical lemmas from~\cite{B3MP} which allow to fix universal constants $\epsi_0$ and $\beta_0$.

\begin{lem}\label{lem:glueneck}
There exists $\epsi_0>0$ such that the following holds. Let $\epsi\in (0,2\epsi_0]$.
Let $(M,g)$ be a riemannian $3$-manifold. Let $y_1,y_2$ be points of $M$. Let $U_1\subset M$ be an $\epsi$-neck centred at $y_1$ with parametrisation $\psi_1:S^2\times (-\epsi^{-1},\epsi^{-1}) \to U_1$ and middle sphere $S_1$. Let $U_2\subset M$ be a $10\epsi$-neck centred at $y_2$ with middle sphere $S_2$. Call $\pi:U_1\to (-\epsi^{-1},\epsi^{-1})$ the composition of $\psi_1^{-1}$ with the projection of $S^2\times (-\epsi^{-1},\epsi^{-1})$ onto its second factor.

Assume that $y_2\in U_1$ and $|\pi(y_2)| \le (2\epsi)^{-1}$. Then the following conclusions hold:
\begin{enumerate}
\item $U_2$ is contained in $U_1$;
\item The boundary components of $\bord \overline{U_2}$ can be denoted by $S_+,S_-$ in such a way that
$$\pi(S_-) \subset [\pi(y_2) - (10\epsi)^{-1} - 10,\pi(y_2) - (10\epsi)^{-1} + 10]\,,$$
and
$$\pi(S_+) \subset [\pi(y_2) + (10\epsi)^{-1} - 10,\pi(y_2) + (10\epsi)^{-1} + 10]\,;$$
\item The spheres $S_1,S_2$ are isotopic in $U_1$.
\end{enumerate}
\end{lem}

\begin{proof} \cite[Lemma 1.2.1]{B3MP}
\end{proof}

Let $K_\mathrm{st}$ be the supremum of the sectional curvatures of the standard solution on $[0,4/5]$.

\begin{lem}\label{lem:neck strengthening}
For all $\epsi\in (0,{10}^{-4})$ there exists $\beta=\beta(\epsi)\in (0,1)$ such that the following holds.

Let $a,b$ be real numbers satisfying $a<b<0$ and $|b|\le 3/4$, let $(M,g(\cdot))$ be a surgical solution defined on $(a,0]$, and $x\in M$ be a point such that:
\begin{itemize}
\item $R(x,b)=1$;
\item $(x,b)$ is centre of a strong $\beta\epsi$-neck;
\item $P(x,b,(\beta\epsi)^{-1},|b|)$ is unscathed and satisfies $|\Rm| \le 2K_\mathrm{st}$.
\end{itemize}
Then $(x,0)$ is centre of a strong $\epsi$-neck.
\end{lem}
 
\begin{proof}
We can argue exactly as in \cite[Lemma 2.3.5.]{B3MP}
\end{proof}

Fix $\epsi_0>0$ so that Lemma~\ref{lem:glueneck} holds. Let $\beta:=\beta(\epsi_0)$ be the constant given by
 Lemma~\ref{lem:neck strengthening}. Finally, define $C_0:=\max (100,2C_\mathrm{sol}(\epsi_0/2),2C_\mathrm{st}(\beta\epsi_0/2))$.

\begin{defi}
Let $r>0$. An evolving riemannian manifold $\{(M(t),g(t))\}_{t\in I}$ has property $(CN)_r$ if for all $(x,t)\in \calM$, if $R(x,t)\geq r^{-2}$, then $(x,t)$ admits an $(\epsi_0,C_0)$-canonical neighbourhood.
\end{defi}

\begin{defi}
Let $\kappa>0$.  An evolving riemannian manifold $\{(M(t),g(t))\}_{t\in I}$  has property $(NC)_\kappa$ if it is $\kappa$-noncollapsed on all scales less than $1$.
\end{defi}

\section{$(r,\delta,\kappa)$-surgical solutions}\label{sec:rdelta}

\subsection{Cutoff parameters and $(r,\delta)$-surgery}

\begin{theo}[Cutoff parameters]\label{thm:echelle de coupure}
For all $r,\delta>0$, there exist $h\in (0,\delta r)$ and $D>10$ such that
if $(M(\cdot),g(\cdot))$ is a complete surgical solution of bounded curvature defined on an interval $[a,b]$, with curvature pinched toward positive and satisfying $(CN)_r$, then the following holds:

Let  $t\in [a,b]$ and $x,y,z\in M(t)$  such that 
$R(x,t)\leq 2 /r^2$, $R(y,t)=h^{-2}$, and $R(z,t)\geq D/h^2$. Assume there is a curve  $\gamma$ connecting $x$ to $z$ via $y$, such that  each point of $\gamma$ with scalar curvature in $[2C_{0}r^{-2},{C_{0}}^{-1}Dh^{-2}]$ is centre of an 
$\epsi_{0}$-neck. Then $(y,t)$ is centre of a strong $\delta$-neck.
\end{theo}

This will be proved in Section~\ref{sec:coupure}. 
In the sequel we fix functions $(r,\delta)\mapsto h(r,\delta)$ and $(r,\delta)\mapsto D(r,\delta)$ with the above property. We set $\Theta:=2Dh^{-2}$. This number will be used as a curvature threshold for the surgery process.

\begin{defi}
We say that two real numbers $r,\delta$ are \bydef{surgery parameters} if $0<r<10^{-3}$ and $0<\delta<\min (\epsi_0,\delta_0)$. The \bydef{associated cutoff parameters} are $h:=h(r,\delta)$, $D:=D(r,\delta)$ and $\Theta:=2Dh^{-2}$.
\end{defi}

From now on, we fix a function $\delta' : (0,\delta_{0}] \longrightarrow 
(0,\epsi_0/10]$ as in the metric surgery theorem. A marked $(\delta,\delta'(\delta))$-almost standard cap
will be simply called a $\delta$-almost standard cap. An open subset $U$ of $M$ is called a $\delta$-almost standard cap if there 
exist $V$, $p$ and $y$ such that $(U,V,p,y)$ is a $\delta$-almost standard cap.

\begin{defi}
Fix surgery parameters $r,\delta$ and let $h,D,\Theta$ be
the associated cutoff parameters. Let $\{(M(t),g(t))\}_{t\in I}$ be an evolving riemannian manifold. Let $t_0\in I$ and $(M_+,g_+)$ be a (possibly empty) riemannian manifold. We say that $(M_+,g_+)$ \bydef{is obtained from $(M(\cdot),g(\cdot))$ by $(r,\delta)$-surgery at time $t_0$}  if the following conditions are satisfied:
\begin{enumerate}
\item $M_+$ is obtained from $M(t_0)$ by cutting along a locally finite collection of disjoint $2$-spheres, capping off $3$-balls, and possibly removing some components that are spherical or diffeomorphic to $\Rr^3$, $S^2\times\Rr$, $RP^3\setminus \{pt\}$, $RP^3\# RP^3$, $S^2\times S^1$. In addition, a spherical manifold $U$ can only be removed if it is contained in $M(t_0)$, and $(U,g(t_0))$ is $\epsi$-round and satisfies $R\ge 1$;
\item For all $x\in M_+\setminus M(t_0)$, there exists a $\delta$-almost standard cap $(U,V,p,y)$ in $M_+$, such that 
  \begin{enumerate}
  \item $x\in V$;
  \item $y\in M(t_0)$;
  \item $R(y,t_0)=h^{-2}$;
  \item $(y,t_0)$ is centre of a strong $\delta$-neck.
 \end{enumerate}
\item $\Rmax(g(t_0))=\Theta$, and if $M_+\neq\emptyset$, then $\Rmax(g_+)\le \Theta/2$.
\end{enumerate}
\end{defi}

\begin{defi}\label{defi:r delta solution}
Fix surgery parameters $r,\delta$ and let $h,D,\Theta$ be
the associated cutoff parameters. Let $I\subset [0,\infty)$ be an interval and $\{(M(t),g(t))\}_{t\in I}$ be a surgical solution. We say that it is an $(r,\delta)$-\bydef{surgical solution} if it has the following properties:
\begin{enumerate}
\item It has curvature pinched toward positive and satisfies $R(x,t)\le \Theta$ for all $(x,t)\in \calM$;
\item For every singular time $t_0\in I$,  $(M_+(t_0),g_+(t_0))$ is obtained from $(M(\cdot),g(\cdot))$ by $(r,\delta)$-surgery at time $t_0$;
\item Condition $(CN)_r$ holds.
\end{enumerate}

Let $\kappa>0$. An $(r,\delta)$-surgical solution which in addition satisfies
Condition~$(NC)_\kappa$ will be called an \bydef{$(r,\delta,\kappa)$-surgical solution.}
 \end{defi}

\subsection{Existence theorem for $(r,\delta,\kappa)$-surgical solutions}

Theorem~\ref{thm:existence surg quant} is implied by the
following result:

\begin{theo}\label{thm:existence 1 norm}
For every $\rho_0>0$ and $T\ge 0$, there exist $r,\delta,\kappa>0$ such that for any complete riemannian $3$-manifold $(M_0,g_0)$ with $| \Rm | \le 1$ and injectivity radius at least $\rho_0$, there exists an $(r,\delta,\kappa)$-surgical solution defined on $[0,T]$ satisfying the initial condition $(M(0),g(0))=(M_0,g_0)$.
\end{theo}

Theorem~\ref{thm:existence 1 norm} is itself a special case of
the following result, which has the advantage of being suitable for iteration:

\begin{theo}\label{thm:existence 1 precis}
For every $Q_0,\rho_0>0$ and all $0\le T_A<T_\Omega$, there exist $r,\delta,\kappa>0$ such that for any complete riemannian $3$-manifold $(M_0,g_0)$ which satisfies $| \Rm | \le Q_0$, has injectivity radius at least $\rho_0$, has curvature pinched toward positive at time $T_A$, there exists an $(r,\delta,\kappa)$-surgical solution defined on $[T_A,T_\Omega]$ satisfying the initial condition $(M(T_A),g(T_A))=(M_0,g_0)$.
\end{theo}

Note that in the statement of Theorem~\ref{thm:existence 1 norm} the assumption of almost nonnegative curvature is not necessary since it is automatic. We shall prove Theorem~\ref{thm:existence 1 precis} directly.

Our next aim is to reduce Theorem~\ref{thm:existence 1 precis} to three results, called Propositions~A, B, C, which are independent of one another. 

\begin{propA}\label{prop:propA}
There exists a universal constant $\bar\delta_A>0$ having the following property: let $r,\delta$ be surgery parameters, $a,b$ be positive numbers with $a<b$, and $\{(M(t),g(t))\}_{t\in (a,b]}$ be an $(r,\delta)$-surgical solution.  Suppose that $\delta\le\bar\delta_A$, and $\Rmax(b)=\Theta$.

Then there exists a riemannian manifold $(M_+,g_+)$, which is obtained from $(M(\cdot),g(\cdot))$ by $(r,\delta)$-surgery at time $b$, and in addition satisfies:
\begin{enumerate}
\item $g_+$ has $\phi_b$-almost nonnegative curvature;
\item $\Rmin(g_+) \geq \Rmin(g(b))$.
\end{enumerate}
\end{propA}

\rem The manifold $M_+$ may be empty. In this case, the second assertion in the conclusion follows from the convention $\Rmin(\emptyset)=+\infty$.

\begin{propB}
For all $Q_0,\rho_0,\kappa>0$ there exist $r=r(Q_0,\rho_0,\kappa)< 10^{-3}$ and $\bar\delta_B=\bar\delta_B(Q_0,\rho_0,\kappa)>0$ with the following property: let $\delta\le\bar\delta_B$, $0\le T_A<b$ and $(M(\cdot),g(\cdot))$ be a surgical solution defined on $[T_A,b]$ such that $g(T_A)$ satisfies $| \Rm | \le Q_0$ and has injectivity radius at least $\rho_0$.

 Assume that $(M(\cdot),g(\cdot))$ satisfies Condition~$(NC)_{\kappa/16}$, has curvature pin\-ched toward positive, and that for each singular time $t_0$, $(M_+(t_0),g_+(t_0))$ is obtained from $(M(\cdot),g(\cdot))$ by $(r,\delta)$-surgery at time $t_0$.

Then $(M(\cdot),g(\cdot))$ satisfies Condition~$(CN)_r$. 
 \end{propB}

\begin{propC}
For all $Q_0,\rho_0>0$ and all $0\le T_A< T_\Omega$ there exists $\kappa=\kappa(Q_0,\rho_0,T_A,T_\Omega)$ such that for all $0<r<10^{-3}$ there exists $\bar\delta_C=\bar\delta_C(Q_0,\rho_0,T_A,T_\Omega,r)>0$ such that the following holds.

For all $0<\delta\le\bar\delta_C$ and $b\in (T_A,T_\Omega]$, every $(r,\delta)$-surgical solution defined on $[T_A,b)$ such that 
$g(T_A)$ satisfies $| \Rm | \le Q_0$ and has injectivity radius at least $\rho_0$,  satisfies $(NC)_\kappa$.
\end{propC}

\begin{rem} The formulation of Proposition B, and its use below, are somewhat different in \cite{B3MP}. In the compact case, it is fairly easy to prove that 
the $(CN)_{r}$ property is open with respect to time (see \cite[Lemma 3.3.2]{B3MP}). This is not the case here. 
\end{rem}

\begin{proof}[Proof of Theorem~\ref{thm:existence 1 precis} assuming Propositions A, B, C]
 We start with two easy lemmas. The first one allows to control the density of surgery times by the surgery parameters.
\begin{lem}\label{lem:no acc}
Let $r,\delta$ be surgery parameters.
Let $\{(M(t),g(t))\}_{t\in I}$ be an $(r,\delta)$-surgical solution.
Let $t_1<t_2$ be two singular times. Then $t_2-t_1$ is bounded from below by a positive number depending only on $r,\delta$.
\end{lem}

\begin{proof}
We can suppose that $M(\cdot)$ is constant and $g(\cdot)$ is smooth on $(t_1,t_2]$. Since $\Rmax(t_2)=\Theta$ we can choose a point
$x \in M(t_2)$ such that $R(x,t_2)\ge \Rmax(t_2)-1$. Since $\Rmax(g_+(t_1)) 
\leq \Theta/2$, there exists $t_+ \in [t_1,t_2]$ maximal such that 
$\lim_{t\rightarrow t_+,t>t_+} R(x,t) = 
\Theta/2$. In particular, $(x,t)$ admits an 
$(\varepsilon_0,C_0)$-canonical neighbourhood for all $t\in (t_+,t_2]$. Integrating inequality \eqref{eq:dR dt} on $(t_+,t_2]$ gives a positive lower bound for $t_2-t_+$ depending only on $\Theta$, hence only on $r,\delta$.
\end{proof}

The second one says that $(NC)_\kappa$ is a closed condition:
\begin{lem}\label{lem:kappa ferme}
Let $(M(\cdot),g(\cdot))$ be a surgical solution defined on an interval $(a,b]$,  $x\in M(b)$ and $r,\kappa>0$. Suppose that for all $t\in (a,b)$, $x\in M(t)$, and $(M(\cdot),g(\cdot))$ is $\kappa$-noncollapsed at $(x,t)$ on all scales less than or equal to $r$.
Then it is $\kappa$-noncollapsed at $(x,b)$ on the scale $r$.
\end{lem}

\begin{proof}\cite[Lemma 2.1.5.]{B3MP}
\end{proof}

Let $Q_0,\rho_0>0$ and $0\le T_A<T_\Omega$. Proposition~C gives a constant $\kappa=\kappa(Q_0,\rho_0,T_A,T_\Omega)$.
Proposition~B gives constants $r,\bar\delta_B$ depending
on $\kappa$. We can assume $r^{-2} > 12Q_{0}$. Then apply Proposition~C again to get
a constant $\bar\delta_C$. Set  $\delta=\min(\bar\delta_A,\bar\delta_B, \bar\delta_C)$. Without loss of generality, we assume that $\kappa\le \kappa_{st}$.

From $r,\delta$ we get the cutoff parameters $h,D,\Theta$.

Let $(M_0,g_0)$ be a riemannian manifold which has $\phi_A$-almost nonnegative curvature, satisfies $\Rmin(g_0)\ge -6/(4T_A+1)$,
$| \Rm | \le Q_0$, and has injectivity radius at least $\rho_0$.

 Let  $\mathcal X$ be
the set of ordered pairs $(b,\{(M(t),g(t))\}_{t\in [T_A,b)})$ consisting
of a number $b\in (T_A,T_\Omega]$ and an $(r,\delta,\kappa)$-surgical solution  such that $(M(T_A),g(T_A))=(M_0,g_0)$. We first show that  $\mathcal X$ is nonempty. 
By standard results on the Ricci flow (see e.g.~\cite[Lemma 6.1]{cln}) there exists a complete Ricci flow solution $g(\cdot)$ defined on $[T_{A},T_{A}+(16Q_{0})^{-1}]$, such that 
$g(T_{A})=g_{0}$ and $|\Rm| \leqslant 2Q_{0}$ on the interval. By Proposition \ref{prop:ham ivey}, $g(\cdot)$ has curvature pinched toward positive on $[T_{A},T_{A}+(16Q_{0})^{-1}]$. As 
$R \leqslant 12Q_{0} < r^{-2} < \Theta$, $g(\cdot)$ satisfies Property (i) of an $(r,\delta)$-surgical solution, and Properties~(ii), (iii) are vacuously satisfied. By Proposition C, it satisfies 
$(NC)_{\kappa}$ on the interval. Hence $g(\cdot)$ is a $(r,\delta,\kappa)$-surgical solution on $[T_{A},T_{A}+(16Q_{0})^{-1}]$. 

The set  $\mathcal X$ has a partial ordering, defined by \linebreak $(b_1,(M_1(\cdot),g_1(\cdot))) \le (b_2,(M_2(\cdot),g_2(\cdot)))$ if $b_1\le b_2$ and $(M_2(\cdot),g_2(\cdot))$ is an extension of $(M_1(\cdot),g_1(\cdot))$.

We want to use Zorn's lemma to prove existence of a maximal
element in $\mathcal X$. In order to do this, we consider
an \emph{infinite chain,} i.~e.~an infinite sequence of numbers
$T_A<b_1<b_2<\cdots b_n<\cdots <T_\Omega$ and of $(r,\delta,\kappa)$-surgical solutions defined on the intervals $[T_A,b_n)$, and which
extend one another. In this way we get an evolving manifold $\{(M(t),g(t))\}$ defined on $[T_A,b_\infty)$, where $b_\infty$ is the supremum
of the $b_n$'s. By Lemma~\ref{lem:no acc}, the set of singular times is a discrete subset of $\Rr$, so  $\{(M(t),g(t))\}_{t\in [T_A,b_\infty)}$ is an $(r,\delta,\kappa)$-surgical solution, thus 
a majorant of the increasing sequence.

Hence we can apply Zorn's lemma. Let $(\bmax,(M(\cdot),g(\cdot)))\in \mathcal X$ be
a maximal element. Its scalar curvature lies between $-6$ and
$\Theta$, so it is bounded independently of $t$. Its curvature is pinched toward
positive so the sectional curvature is also bounded  independently of $t$. Using the Shi estimates, we deduce that all derivatives of the curvature are also bounded at time $\bmax$. This allows to take a smooth limit and extend $(M(\cdot),g(\cdot))$ to a surgical solution defined on $[T_A,\bmax]$, with $\Rmax(\bmax)\le\Theta$.  Condition~$(NC)_\kappa$ is still satisfied on this closed interval by Lemma~\ref{lem:kappa ferme}. Hence we can apply Proposition~B, which implies that Property~$(CN)_r$ is satisfied on $[T_A,\bmax]$.
We thus obtain an $(r,\delta,\kappa)$-surgical solution
on the closed interval  $[T_A,\bmax]$.

To conclude, we prove by contradiction that $\bmax=T_\Omega$. Assume
that $\bmax<T_\Omega$ and consider the following two cases:

\paragraph{Case 1} $\Rmax(\bmax)<\Theta$.

Applying the short time existence theorem for Ricci flow with initial metric $g(\bmax)$, we can extend $g(\cdot)$ to a surgical solution defined on an interval $[T_A,\bmax+\alpha)$ for some $\alpha>0$. We choose
$\alpha$ sufficiently small so that we still have $\Rmax(t)<\Theta$ on $[T_A,\bmax+\alpha)$. There are no singular times in $[\bmax,\bmax+\alpha)$, and by Proposition~\ref{prop:ham ivey} the extension satisfies the hypothesis that the curvature is pinched toward positive.

\begin{lem}\label{lem:open Kappa}
There exists $\alpha'\in (0,\alpha]$ such that Condition $(NC)_{\kappa/16}$ holds for $\{g(t)\}_{t\in [T_A,\bmax+\alpha')}$.
\end{lem}

\begin{proof}

Let $x\in M(\bmax)$, $t\in (\bmax,\bmax+\alpha')$ and $\rho\in (0,10^{-3})$ be such that $|\Rm| \le \rho^{-2}$ on $P(x,t,\rho,-\rho^2)$. Choosing $\alpha'$ small enough, we can ensure that 
$B(x,\bmax,\rho/2) \subset B(x,t,\rho)$ and moreover that $P(x,\bmax,\rho/2,-\rho^2/4) \subset P(x,t,\rho,-\rho^2)$. It follows that $|\Rm| \le 4 \rho^{-2}$ on $P(x,\bmax,\rho/2,-\rho^2/4)$. Since $(CN)_\kappa$ is satisfied up to time $\bmax$, we deduce that $\vol B(x,\bmax,\rho/2) \ge \kappa (\rho/2) ^3$. Again by proper choice of $\alpha'$,  $\vol_{g(t)} B(x,\bmax,\rho/2)$ is at least half  of $\vol B(x,\bmax,\rho/2)$. Hence
$$\vol B(x,t,\rho) \ge \vol_{g(t)} B(x,\bmax,\rho/2) \ge \frac12 \vol B(x,\bmax,\rho/2) \ge \frac{\kappa}{16} \rho ^3.$$ 
\end{proof}

Applying Proposition~B, we deduce that $\{(M(t),g(t))\}_{t\in [T_A,\bmax+\alpha')}$ is an $(r,\delta)$-surgical solution. By Proposition~C, it is an $(r,\delta,\kappa)$-surgical solution.
This contradicts maximality of $\bmax$.

\paragraph{Case 2}  $\Rmax(\bmax)=\Theta$.

Proposition~A yields a riemannian manifold $(M_+,g_+)$. If $M_+$ is empty, then the construction stops. Suppose $M_+\neq\emptyset$. Applying Shi's short time existence theorem for Ricci flow on $M_+$ with initial metric $g_+$, we obtain a positive number $\alpha$ and an evolving metric $\{g(t)\}_{t\in (\bmax,\bmax+\alpha)}$ on $M_+$ whose limit from the right as $t$ tends to $\bmax$ is equal to $g_+$. Since $\Rmax(g_+)\le \Theta/2$, we may also assume that $\Rmax$ remains bounded above by $\Theta$. By Proposition~\ref{prop:ham ivey} it has curvature pinched toward positive. Setting $M(t):=M_+$ for $t\in (\bmax,\bmax+\alpha)$, we obtain an evolving manifold $\{(M(t),g(t))\}_{t\in [T_A,\bmax+\alpha)}$ satisfying the first two properties of the definition of $(r,\delta)$-surgical solutions.

\begin{lem}\label{lem:open Kappa sing}
There exists $\alpha'\in (0,\alpha]$ such that Condition $(NC)_{\kappa/16}$ holds on $[T_A,\bmax+\alpha')$.
\end{lem}

\begin{proof}
Let $x\in M(\bmax)$, $t\in (\bmax,\bmax+\alpha')$ and $\rho\in (0,10^{-1})$ be such that $|\Rm| \le \rho^{-2}$ on $P(x,t,\rho,-\rho^2)$. If $B(x,t,\rho/2)$ is unscathed and stays so until $\bmax$, then we can repeat the argument used to prove Lemma~\ref{lem:open Kappa}. Otherwise it follows from the assumption $\kappa\le \kappa_{st}$ and properties of almost standard caps.
\end{proof}

Applying Proposition~B, we deduce that $\{(M(t),g(t))\}_{t\in [T_A,\bmax+\alpha')}$ is an $(r,\delta)$-surgical solution. By Proposition~C, it is an $(r,\delta,\kappa)$-surgical solution.
Again this contradicts the assumption that $\bmax$ should be maximal.
\end{proof}

\section{Choosing cutoff parameters}\label{sec:coupure}

In this section, we give some technical results necessary  to prove Theorem~\ref{thm:echelle de coupure}. Their statements are nearly identical to those of the corresponding results of \cite[Section 4]{B3MP}, \emph{surgical solutions} replacing \emph{Ricci flow with bubbling-off}. The proofs are 
also almost identical, the minor adaptations being precised below. 

\subsection{Bounded curvature at bounded distance} \label{subsec:curvature-distance proof}

We have the following technical lemmas, as in \cite[Section 4.2]{B3MP}:

\begin{lem}[Local curvature-distance lemma] \label{lem:courbure distance}
Let $(U,g)$ be a Riemannian manifold. Let $Q\geqslant 1$, $C>0$, $x\in U$ and set $Q_{x}= \mid R(x) \mid + Q$. Suppose that there exist
$y \in U$ such that $R(y) \geqslant 2Q_{x}$, and a minimising segment $[xy]$ where 
$$\mid \nabla R\mid \leqslant CR^{3/2} \qquad (*)$$
holds at each point of scalar curvature at least $Q$.
Then $d(x,y) \geqslant \frac{1}{2C\sqrt{Q_{x}}}$. 
\end{lem}

\begin{lem}[Local curvature-time lemma] \label{lem:courbure temps}
Let $(U,g(\cdot))$ be a Ricci flow defined on $[t_{1},t_{2}]$. Let  $Q\geqslant 1$, $C>0$, $x\in U$ and set $Q_{x}= \mid R(x,t_{2}) \mid + Q$.
Suppose that there exists $t\in [t_{1},t_{2}]$ such that $R(x,t) \geqslant 2Q_{x}$, and that 
$$\mid \frac{dR}{dt} \mid \leqslant CR^2 \qquad (**)$$ 
holds at $(x,s)$ if $R(x,s) \geqslant Q$. Then $t_{2} - t \geqslant (2CQ_{x})^{-1}$.
\end{lem}

\begin{lem}[Local curvature-control lemma]\label{lem:courbure distance temps}
Let $Q>0$, $C>0$, $\epsi \in (0,2\epsi_0]$,  and $\{(M(t),g(t))\}_{t\in I}$ be a surgical solution on $M$. Let $(x_{0},t_{0}) \in M\times I$ and  set 
$Q_0=|R(x_0,t_0)|+Q$. 
Suppose that $P=P(x_0,t_0, \frac{1}{2C \sqrt{Q_0}}, -\frac{1}{8C Q_0})$ is unscathed and that each 
$(x,t)\in P$ with $R(x,t)\geqslant Q$ has an  $(\epsi,C)$-canonical neighbourhood. Then for all $(x,t) \in P$, 
$$R(x,t)\leqslant 4Q_0\,.$$
\end{lem}

We shall use repeatedly the following well-known consequences of curvature pinching:
\begin{prop}\label{prop:limite positive}
Let $(U_{k},g_{k}(\cdot),*_k)$ be a sequence of pointed evolving metrics defined on intervals $I_{k} \subset \RR_{+}$, and having curvature pinched toward positive. Let $(x_{k},t_{k})\in U_{k}\times I_{k}$ be a sequence such that $(1+t_{k})R(x_{k},t_{k})$ goes to $+\infty$. Then the sequence $\bar g_{k}:=R(x_{k},t_{k})g(t_{k})$ has the following properties:
\begin{enumerate}
\item The sequence $\Rmin(\bar g_{k})$ tends to $0$.
\item If $(U_{k},\bar g_{k},*_{k})$ converges in the pointed $\mathcal{C}^2$ sense, then the limit has nonnegative curvature operator. 
\end{enumerate}
\end{prop}

We also recall:
 \begin{lem}\label{lem:odd intersection}
Let $\epsi \in (0,10^{-1}]$. Let $(M,g)$ be a riemannian $3$-manifold,
$N \subset M$ be an $\epsi$-neck, and $S$ be a middle sphere of $N$. Let $[x y]$ be a geodesic segment such that $x,y\in M\setminus N$ and $[x y]\cap S\neq\emptyset$. Then the intersection number of $[x,y]$ with $S$ is odd. In particular, if $S$ is separating in $M$, then $x,y$ lie in different components of $M\setminus S$. 
\end{lem}

\begin{proof}\cite[Lemma 1.3.2.]{B3MP}
\end{proof}

We summarise the conclusion of Lemma~\ref{lem:odd intersection} by saying that $N$ is \bydef{traversed} by the segment $[x y]$.

\begin{corol}\label{cor:VC=cou}
Let $\epsi \in (0,10^{-1}]$. Let $(M,g)$ be a riemannian $3$-manifold,
$U \subset M$ be an $\epsi$-cap and $V$ be a core of $U$. Let
$x,y$ be points of $M\setminus U$ and $[x y]$ a geodesic segment connecting $x$ to $y$. Then $[xy]\cap V=\emptyset$. 
\end{corol}

As for Ricci flow with bubbling-off, we then have

\begin{theo}[Curvature-distance]\label{thm:courbure distance}
For all $A,C>0$ and all $\epsi\in (0,2\epsi_0]$, there exist $Q=Q(A,\epsi,C)>0$ and $\Lambda=\Lambda (A,\epsi,C)>0$ with the following property. Let $I \subset [0,+\infty)$ be an interval
and $\{(M(t),g(t))\}_{t\in I}$ be a surgical solution with curvature pinched toward positive. Let $(x_{0},t_{0})\in \calM$ be such that: 
\begin{enumerate} 
 \item $R(x_{0},t_{0}) \geq Q$;
\item For each point $y \in B(x_{0},t_{0},AR(x_{0},t_{0})^{-1/2})$, if $R(y,t) \geq 2R(x_{0},t)$, then $(y,t)$ has an
$(\epsi,C)$-canonical neighbourhood.
\end{enumerate}
Then for all $y  \in B(x_{0},t_{0},AR(x_{0},t_{0})^{-1/2})$, we have
 $$\frac{R(y,t_{0})}{R(x_{0},t_{0})} \leq  \Lambda.$$
\end{theo}

\begin{proof}
It suffices to redo the proof of \cite[Theorem 4.2.1]{B3MP}, with the following minor differences: 
\begin{itemize} 
\item In Step 1, to control the injectivity radius, one can use property iii) in the definition of $(\epsi,C)$-canonical neighbourhoods, as the canonical neighbourhood considered is not $\epsi$-round.
\item In Step 2, to prove that $[x'_{k}y'_{k}]$ is covered by strong $\epsi$-necks, one has to rule out closed canonical neighbourhoods. This is clear by the curvature ratio properties. 
Then use Corollary \ref{cor:VC=cou} instead of \cite[Lemma 1.3.2]{B3MP}
\end{itemize}
 \end{proof}

\subsection{Existence of cutoff parameters}

For the convenience of the reader, we restate Theorem~\ref{thm:echelle de coupure}.

\begin{theo}[Cutoff parameters]\label{thm:cutoff-bis}
For all $r,\delta>0$, there exist $h\in (0,\delta r)$ and $D>10$ such that
if $(M(\cdot),g(\cdot))$ is a complete surgical solution of bounded curvature defined on an interval $[a,b]$, with curvature pinched toward positive and satisfying $(CN)_r$, then the following holds:

Let  $t\in [a,b]$ and $x,y,z\in M(t)$  such that 
$R(x,t)\leq 2 /r^2$, $R(y,t)=h^{-2}$, and $R(z,t)\geq D/h^2$. Assume there is a curve  $\gamma$ connecting $x$ to $z$ and containing $y$, such that  each point of $\gamma$ with scalar curvature in $[2C_{0}r^{-2},{C_{0}}^{-1}Dh^{-2}]$ is centre of a 
$\epsi_{0}$-neck. Then $(y,t)$ is centre of a strong $\delta$-neck.
\end{theo}

\begin{proof}

The proof is almost the same as for \cite[Theorem 4.3.1]{B3MP}, arguing by contradiction. We only  need to adapt Step 1.

 \setcounter{etape}{0} Fix two constants
 $r>0, \delta>0$, sequences 
$h_k \to 0$, $D_k\to+\infty$, a sequence $(M_k(\cdot),g_k(\cdot))$ of surgical solutions satisfying the above hypotheses, and  sequences $t_k>0$, $x_k, y_k, z_k \in M$ such that $R(x_k, t_k) \leq 2r^{-2}$,  $R(z_k,t_k)\geq D_kh_k^{-2}$, and $R(y_k,t_k)=h_k^{-2}$. Let $\gamma_k$ be a curve from 
  $x_k$ to $z_k$ such that $y_k\in \gamma_k$, whose points of scalar curvature in $[2C_{0}r^{-2},{C_{0}}^{-1}Dh^{-2}]$ are centre of $\epsi_{0}$-neck. Finally assume that  $(y_k,t_k)$ is not centre of a strong
$\delta$-neck.

Consider the sequence  $(\bar M_k(\cdot),\bar g_k(\cdot))$ defined
by the following parabolic rescaling $$\bar g_k(t)=h_k^{-2}g_k(t_k+th_k^2)\,.$$
In order to clarify notation, we shall put a bar on points when they are involved in geometric quantities computed with respect of the metric $\bar g_{k}$. Thus for instance, we have $R(\bar y_k,0)=1$. The contradiction will come from extracting a converging subsequence
of the pointed sequence $(\bar M_k(\cdot),\bar g_k(\cdot),\bar y_k,0)$ and
showing that the limit is the cylindrical flow on $S^2\times\Rr$,
which implies that for $k$ large enough, $y_k$ is centre of
some strong $\delta$-neck, contradicting our hypothesis.

\begin{etape}
$(\bar M_k(0),\bar g_k(0),\bar y_k)$ subconverges in the pointed $\mathcal{C}^\infty$ sense to $(S^2\times \textbf{R},g_\infty)$ where $g_\infty$ is a product metric of nonnegative curvature operator and scalar curvature at most $2$.
\end{etape}

\begin{proof}
First we control the curvature on balls around
$\bar y_k$. Since $R(y_k,t_k)$ goes to $+\infty$, Theorem~\ref{thm:courbure distance} implies that for all $\rho>0$, there exists $\Lambda(\rho)>0$ and $k_0(\rho)>0$ such that $\bar g_k(0)$ has scalar curvature bounded above by $\Lambda(\rho )$ on $B(\overline y_k,\rho)$ for
$k\geq k_0(\rho)$. Moreover, by Assumption~(iii) of the definition of canonical neighbourhoods, $\bar g_{k}(\cdot)$ is ${C_{0}}^{-1}$-noncollapsed at  $(y_k,0)$. Indeed  $R(y_{k},t_{k})={h_{k}}^{-2} \in [2C_{0}r^{-2},{C_{0}}^{-1}Dh^{-2}]$, hence $y_{k}$ is centre 
of a $\epsi_{0}$-neck. Thus we can apply Gromov's compactness theorem to  extract a converging subsequence with regularity $C^{1,\alpha}$. 

Let us prove that for large $k$, the ball $B(\bar y_k,\rho )$ is covered by $\epsi_0$-necks. Recall that $g_k(t_k)$ satisfies
$$|\nabla R| < C_{0} R^{3/2},$$
at points covered by canonical neighbourhoods. 
Take a point $y$ such that $R(y,t_k)\leq  2C_{0}r^{-2}$ and integrate the previous inequality on the portion of $[y_ky]$ where $R \geq 2C_{0}r^{-2}$. An easy computation yields
\begin{equation}\label{distance}
d(\bar y,\bar y_k)\geq {1\over h_k}{2\over C_{0}}(\frac{r}{\sqrt{2C_{0}}}-h_k) \geqslant \rho,,
\end{equation}
for $k$ larger than some  $k_1(\rho) \geq k_0(\rho)$. It follows that  the scalar curvature of $g_k(t_k)$ is at least $2C_{0}r^{-2}$ on $B(\bar y_k,\rho )$
for every integer $k \geq k_1(\rho)$. It follows that $x_{k} \notin B(\bar y_k,\rho )$ and that $B(\bar y_k,\rho )$ is covered by $(\epsi_{0},C_{0})$-canonical neighbourhoods. 
On the other hand, for $k$ larger than some  $k_2(\rho)$, we have $R(\bar y,0) \leqslant \Lambda(\rho) <  {C_{0}}^{-1}Dh^{-2}$ for all $\bar y\in B(\bar y_k,\rho )$. 
It follows that  $\gamma \cap B(\bar y_k,\rho )$ is covered by $\epsi_{0}$-necks. As $z_{k} \notin B(\bar y_k,\rho )$, it follows that  $B(\bar y_k, \rho )$ is contained in the union $U_{\rho,k}$ of these necks: indeed, every segment coming from $\bar y_k$ and of length less than $\rho$ lies there.

Now by the $(CN)_{r}$ assumption, these  necks are strong $\epsi_0$-necks. The scalar curvature on $B(\bar y_k,\rho )$ being less than $\Lambda(\rho)$ for $k \geqslant k_{0}$, we deduce 
that on each strong neck, $\bar g_k(t)$ is smoothly defined on $[-\frac{1}{2\Lambda(\rho)},0]$, and has curvature bounded above by $2\Lambda(\rho)$. Hence for each $\rho >0$, the parabolic balls $P(\bar y_k,0,\rho,-\frac{1}{2\Lambda(\rho)})$ are
unscathed, with scalar curvature bounded above by  $2\Lambda(\rho)$ for all  $k \geq k_2(\rho)$. Since $g_k(\cdot)$ has curvature pinched toward positive, this implies a uniform control of the curvature operator there.

Hence we can apply the local compactness theorem \ref{thm:local compactness}. Up to extracting, $(\bar M_k(0),\bar g_k(0),\bar y_k)$ converges to some complete noncompact pointed riemannian $3$-manifold $(\overline M_\infty,\bar g_\infty,\bar y_\infty)$. By Proposition~\ref{prop:limite positive}, the limit has nonnegative curvature operator. 

Passing to the limit, we get a covering of $\overline M_\infty$ by $2\epsi_0$-necks. Then Proposition~\ref{prop:open tube} shows that $\overline M_\infty$ is diffeomorphic to $S^2\times\Rr$. In particular, it has two ends, so it contains a line, and Toponogov's theorem implies that it is the metric product of some (possibly nonround) metric on $S^2$ with $\Rr$.

As a consequence, the spherical factor of this product must be $2\epsi_0$-close to the round metric on $S^2$ with scalar curvature $1$. Hence the scalar curvature is bounded above by $2$ everywhere. This finishes the proof of Step~1.
\end{proof}

Henceforth we pass to a subsequence, so that $(\bar M_k(0),\bar g_k(0))$ satisfies the conclusion of Step~1.

The rest of the proof is the same as for  \cite[Theorem 4.3.1]{B3MP}. 
\end{proof}

\section{Proof of Proposition A}\label{sec:proofA}

\subsection{Piecing together necks and caps}

\begin{defi}
An $\epsi$-\bydef{tube} is an open subset $U\subset M$ which is equal to a union of $\epsi$-necks, and whose closure in $M$ is diffeomorphic to $S^2\times I$, $S^2\times \Rr$, or $S^2\times [0,+\infty)$.
\end{defi}

\begin{prop}\label{prop:caps tubes}
Let $\epsi\in (0,2\epsi_0]$. Let $(M,g)$ be a connected, orientable riemannian $3$-manifold. Let $X$ be a closed, connected subset of  $M$ such that every point of $X$ is the centre of an $\epsi$-neck or an $\epsi$-cap. Then there exists an open subset $U\subset M$ containing $X$ such that either
\begin{enumerate}
\item $U$ is equal to $M$ and diffeomorphic to $S^3$, $S^2\times S^1$, $RP^3$, $RP^3\# RP^3$, $\Rr^3$,  $S^2\times \Rr$, or a punctured $RP^3$, or
\item $U$ is a $10\epsi$-cap, or
\item $U$ is a $10\epsi$-tube.
\end{enumerate}
\end{prop}

\begin{proof}
First we deal with the case where $X$ is covered by $\epsi$-necks.

\begin{lem}\label{lem:onlynecks}
If every point of $X$ is centre of an $\epsi$-neck, then there exists an open set $U$ containing $X$ such that $U$ is a $10\epsi$-tube, or $U$ is equal to $M$ and diffeomorphic to $S^2\times S^1$ or $S^2\times \Rr$.
\end{lem}

\begin{proof}
Let $x_0$ be a point of $X$ and $N_0$ be a $10\epsi$-neck centred at $x_0$, contained in an $\epsi$-neck $U_0$, also centred at $x_0$.
If $X\subset N_0$ we are done. Otherwise, since $X$ is connected, we can pick a point $x_1\in X\cap N_0$ and a $10\epsi$-neck $N_1$ centred at $x_1$, with $x_1$ arbitrarily near the boundary of $N_0$. By Lemma~\ref{lem:glueneck}, an appropriate choice of $x_1$ ensures that $N_1\subset U_0$ and the middle spheres of $N_0$ and $N_1$ are isotopic. In particular, the closure of $N_0\cup N_1$ is diffeomorphic to $S^2\times I$, so $N_0\cup N_1$ is a $10\epsi$-tube.

If $X\subset N_0\cup N_1$ then we can stop. Otherwise, we pick a $10\epsi$-neck $N_2$ centred at some point $x_2$ near the boundary component of $N_1$ that does not lie in $N_0$, and iterate the construction as long as possible. Three cases may occur.

\paragraph{Case 1} The construction stops with some $10\epsi$-tube $N_0\cup\cdots\cup N_k$ containing $X$. Then we are done.

\paragraph{Case 2} The construction stops with some $10\epsi$-tube $N_0\cup\cdots\cup N_k$ such that adding another neck $N_{k+1}$ does \emph{not} produce a $10\epsi$-tube.

This can only happen if $N_{k+1}\cap N_0$ is non empty. In this case, by adjusting the centre $x_{k+1}$ of $N_{k+1}$, we can ensure that $N_0,\ldots,N_{k+1}$ cover $M$ and that the intersection of $N_{k+1}$ and $N_0$ is topologically standard. In this case, $M$ fibers over the circle with fiber $S^2$. Since $M$ is orientable, it follows that  $M$ is diffeomorphic to $S^2\times S^1$.

\paragraph{Case 3} The construction can be iterated forever.

In this case, the union $U$ of all $N_k$'s is a $10\epsi$-tube.

\begin{claim}
The frontier of $U$ is connected, equal to the boundary component of $\bar N_0$ which does not lie in $N_1$.
\end{claim}

We prove the claim by contradiction. If it is not true, then we can pick two points $x,y\in X$, each one being close to a distinct component of the frontier of $U$. Since $U\cap X$ is connected, we can find a path $\gamma$ connecting $x$ to $y$ in $U\cap X$. Now $\gamma$ is compact, so it can be covered by finitely many $10\epsi$-necks, each of which is contained in some $\epsi$-neck. We thus obtain a finite collection of $\epsi$-necks which cover $U$. Hence $U$ is relatively compact. This shows that the scalar curvature is bounded on $U$. Hence each $N_k$ has a definite size, and adding each $N_k$ to $N_0,\ldots,N_{k-1}$ adds definite volume. It follows that $\vol U$ is infinite, which is a contradiction. This proves the claim.

We continue the proof of Lemma~\ref{lem:onlynecks}.
 If $U$ contains $X$, then we are done. Otherwise, we pick a point $x_{-1}\in N_0\cap X$ close to the frontier, and choose a neck $N_{-1}$ centred at $x_{-1}$. We perform the same iterated construction as before. At each stage, we have a $10\epsi$-tube $N_{-k}\cup \cdots \cup N_{-1} \cup U$ whose frontier is connected. Hence the analogue of Case 2 above cannot occur. If the construction stops, then we have found a $10\epsi$-tube containing $X$. Otherwise the union $V$ of all $N_k$'s for $k\in\Zz$ is a $10\epsi$-tube. Repeating the argument used to prove the claim, we see that the frontier of $V$ is empty. Since $M$ is connected, it follows that $V=M\cong S^2\times \Rr$. 
\end{proof}

To complete the proof of Proposition~\ref{prop:caps tubes}, we need to deal with the case where there is a point $x_0\in X$ which is the centre of an $\epsi$-cap $C_0$. By definition of a cap, some collar neighbourhood of the boundary of $C_0$ is an $\epsi$-neck $U_0$. If $X\not\subset C_0$, pick a point $x_1$ close to the boundary of $C_0$. If $x_1$ is centre of a $10\epsi$-neck $N_1$, then we apply Lemma~\ref{lem:glueneck} again to find that $C_1:=C_0\cup N_1$ is a $10\epsi$-cap. Again we iterate this construction until one of the following things occur:

\paragraph{Case 1} The construction stops with a  $10\epsi$-cap containing $X$.

\paragraph{Case 2} The construction stops with a  $10\epsi$-cap $C_k=C_0\cup\cdots\cup N_k$ and a point $x_{k+1}$ near its frontier, such that $x_{k+1}$ is centre of  a  $10\epsi$-cap $C$ whose boundary is contained in $C_k$. Then $C_k\cup C$ equals $M$ and is diffeomorphic to $S^3$, $RP^3$, or $RP^3\# RP^3$.

\paragraph{Case 3} The construction goes on forever.
Then the same volume argument as in the proof of the above Claim shows that the union of all $C_k's$ is $M$. Thus $M$ itself is a $10\epsi$-cap, diffeomorphic to $\Rr^3$ or a punctured $RP^3$.
\end{proof}

Putting $X=M$, we obtain the following corollary:
\begin{theo}\label{thm:reconnait topo}
Let $\epsi\in (0,2\epsi_0]$. Let $(M,g)$ be a connected, orientable riemannian $3$-manifold. If every point of $M$ is the centre of an $\epsi$-neck or an $\epsi$-cap, then $M$ is diffeomorphic to $S^3$,  $S^2\times S^1$, $RP^3$, $RP^3\# RP^3$, $\Rr^3$,  $S^2\times \Rr$, or a punctured $RP^3$.
\end{theo}

Here is another consequence of Proposition~\ref{prop:caps tubes}:
\begin{corol}\label{cor:caps tubes}
Let $\epsi\in (0,2\epsi_0]$. Let $(M,g)$ be an orientable riemannian $3$-manifold. Let $X$ be a closed submanifold of $M$ such that every point of $X$ is the centre of an $\epsi$-neck or an $\epsi$-cap. Then one of the following conclusions holds:
\begin{enumerate}
\item $M$ is diffeomorphic to $S^3$, $S^2\times S^1$, $RP^3$, $RP^3\# RP^3$, $\Rr^3$,  $S^2\times \Rr$, or a punctured $RP^3$, or
\item There exists a locally finite collection $N_1,\ldots,N_p$ of $10\epsi$-caps and $10\epsi$-tubes with disjoint closures such that $X\subset \bigcup_i N_i$.
\end{enumerate}
\end{corol}

\begin{proof}
We apply Proposition~\ref{prop:caps tubes} to each connected component of $X$. If Case~(i) of the required conclusion does not hold, then we have found a locally finite collection of $10\epsi$-caps and $10\epsi$-tubes which cover $X$. By merging some of them if necessary, we can ensure that they have disjoint closures.
\end{proof}

Finally, we have a more precise result when there are just necks:

\begin{prop}\label{prop:open tube}
Let $\epsi\in (0,2\epsi_0]$. Let $(M,g)$ be an open riemannian $3$-manifold such that every point of $M$ is centre of an $\epsi$-neck. Then $M$ is diffeomorphic to $S^2\times\Rr$.
\end{prop}

This follows immediately from the proof of Lemma~\ref{lem:onlynecks}.

\subsection{Proof of Proposition A}

Recall the statement:
\begin{propA}
There exists a universal constant $\bar\delta_A>0$ having the following property: let $r,\delta$ be surgery parameters, $a,b$ be positive numbers with $a<b$, and $\{(M(t),g(t))\}_{t\in (a,b]}$ be an $(r,\delta)$-surgical solution.  Suppose that $\delta\le\bar\delta_A$, and $\Rmax(b)=\Theta$.

Then there exists a riemannian manifold $(M_+,g_+)$, which is obtained from $(M(\cdot),g(\cdot))$ by $(r,\delta)$-surgery at time $b$, and in addition satisfies:
\begin{enumerate}
\item $g_+$ has $\phi_b$-almost nonnegative curvature;
\item $\Rmin(g_+) \geq \Rmin(g(b))$.
\end{enumerate}
\end{propA}

Throughout this section we shall work in the riemannian manifold $(M(b),g(b))$. In particular all curvatures and distances are taken with respect to this metric.

Let $\mathcal G$ (resp.~$\mathcal O$, resp.~$\mathcal R$) be the set of points of $M(b)$ of scalar curvature less than $2 r^{-2}$, (resp.~$\in [2r^{-2}, \Theta /2)$, resp.~$\ge \Theta /2$.)

For brevity, we call \bydef{cutoff neck} a strong $\delta$-neck centred at some point of scalar curvature $h^{-2}$. Note that cutoff necks are contained in $\mathcal O$, and have diameter and volume bounded by functions of $h,\delta$-alone.

\begin{lem}\label{lem:propA}
There exists a locally finite collection $\{N_i\}$ of pairwise disjoint cutoff necks such that any connected component of $M(b)\setminus \bigcup_i N_i$ is contained in either $\mathcal G\cup \mathcal O$ or  $\mathcal R\cup \mathcal O$. 
\end{lem}

\begin{proof}
By Zorn's Lemma, there exists a maximal collection  $\{N_i\}$ of pairwise disjoint cutoff necks. Such a collection is automatically locally finite, e.g. because if $K$ is any compact subset, all cutoff necks that meet $K$ are contained in the $h(2\delta^{-1}+1)$-neighbourhood of $K$, which has finite volume.

Suppose that some component $X$ of $M(b)\setminus \bigcup_i N_i$ contains at least one point $x\in \mathcal G$ and one point $z\in \mathcal R$. Since $X$ is a closed subset of $M(b)$, there exists a geodesic path $\gamma$ in $X$ connecting $x$ to $z$.

\begin{claim}
The intersection of $\gamma$ and $\bord X$ is empty.
\end{claim}

\begin{proof}[Proof of the claim]
First observe that if $y$ is a point of $\bord X$, then $y$ has a canonical neighbourhood $U$. This neighbourhood cannot be a cap, because then $U$ would contain the whole of $X$, which would imply that $X\subset\mathcal O$. Hence $U$ is a neck.

If such a point $y$ belonged to $\gamma$, then by Lemma~\ref{lem:odd intersection} the neck $U$ would be traversed by $\gamma$.
This contradicts the fact that $\gamma\subset X$.
\end{proof}

In order to apply Theorem~\ref{thm:echelle de coupure}, one has to prove the 
following

\begin{claim} Each point of $\gamma$ with scalar curvature in $[2C_{0}r^{-2},{C_{0}}^{-1}Dh^{-2}]$ is centre of some $\epsi_{0}$-neck.
\end{claim}

\begin{proof}[Proof of the claim]
Let $y\in \gamma$ be such a point. By the curvature assumptions, $y$ is centre of a $(\epsi_{0},C_{0})$-canonical neighbourhood $U$, disjoint from $x$ and $z$. Hence 
$U$ cannot be a closed manifold. It remains to rule out the $(\epsi_{0},C_{0})$-cap case. We argue by contradiction. Assume that $U$ is an $(\epsi_{0},C_{0})$-cap. Then 
$U = N \cup C$, where $N$ is a $\epsi_{0}$-neck, $N\cap C = \emptyset$, $\overline N \cap C = \partial C$ and $y \in \Int {C}$. For simplicity dilate the metric by a factor such that the scalar curvature of $N$ is close to $1$.  Denote by $S$ the middle 
sphere of $N$. 
The curve $\gamma$ is clearly not minimizing in $U$. In particular if $x'$ (resp. $z'$) is an intersection  point of $\gamma$ with $S$ lying between $x$ and $y$  (resp. $y$ and $z$), 
then $d(x',z') \leq \diam(S) << 2{\epsi_{0}}^{-1} < d(x',y) + d(y,z')$. The geodesic segment $[x'z']\subset U$ is not contained in $X$, otherwise this would contradict the minimality of $\gamma$ in $X$. 
Hence there exists $p \in [x',z'] \cap \partial X$. By definition of $X$, the corresponding component of $\partial X$ is a boundary component of some neck  $N_{i}$.  Let us prove that 
$\gamma$ intersects $N_{i}$, which is a contradiction. Denote by ${S_{i}}^+$ the above boundary component of $N_{i}$, and note that $d({S_{i}}^+,S) < \diam(S)$. Let 
$N'$ be a $10\epsi_{0}$-subneck of $N_{i}$ which admits $ {S_{i}}^+$ as a boundary component, and $p' \in N'$ be its centre. Then $d(p',S) < \diam(S) + (10\epsi_{0})^{-1} < (4\epsi_{0})^{-1}$. 
It follows from Lemma \ref{lem:glueneck} that $S'$ is isotopic to $S$ in $N$. In particular $\gamma$ intersects $S'$.  
\end{proof}

 Let $y$ be a point of $\gamma$ of scalar curvature $h^{-2}$. Theorem~\ref{thm:echelle de coupure} yields a cutoff neck $N$ centred at $y$. Any $\delta$-neck meeting $N$ has to be traversed by $\gamma$, so $N$ is disjoint from the $N_i$'s. This contradicts maximality of $\{N_i\}$.
\end{proof}

Having established Lemma~\ref{lem:propA}, we prove Proposition~A. Let $\{N_i\}$ be a collection of cutoff necks given by that lemma.
Applying Theorem~\ref{thm:chirurgie metrique}, we obtain a Riemannian manifold $(M',g_+)$. By construction, the components of $M'$ fall into two types. Either they have curvature less than $\Theta/2$, or they are covered by canonical neighbourhoods. Applying Theorem~\ref{thm:reconnait topo}, we may safely throw away the components of the second type, obtaining the manifold $(M_+,g_+)$.
We remark that the operation cannot decrease $\Rmin$ (in fact $\Rmin(g_+)$ is equal to $\Rmin(g(b))$ unless $M_+$ is empty, in which case it is equal to $+\infty$). Thus the proof of Proposition~A is complete.

\section{Persistence}\label{sec:persistence}

\paragraph{Notation} 
If $(M(\cdot),g(\cdot))$ is a piecewise $\mathcal C^1$ evolving manifold defined on some interval $I\subset \RR$ and $[a,b] \subset I$, we call
\bydef{restriction of} $g$ to $[a,b]$ the evolving manifold  
\[ t \mapsto 
\left\{
\begin{array}{ccc}
  (M_+(a),g_+(a)) & \text{ if }  & t=a  \\
  (M(t),g(t)) & \text{ if }   & t \in (a,b]  
\end{array}
\right.
\]
We shall still denote by $g(\cdot)$ the restriction.  
Given $(x,t)\in \calM$, $r>0$ and $\Delta t>0$ we define the forward parabolic neighbourhood $P(x,t,r,\Delta t)$ as the set 
$$ P(x,t,r,\Delta t) = \{(x',t') \in \calM \mid x' \in B(x,t,r), t\leqslant t' \leqslant t + \Delta t\}$$
 When we consider a restriction of $g(\cdot)$ to some $[a,b]\subset I$, the parabolic neighbourhood $P(x,a,r,\Delta t)$ will be defined using the ball
$B(x,a,r)$ of radius $r$ with respect to the metric $g_+(a)$. \\
A parabolic neighbourhood $P(x,t,r,\Delta t)$ is said to be \bydef{unscathed} 
if $x' \in M_\reg(t')$ for all $x ' \in B(x,t,r)$ and $t' \in [t,t+\Delta t)$. Otherwise 
it is \bydef{scathed}. 

Given two surgical solutions $(M(\cdot),,g(\cdot))$ and $(M_0(\cdot),g_0(\cdot))$, we say that an unscathed parabolic neighbourhood $P(x,t,r,\Delta t)$ of $(M(\cdot),g(\cdot))$ is \bydef{$\epsi$-close}  to another unscathed parabolic neighbourhood $P(x_0,t,r_0,\Delta t)$ of $(M_0(\cdot),g_0(\cdot))$ if $(B(x,t,r),g(\cdot))$  is  $\epsi$-close to 
$(B(x_0,t,r_0),g_0(\cdot))$ on $[t,t+\Delta t]$ . 
We say that  $P(x,t,r,\Delta t)$ is $\epsi$-\bydef{homothetic} to  $P(x_0,t_0,r_0,\lambda \Delta t))$ if it is 
$\epsi$-close after a parabolic rescaling by $\lambda$.

The goal of this section is to prove the following technical theorem:

\begin{theo}[Persistence of almost standard caps]\label{thm:persistence-cap}
For all $A>0$, $\theta\in [0,1)$ and $\hat r>0$, there exists $\bar \delta=\bar \delta_{\rm per} (A, \theta, \hat r)$ with the following property. Let $(M(\cdot),g(\cdot))$ be a surgical solution defined on some interval $[a,b]$, which is a $(r,\delta)$-surgical solution on $[a,b)$, with 
$r \geqslant \hat r$ and $\delta \geqslant \bar \delta$.
Let $t_0 \in [a,b)$ be a singular time and consider the restriction of $(M(\cdot),g(\cdot))$ to $[t_0,b]$. Let $p\in (M(t_0),g(t_0))$ be the tip of some $\delta$-almost standard cap of scale $h$. Let $t_1\le \min(b,t_0+\theta h^2)$ be maximal such that $P(p,t_0,Ah,t_1-t_0)$ is unscathed. Then the following holds:
\begin{enumerate}
\item The parabolic neighbourhood $P(p,t_0,Ah,t_1-t_0)$ is $A^{-1}$-homothetic to $P(p_0,0,A,(t_1-t_0)h^{-2})$;
\item If  $t_1< \min(b,t_0+\theta h^2)$, then $B(p,t_0,Ah)\subset M_\sing(t_1)$ disappear at time $t_1$.
\end{enumerate}
\end{theo}

\begin{rem}
 In \cite{B3MP}, the conclusion in the last case was that $B(p,t_0,Ah) \subset M_\sing(t_1)$.\\
\end{rem}

Before giving the proof of the theorem, we summarize some technical results from~\cite{B3MP} Sections 6.1 and 6.2. 
 
Let $T_0$ be a positive real number. Let ${\mathcal M_0}=(X_0, g_0(\cdot), p_0)$ (the model) be a complete $3$-dimensional Ricci flow defined on $[0,T_0]$ such that
the quantity $$\Lambda(N) := \max_{X_{0}\times [0,T_0]} \big\{ |\nabla^p\Rm| \mid \,0\leq p\leq N\big\}
$$ is finite for all $N\in\textbf{N}$.

\begin{corol}[Persistence of the model in dimension $3$]\label{cor:per3}
Let $A>0$, there exists $\rho=\rho ({\mathcal M_0}, A)>A$ with the following property. Let $\{(M(t),g(t))\}_{t\in [0,T]}$ be a surgical solution  with $T\leq T_0$. Suppose that 
\begin{enumerate}
\item[a)] $(M(\cdot),g(\cdot))$ has curvature pinched toward positive. 
\item[b)] $|{\partial R \over \partial t}|\leq C_0 R^2$ at any $(x,t)$ with $R(x,t) \geqslant 1$. 
\end{enumerate} Let $p \in M(0)$ and $t \in (0,T]$ be such that
\begin{itemize}
\item [c)]$B(p,0,\rho )$ is $\rho^{-1}$-close to $B(p_0, 0, \rho )\subset X_0$,
\item [d)]$P(p, 0, \rho , t)$ is unscathed and $|\Rm| \leqslant \Lambda([A]+1)$ there. 
\end{itemize}
Then $P(p, 0, A, t)$ is $A^{-1}$-close to $P(p_0, 0, A, t)$.
\end{corol}

\begin{proof}
The proofs of \cite[Corollary 6.2.4 and 6.2.6]{B3MP} work for surgical solutions.\\
\end{proof}

\begin{proof}[Proof of Theorem~ \ref{thm:persistence-cap}]

We consider as model the standard solution ${\mathcal X_0}:=({\mathcal S_0},g_0(\cdot),p_0)$ restricted to $[0,\theta]$. 
Let us assume for simplicity that $T\geq t_{0}+\theta h^2$, so that $t_{1}=t_{0}+\theta h^2$. For any 
nonnegative integer $N$, recall that 
$$\Lambda(N)=\max_{\mathcal{S}_{0}\times [0, \theta ]} \big\{ |\nabla^p\Rm|, |R|;\,0\leq p\leq N\big\}.$$ 

In the sequel we consider the restriction of $(M(\cdot),g(\cdot))$ to $[t_0,b]$ and we define:
$$
\bar g(t) := h^{-2}g(t_{0}+th^2) \textrm{ for }t\in [0, \min\{\theta,(b-t_0)h^{-2}\}].
$$
Note that $\bar g(\cdot)$ satisfies Assumptions a) and b) of Corollary~\ref{cor:per3}. Indeed, it is readily checked that the curvature pinched toward positive property 
is preserved by the parabolic rescaling, since $t_0\geqslant 0$ and $h^{-2}\geqslant 1$. On the other hand, if $g(\cdot)$ satisfies 
$(CN)_r$ on $[0,b)$, it follows easily by a continuity argument that any $(x,b)$ with $R(x,b) \geqslant 2r(b)^{-2}$ satisfies 
the estimate $\left| \frac{\partial R}{\partial t} \right| \leqslant C_0R^2$ at $(x,b)$. After rescaling by $h(b)^{-2} >> 2r^{-2}$, this property 
holds at points with scalar curvature above $1$. 

Fix $A>0$ and set $\rho:=\rho({\mathcal M_0},A)$.
By the definition of an $\delta$-almost standard cap, the ball $B_{\bar g}(p,0,5+\delta^{-1})$ is $\delta '$-close to $B(p_{0}, 0, 5+\delta^{-1})\subset \mathcal S_{0}$. Let $T_{\max}\in [0, \theta ]$ be the maximal time such that $P_{\bar g}(p, 0, A, T_{\max})$ is unscathed. By closeness at time $0$ one has $|R_{\bar g}|\leq 2\Lambda_{0}$ on $B_{\bar g}(p,0,\delta^{-1})$.

Now for $t\in [0, \min ((4\Lambda C_0)^{-1}, \theta )]$ such that $P_{\bar g}(p, 0, \rho , t)$ is unscathed, we have $|R_{\bar g}|\leq 4\Lambda_{0}$ on  $P_{\bar g}(p, 0, \rho , t)$ by the time derivative estimate on the scalar curvature. Using the pinching assumptions, we deduce
$|\Rm_{\bar g}|\leq \Lambda'_{0}$ on the same neighbourhood. 

Set $T_{\times 2}:=\min (\theta, (4\Lambda_{0}C_0)^{-1}, (4\Lambda'_{0})^{-1})$. The above curvature bound gives, for $t\leq \min (\theta , T_{\times 2})$,
$$1/2\leq \bar g(t)/\bar g(0)\leq 2$$
on $B_{\bar g}(p,0,\delta^{-1})$. In particular, for all $x\in B_{\bar g}(p,0,\rho)$ and all $0<t\leq \min (\theta, T_{\times 2})$, $(x,t)$ is not centre of a $\delta$-neck because $d_{\bar g(t)}(x,p)\leq 2 d_{\bar g(0)}(x,p)\leq 2\rho $ and the length of a $\delta$-neck at time $t$ is larger than $1/2\delta^{-1}R(x,t)^{-1/2}\geq (4\delta \Lambda_{0})^{-1}>4\rho$, if $\delta$ is small enough.

This implies that $P_{\bar g}(p, 0, \rho , t)$ is unscathed if $t\leq \min (T_{\times 2}, T_{\max})$. Indeed, if not, there exists $t'<t$ such that $P_{\bar g}(p, 0, \rho , t')$ is unscathed but $B_{\bar g}(p, 0, \rho )\cap M_\sing(t')\ne \emptyset$. If $B_{\bar g}(p, 0, \rho )$ is not contained in $M_\sing(t')$, then it must intersect a surgery sphere of $\cals(t')$, which is the middle sphere of a $\delta$-neck centred at $(x,t')$. The above estimate rules out this possibility. Hence $B_{\bar g}(p, 0, A)\subset B_{\bar g}(p, 0, \rho )\subset M_\sing(t')$ for $t'<T_{\max}$. This is impossible by assumption. Remark that if $B(p,0,\rho) \cap M_\sing(t) \ne \emptyset$ the same arguments shows that $B(p,0,\rho)$ is contained in $M_\sing(t)$ and disappears at time $t$.

We can now apply Corollary~\ref{cor:per3}. We get that $P_{\bar g}(p, 0, \rho_{1} , t)$ is $\rho_{1}^{-1}$-close to $P(p_{0}, 0, \rho_{1} , t)$ for all $t\leq\min (T_{\times 2}, T_{\max})$. If $T_{\times 2}\geq T_{\max}$ we are done. Otherwise by closeness we have that $|\Rm_{\bar g}|$ and $|R_{\bar g}|$ are no greater than 
$2\Lambda_{0}$ on $P_{\bar g}(p, 0, \rho_{1} , T_{\times 2})$. Then $|R_{\bar g}|\leq 4\Lambda_{0}$ on $P_{\bar g}(p, 0, \rho_{1} , \min (2T_{\times 2}, T_{\max}))$ where $\rho_{1}=\rho({\mathcal M_0},\rho_{2})$ and $\rho_{2}>0$. We then iterate the above argument, which terminates in finitely many steps, as in \cite[Corollary 6.2.4]{B3MP}.
If $T_{\max} < \theta$, then Conclusion~(ii) follows from the previous remark. This finishes the proof of Theorem~\ref{thm:persistence-cap}.
\end{proof}

\section{Proposition B}\label{sec:proofB}

\setcounter{etape}{0}

Recall the statement of Proposition B:

\begin{propB}
For all $Q_0,\rho_0,\kappa>0$ there exist $r=r(Q_0,\rho_0,\kappa)< 10^{-3}$ and $\bar\delta_B=\bar\delta_B(Q_0,\rho_0,\kappa)>0$ with the following property: let $\delta\le\bar\delta_B$, $0\le T_A<b$ and $(M(\cdot),g(\cdot))$ be a surgical solution defined on $[T_A,b]$ such that $g(T_A)$ satisfies $| \Rm | \le Q_0$ and has injectivity radius at least $\rho_0$.

 Assume that $(M(\cdot),g(\cdot))$ satisfies Condition~$(NC)_{\kappa/16}$, has curvature pin\-ched toward positive, and that for each singular time $t_0$, $(M_+(t_0),g_+(t_0))$ is obtained from $(M(\cdot),g(\cdot))$ by $(r,\delta)$-surgery at time $t_0$.

Then $(M(\cdot),g(\cdot))$ satisfies Condition~$(CN)_r$. 
 \end{propB}

We recall a lemma from~\cite{B3MP}.
 
\begin{lem}[distance distorsion]\label{lem:distance-distorsion} Let $\{g(t)\}$ be a Ricci flow solution on a $n$-dimensional manifold $U$, defined 
for $t\in [t_{1},t_{2}]$. Suppose that $\mid \Rm \mid \leq \Lambda$ on $U \times [t_{1},t_{2}]$. Then 
$$  e^{-2(n-1)\Lambda(t_{2}-t_{1})} \leq \frac{g(t_{2})}{g(t_{1})} \leq   e^{2(n-1)\Lambda(t_{2}-t_{1})} $$
\end{lem}
\begin{proof} 
\cite[ Lemma 0.6.6.]{B3MP}.
\end{proof}

In order to prove Proposition~B, we argue by contradiction. Suppose that some fixed numbers $Q_0,\rho_0,\kappa >0$ have the property that for all $r \in (0,10^{-3})$ and $\bar\delta_B>0$ there exist counterexamples.
 Then we can consider sequences  $r_k\to 0$, $\delta_k\to 0$,   and a sequence of $(r_k,\delta_k, \kappa)$-surgical solutions $(M_k(\cdot),g_k(\cdot))$ on $[0,b)$ which satisfy Condition~$(NC)_{\kappa/16}$, have curvature pinched toward positive, and such that for each singular time $t_0$, $(M_{k,+}(t_0),g_{k,+}(t_0))$ is obtained from $(M(\cdot),g(\cdot))$ by $(r,\delta)$-surgery at time $t_0$, but
$(CN)_{r_k}$ fails for some $t_k$. This last assertion means that there exists $x_k\in M_k(t_k)$ such that 
$$Q_k:=R(x_k,t_k)\geq r_k^{-2}\,,$$
and yet $(x_k,t_k)$ does not have a $(\epsi_0, C_0)$-canonical neighbourhood.

By a standard point-picking argument (see \cite[Lemma 52.5]{Kle-Lot:topology}), we may choose the sequence of bad points $(x_k,t_k)$ and  $H_k\to +\infty$ such that for all $t\in [t_k-H_kQ_k^{-1},t_k]$ and $x\in M_k(t)$, if $R(x,t)\ge 2Q_k$ then $(x,t)$ has a $(\epsi_0, C_0)$-canonical neighbourhood.

Without loss of generality, we assume that
$$\delta_{k} \leq \bar \delta(k,1-\frac{1}{k},r_{k})$$ (the right-hand side being the parameter given by  the Persistence Theorem~\ref{thm:persistence-cap}). From now on, the proof follows the lines of 
\cite[Section 7.2]{B3MP}.

We need a preliminary lemma.

\begin{lem}[Parabolic balls of bounded curvature are unscathed]\label{lem:Pintact}
For all $K >0, \rho >0$ and $\tau >0$ there exists an integer $k_0=k_0(K,\rho ,\tau )$ such that  for all $k\geqslant k_0$, if $|\Rm |\leqslant K$ on $B_{\bar g_k}(x_k, 0, \rho )\times (-\tau , 0]$ then $P_{\bar g_k}(x_k, 0, \rho, -\tau )$ is unscathed.
\end{lem}
\begin{proof}
Arguing by contradiction, fix $k$ and assume that there exist $z_k\in B_{\bar g_k}(x_k, 0, \rho)$ and $s_k\in [-\tau , 0)$ such that 
$z_k \notin M_\reg(s_k)$. As $z_k$ exists after $s_k$, there is an added cap 
$V$ in $M_+(s_k)$ such that $z_k \in V$. We can take $s_k$ to be maximal satisfying this property, i.e.~the set $B_{\bar g_k}(x_k, 0, \rho )\times (s_k, 0]$ is unscathed. In the sequel, we consider the restriction of $\bar g_k(\cdot)$ to $[s_k,0]$, as explained in the 
beginning of Section \ref{sec:persistence}. We drop indices for  simplicity. 

By definition, there exists a marked $\delta$-almost standard cap $(U,V,p,y)$  such that $(U, R_{\bar g}(y , s)\bar g(y,s))$ is $\delta '$-close to $B(p_0, 5+\delta^{-1})\subset \mathcal S_0$. Set $\tilde g(y,s):=R_{\bar g}(y , s)\bar g(y,s))$.  We shall now show that $d_{\tilde g}(x, p)$ is bounded independently of $k$ (if $k$ is large enough.)

\vspace{0,5cm}
\input{propB_lem_Pintact.pstex_t} 
\vspace{0,5cm}

Since $B_{\bar g}(x, 0, \rho )\times (s, 0]$ is unscathed with $|\Rm |\leq K$, by the distance-distorsion Lemma \ref{lem:distance-distorsion} we have
$$e^{-4K\tau }\leq \bar g(0)/ \bar g(s)\leq e^{4K\tau }\,,$$
which implies $d_{\bar g(s)}(x, z)\leq e^{2K\tau }\rho$. Since $y \in U$,
$$R_{\bar g}(y,s)\leq 2 R_{\bar g}(z,s)\leq 12K\,.$$
Notice that since $y\in \partial V$, $R_{\bar g}(y,s)= R_{\bar g}(y,s)$.
We now get
$$d_{\tilde g}(x,y)\leq \sqrt{12K}e^{2K\tau}=:K'(K,\rho, \tau ).$$
Finally,
$$d_{\tilde g}(x,p)\leq K'+5\,.$$

Let $A>2(K'+5)$ and $\theta <1$. The
persistence theorem \ref{thm:persistence-cap} implies that the set $P_{\bar g}(p,s,AR_{\bar g}(y,s)^{-1/2}, \min(\theta R_{\bar g}(y,s)^{-1},|s|)$
is, after parabolic rescaling at time $s$, $A^{-1}$-close to 
$P(p_0, 0, A, \min (\theta, |s|\bar R_k(y,s)))\,.$
Indeed, the second alternative of the persistence theorem does not occur since in this case,  $x\in B_{\tilde g}(p,A)\subset M_\sing(t_+)$ for some $t_+\in (s, 0)$, 
and $x$ disappears at time $t_+$. \\

We now choose $A:=k$ and $\theta := 1-1/k$.
\begin{claim}
We have $s+\theta  R_{\bar g}(y,s)^{-1} > 0$ for large $k$.
\end{claim}
\begin{proof}
If $R_{\bar g}(y,s)<1/2\tau$ then 
$$s+ \theta R_{\bar g}(y,s)^{-1}\geqslant -\tau+1/2 R_{\bar g}(y,s)^{-1}>0\,,$$
since $\theta >1/2$.

Suppose that $R_{\bar g}(y,s)\geqslant 1/2\tau$. Seeking a contradiction, assume that $s_1:=s+\theta R_{\bar g}(y,s)^{-1}\leq 0$ and apply the persistence theorem up to this time. By Proposition \ref{prop:standard VC}, $\Rmin(t) \geqslant \constst (1-t)^{-1}$, hence we have
$$ R_{\bar g}(x,s_1)\geqslant {1\over 2}R_{\bar g}(y,s)\constst(1-\theta)^{-1}\geqslant \constst(4\tau (1-\theta ))^{-1} =\frac{k.\constst}{4\tau}\,.$$
On the other hand, $R_{\bar g}(x,s_1)\leqslant 6K$, which  gives a contradiction for sufficiently large $k$. This proves the claim.
\end{proof}

Denote by $\tilde g(\cdot)$ the parabolic rescaling of $\bar g$ by $R_{\bar g}(y,s)$ at time $s$: 
$$ \tilde g(t) = R_{\bar g}(y,s)g\left(s+tR_{\bar g}(y,s)^{-1}\right).$$
By the conclusion of the persistence theorem, there exists a diffeomorphism $\psi : B(p_0,0,A)\longrightarrow B_{\tilde g}(p,0,A)$ such that 
$\psi^* \tilde g(\cdot)$ is $A^{-1}$-close to $g_0(\cdot)$ on $B(p_0,0,A)\times [0,\min\{\theta,|s|R_{\bar g}(y,s)\}]$. By the above claim, 
the minimum is $|s|R_{\bar g}(y,s):=s'$. Set $x':=\psi^{-1}(x)$. Proposition~\ref{prop:standard VC} implies that for every $\epsilon>0$, there exists $C_{st}(\epsi)$ such that any point $(x',t)$ of the standard solution has an $(\epsi, C_{st}(\epsi))$-canonical neighbourhood unless $t<3/4$ and $x'\notin B(p_0, 0, \epsi^{-1})$. Let us choose $\epsi:=\epsi_0\beta /2<<\epsi$ and $C_{st}:=C_{st}(\epsi)$. There are again two possibilities.

\paragraph{Case 1} The point $(x',s')$ has an $(\epsi,C_{st})$-canonical neighbourhood
$U'\subset B(x', s', 2C_{st}R(x',s')^{-1/2})$, where $R(x',s')$ is the scalar curvature of the standard solution at $(x',s')$.
The $A^{-1}$-closeness between $g_0(s')$ and  $\psi^*\tilde g(s')$ gives
$$R(x',s') \simeq R_{\tilde g}(x,s')=R_{\bar g}(y,s)^{-1} R_{\bar g}(x, 0) =R_{\bar g}(y,s)^{-1}\,.$$
On the other hand,
$$U'\subset B(x',s', \frac{\rho}{2} R_{\bar g}(y,s)^{1/2})\subset \psi^{-1}(B_{\bar g}(x, 0, \rho ))$$
since 
$$\frac{\rho}{2} R_{\bar g}(y,s)^{1/2}\simeq \frac{\rho}{2} R(x',s')^{-1/2}>2C_{st}R(x',s')^{-1/2}\,.$$
Therefore $\psi (U')$ is a  $(2\epsi, 2C_{st})$-canonical neighbourhood for $(x,0)$, hence an $(\epsi_0, C_0)$-canonical neighbourhood for this point.

\paragraph{Case 2} The point $(x',s')$ has no $(\epsi, C_{st})$-canonical neighbourhood. Then $s'\leq 3/4$ and  $x'\notin B(p_0, 0, \epsilon_{st}^{-1})$. Hence we have 
$$d_{\tilde g}(x,s',p)\geqslant 9/10 \epsi_{st}^{-1}\geqslant 3/2 (\epsi\beta )^{-1}>(\epsi\beta )^{-1}+5.$$
We infer that $(x,-s')$ is centre of an $(\epsi_0\beta)$-neck, coming from the strong $\delta$-neck there at the singular time. 
 We now apply the neck strengthening lemma \ref{lem:neck strengthening} which asserts that $(x,0)$ is centre of a strong $\epsi$-neck. Indeed the closeness with the standard solution ensures that $P_{\bar g}(x,-s', (\epsi_0\beta )^{-1}, 0) \subset P_{\bar g}(x, -s', A, 0)$ is unscathed and has $|\Rm |\leqslant 2 K_{st}$. This proves Lemma~\ref{lem:Pintact}.
\end{proof}

Now we begin the proof of Proposition~B proper. We consider parabolic rescalings.

\begin{etape} 
The sequence $(\bar M_{k}(0), \bar g_k(0),\bar x_k)$ subconverges  to some complete pointed riemannian manifold $(M_\infty, g_\infty, x_\infty)$ 
of nonnegative curvature operator. 
\end{etape}

\begin{proof}
We have to show that the sequence satisfies the hypothesis of  the local compactness theorem for flows (Theorem~\ref{thm:local compactness}.)
By choice of the basepoint and curvature pinching, we can apply Theorem~\ref{thm:courbure distance}. Hence for every $\rho$, the scalar curvature of $\bar g_{k}(0)$ 
is bounded above on $B(\bar x_k,0,\rho)$  by some constant $\Lambda(\rho)$ if $k\ge k(\rho)$.

Next we wish to obtain similar bounds on parabolic balls  
$P(\bar x_k,0,\rho,-\tau(\rho))$ for some $\tau(\rho)>0$, and show that they are unscathed. Set  $C(\rho):=\Lambda(\rho )+2$. Let $k_{1}(\rho) := k_0(K(\rho ),\rho, (2C_0C(\rho )^{-1})$ be the parameter given by Lemma \ref{lem:Pintact}.

\begin{claim}
If $k\geq k_1(\rho )$, then $P_{\bar g_k}(\bar x_k, 0, \rho, -(2C_0C(\rho ))^{-1})$ is unscathed and satisfies  $|\Rm |\leqslant 2C(\rho )$. 
\end{claim}

\begin{proof}
 Choose $s=s_k \in [-(4C_0C(\rho ))^{-1}, 0]$ minimal such that $\bar B_k(x_k, 0, \rho )\times (s, 0]$ is unscathed. By the curvature-time lemma \ref{lem:courbure temps} we have $R\leqslant 2C(\rho )$ on this set, which implies $|\Rm |\leqslant 2C(\rho )$ by the Pinching Lemma \ref{lem:pinching}. By Lemma~\ref{lem:Pintact}, $P_{\bar g_k}(x_k, 0, \rho, s)$ is unscathed. By minimality of $s$ we then have $s=-(2C_0C(\rho ))^{-1}$.
\end{proof}

By hypothesis,  the solutions $g_{k}(\cdot)$ are $\kappa$-noncollapsed on scales less than $r_0$. Hence $\bar g_{k}(0)$ is
$\kappa$-noncollapsed on scales less than $\sqrt{Q_k}r_0$.
This, together with the curvature bound implies a positive lower bound for the injectivity radius at $(\bar x_{k},0)$. Hence Theorem \ref{thm:local compactness} applies to the sequence $(\bar M_{k},\bar g_{k}(\cdot),x_{k})$. It implies that the sequence subconverges to $(M_{\infty},g_\infty(\cdot),x_{\infty})$, where $M_{\infty}$ is a smooth manifold, $g_{\infty}(0)$ is complete and  $g_\infty(.)$ is defined  on $B(x_\infty, 0, \rho )\times (-(2C_0C(\rho ))^{-1}, 0]$ for each $\rho >0$.

Lastly, since the metrics $g_{k}(0)$ are normalised and the scaling factor $Q_k$ goes to $+\infty$, the limit metric $g_{\infty}(t)$ has nonnegative curvature operator by curvature pinching. This argument completes the proof of Step~1.
\end{proof} 

Since $R(x_\infty)=1$, $M_\infty$ is a complete, nonflat, nonnegatively curved riemannian $3$-manifold. By the Cheeger-Gromoll-Hamilton classification of such manifolds,  $M$ is diffeomorphic to $\Rr^3$, $S^1\times \Rr^2$, $S^1\times S^2$, a line bundle over a closed surface, or a spherical space form. In particular, if $M_\infty$ is noncompact, then every smoothly embedded $2$-sphere in $M_\infty$ is separating.

\begin{etape} 
The riemannian manifold $(M_\infty, g_\infty(0))$ has bounded curvature.
\end{etape}

\begin{proof}
Of course we may assume that $M_\infty$ is noncompact.
By passing to the limit, we see that every point $p_0 \in M_\infty$ of  scalar curvature at least $3$ is centre of a $(2\epsi_0,2C_0)$-cap or a (not necessarily strong) $2\epsi_0$-neck. In the sequel, we refer to this fact by saying that the limiting partial flow $g_{\infty}(\cdot)$ satisfies the \bydef{weak canonical neighbourhood property.}

Let $(p_k)$ be a sequence of points of $M_\infty$ such that $R_\infty (p_k,0)\longrightarrow +\infty$; in particular, $d_\infty (p_k,x_\infty )\longrightarrow +\infty$ as $k\to +\infty$. Consider segments $[x_\infty p_k]$ which, after passing to a subsequence, converge to a geodesic ray $c$ starting at $x_\infty$.

On $[x_\infty p_k]$ we pick a point $q_k$ such that $R_\infty (q_k )={R_\infty (p_k )\over 3C_0}$.
For sufficiently large $k$, the point $(q_{k},0)$ is centre of a weak $(2\epsi_0,2C_{0})$-canonical neighbourhood $U_{k}$. Now the curvatures on $U_{k}$ belong to 
$[(2C_{0})^{-1}R_{\infty}(q_{k},0),2C_{0}R_{\infty}(q_{k},0)]$.
As a consequence, for large $k$, $x_{\infty}$ and 
$p_{k}$ do not belong to $U_{k}$. By Corollary~\ref{cor:VC=cou} this neighbourhood is a $2\epsi_0$-neck.

\vspace{0,5cm}
\centerline{\input{convergence-segment.pstex_t}}
\vspace{0,5cm}

\begin{lem}\label{lem:traverse}
For $k$ large enough, $c$ traverses $U_k$.
\end{lem}

\begin{proof}
Assume it does not. Recall that $g_\infty(0)$ is nonnegatively curved. Consider a geodesic triangle $x_\infty q_k c(t)$ for $t\geq 10C_0$. Choose $k$ large enough so that $\measuredangle (p_k x_\infty c(t) )\leq \pi/100$. Let $S_-$ (resp. $S_+$ ) be the component of $\bord U_k$ which is closest to (resp. farthest from) $x_\infty$. By comparison with a Euclidean triangle, we see that $\liminf_{t\to \infty} \measuredangle (x_\infty q_k c(t)) \ge\pi$.

\vspace{0,5cm}
\centerline{\input{propB_dessin3.pstex_t}}
\vspace{0,5cm}

Fix $t$ large enough so that this angle is greater than $98/100 \pi$. Then $q_kc(t)$ intersects $S_+$. The loop $\gamma$ obtained by concatenating $x_\infty q_k$, $q_kc(t)$ and $c(t)x_\infty$ has odd intersection number with $S_+$. This implies that $S_+$ is nonseparating. This contradiction proves Lemma~\ref{lem:traverse}.
\end{proof}

We proceed with the proof of Step~2.
Pick $k_0$ large enough so that $U_k$ is traversed by $c$, and let $S_0$ be the middle sphere of $U_{k_0}$. Let $a_0,a'_0\in S_0$ be two points maximally distant from each other.  Call $a_k,a'_k$ the respective intersections of the segments $[a_0c(t)]$ and $[b_0c(t)]$ with the middle sphere $S_k$ of $U_k$. 

\vspace{0,5cm}
\centerline{\input{propB_dessin4.pstex_t}}
\vspace{0,5cm}

By comparison with Euclidian triangles, when $t$ is large enough, the distance between $a_k$ and $a'_k$ is greater than $1/2d_\infty (a_0, a'_0)$. Thus we have
$$\diam (S_k)\geq d_\infty (a_k, a'_k)\geq 1/2 d_\infty (a_0, a'_0)=1/2\diam (S_0)\,.$$
Now the diameter of $S_k$ is close to $\pi\sqrt{2} R_\infty(q_k,0)^{-1}$ and tends to $0$ by hypothesis, which gives a contradiction.
This completes the proof of Step~2.
\end{proof} 

Applying again Lemma~\ref{lem:Pintact} and the distance-curvature Lemma, there exists $\tau>0$ such that $(\overline M_k, \bar g_k(t), (\bar x_k,0))$ converges to some Ricci flow on $M_\infty\times [-\tau , 0]$. Define
$$\begin{array}{rl}
\tau_0:= & \sup\{\tau \geq 0, \exists K(\tau ), \forall\rho >0, \exists k(\rho, \tau ) \textrm{ s.t. } B(\bar x_k, 0, \rho )\times [-\tau, 0] \\
& \textrm{is unscathed and has curvature bounded by }K(\tau ) \textrm{ for }k\geq k(\rho, \tau )\}\,.
\end{array}$$
We already know that $\tau_0>0$. The compactness theorem~\ref{thm:hamilton compactness} enables us to construct a flow $g_\infty(\cdot)$ on $M_\infty \times (-\tau_0, 0]$, which is a pointed limit of the flows $(\bar M_k, \bar g_k(\cdot))$. Moreover, passing to the limit, we see that the scalar curvature of $g_\infty$ is bounded by $K(\tau)$ on $(-\tau , 0]$ for all $\tau \in [0,\tau_{0})$.

\begin{etape} 
There exists $Q>0$ such that the curvature of $g_\infty (t)$ is bounded above by $Q$ for all $t\in (-\tau_0, 0]$.
\end{etape}

\begin{proof}
We know that $g_\infty (t)$ is nonnegatively curved and has the above-mentioned `weak canonical neighbourhood property'.
We show that the conclusion of the curvature-distance theorem holds on  $M_{\infty}$, at points of scalar curvature $>1$. For this, we let $p \in M_{\infty}$ and  
$t\in (-\tau_{0},0]$ be such that $R_{\infty}(p,t) >1$. Then there exists a sequence
$(\bar p_{k},t_{k})$, where $\bar p_{k} \in \overline M_{k}$, converging 
to $(p,t)$ and such that $R(\bar p_{k},t_{k}) \geq 1$ for $k$ large enough. As a consequence, they satisfy the hypotheses of the curvature-distance theorem as explained in the proof of Step~1. Passing to the limit, we deduce that for every $A>0$, there exists $\Lambda(A)>0$ such that for all $q \in M_{\infty}$, 
\begin{equation}
\frac{R_\infty(q,t)}{R_\infty(p,t)} \leq \Lambda\left(d_\infty(p,q,t)R_\infty(p,t)^{-1}\right).
\label{eq:courbure distance}
\end{equation} 

Let us estimate the variation of curvatures and distances on $M_\infty\times (-\tau_0, 0]$. We recall Hamilton's Harnack inequality for the scalar curvature (cf.~\cite[Appendix F]{Kle-Lot:topology})
$${\partial R_\infty\over\partial t}+{R_\infty\over t+\tau_0}\geq 0\,,$$
which implies
$$R_\infty(., t)\leq K(0){\tau_0\over t+\tau_0}\,.$$

Ricci curvature, which is positive, satisfies a similar estimate, which implies
$$const({\tau_0\over t+\tau_0})g_\infty\leq {\partial g_\infty\over \partial t}\leq 0\,,$$
thus
$$const	 \sqrt{{\tau_0\over t+\tau_0}}\leq {\partial d_\infty\over \partial t}(x,y,t)\leq 0\,.$$
By integration we obtain
$$|d_\infty (x,y,t)-d_\infty (x,y,0)|\leq C\,.$$

Since $M_\infty$ is nonnegatively curved, there exists $D>0$ such that for all $y\in M_{\infty}$, if  $d_\infty (x_\infty, y, 0)>D$, then there exists $z\in M_{\infty}$ such that 
\begin{equation}
 d_\infty (y,z,0)=d_\infty (x_\infty, y, 0)\textrm{ and }d_\infty (x_\infty,z,0)\geq 1.99 d_\infty (x, y, 0) \, \label{eq:presque alignes}
 \end{equation}
i.e.~the points $x_\infty,y,z$ almost lie on a line.
(Note that this is true even if $M_\infty$ is compact, because then it is vacuous!)

\vspace{0,5cm}
\centerline{\input{propB_dessin5.pstex_t}}
\vspace{0,5cm}

By comparison with Euclidean space, if $D$ is large enough we have $\pi-1/100\leq \measuredangle (x_{\infty}yz)\leq \pi$ for any such $y$ and $z$. Observe that since $|d_\infty (.,.,t)-d_\infty (.,.,0)|<C$ uniformly in $t$, we can choose $D>>C$ large enough so that, for all $t\in (-\tau_0, 0]$, we have
$$|d_\infty (y,z,t)-d_\infty (x_\infty, y, t)|<2C\textrm{  and  }d_\infty (x_{\infty},z,t)\geq 1.98 d_\infty (x_{\infty},y,t)\,,$$
thus $\measuredangle_t (x_{\infty}yz)\geq \pi-1/50$.

Let us show that the scalar curvature of $g_{\infty}(t)$ is uniformly bounded above on $M_{\infty} \backslash B_{\infty}(x_{\infty},0,2D)$. We argue by contradiction. 
Suppose that there exists $(y_i,t_i)$ such that $d_\infty(x_\infty,y_i,0) > 2D$ and such that $R_\infty (y_i,t_i)\underset{i\to +\infty}{\longrightarrow}+\infty$.  Each $(y_i,t_i)$ has a weak $(3\epsi_0,C_{0})$-canonical neighbourhood. If this neighbourhood is a $(3\epsi_0,C_{0})$-cap, then we let $y'_{i}$ be the centre of its $3\epsi_0$-neck. Since the diameter for $g(t_{i})$ of the 
cap is small (less than $4C_{0}R_{\infty}(y_{i},t_{i})^{-1} < C$ for large $i$), we still have (for large $i$):
\begin{eqnarray*}
d_\infty(x_\infty,y'_i,0) & \geq &
d_\infty(x_\infty,y'_i,t_{i}) - C \\
& \geq & d_\infty(x_\infty,y_i,t_{i}) -2C \\
& \geq & d_\infty(x_\infty,y_i,0)-3C >D.
\end{eqnarray*}

Furthermore, the curvature of $R_{\infty}(y'_{i},t_{i}) \geq \frac{1}{3C_{0}}R_{\infty}(y_{i},t_{i})$ tends to $+\infty$. 
Up to replacing $y_{i}$ by $y'_{i}$,  we may assume that there exists a sequence $(y_i,t_i)$ such that $d_\infty(x_\infty,y_i,0) > D$,  $R_\infty (y_i,t_i)\underset{i\to +\infty}{\longrightarrow}+\infty$ and 
$(y_{i},t_{i})$ is centre of a $3\epsi_0$-neck $U_{i}$. For each  
$i \in \NN$, pick $z_{i} \in M_{\infty}$ satisfying (\ref{eq:presque alignes}). By the above remark, 
$\measuredangle_{t_i}(x_\infty y_iz_i)\geq \pi-\frac{1}{50}$.

The points  $x_\infty$ and $z_i$ being outside $U_{i}$, we deduce that $x_{\infty}y_{i}$ and $y_{i}z_{i}$ each intersect some component $\bord U_{i}$. Let $S_{i}$ be the middle sphere of $U_i$. This sphere separates $x_\infty$ from $z_i$ in the sense that any curve connecting $x_\infty$ to $z_i$ passes through $S_i$. Indeed, otherwise the loop obtained by concatenating $[x_{\infty}y_{i}]$, $[y_{i}z_{i}]$ and $[z_{i}x_{\infty}]$ would have odd intersection number with $S_{i}$. As before this leads to a contradiction.

Now $\diam (S_i,t_i)\longrightarrow 0$ since $R_\infty (y_i,t_i)\longrightarrow +\infty$ as $i\to+\infty$. Since $g_\infty (t)$ is nonnegatively curved, distances are nonincreasing in $t$.
As a consequence, 
$$\diam (S_i, 0)\underset{i\to+\infty}{\longrightarrow 0}\,.$$
At time $0$ the curvature bounds and the hypothesis of $\kappa$-noncollapsing imply a uniform lower bound on the injectivity radius. Then for large $i$, the diameter of $S_i$ is less than the injectivity radius of $g_\infty (0)$. This implies that $S_i$ bounds a $3$-ball which contains neither $x_\infty$ nor $z_i$. Hence we can connect $x_\infty$ to $z_i$ by an arc which avoids $S_i$.

This implies that  the curvature is bounded outside the $g_\infty(0)$-ball of radius $2D$ around $x_\infty$. We deduce a uniform curvature bound on $(-\tau_0,0]$ in the ball using Equation (\ref{eq:courbure distance}). 
\end{proof}

\begin{etape} 
We have $\tau_0=+\infty$.
\end{etape}

\begin{proof}
Consider a subsequence $(\overline M_k \times (-\tau_0,0],\bar g_k(t),(x_k,0))$ that converges to $(M_\infty \times (-\tau_{0},0],g_\infty(t),(x_\infty,0))$. Since the limit has scalar curvature bounded above by $Q$ we deduce that (up to replacing 
$Q$ by $Q+1$), for all $0< \tau < \tau_0$, for all $\rho>0$, there exists 
$k'(\tau,\rho)\in \NN$ such that for all $k\geq k'(\tau,\rho)$, the parabolic neighbourhood $P_{\bar g_k}(x_k,0,\rho,-\tau)$ is unscathed and has scalar curvature $\leq Q$. 

Suppose that $\tau_0 < +\infty$ and let $0<\sigma<(4C_0(Q+2))^{-1}$. Then up to extracting a subsequence, for every $K>0$, there exists $\rho=\rho(\sigma,K)$ such that $P_{k}:=P_{\bar g_k}(x_k,0,\rho,-(\tau_0+\sigma))$ is scathed or does not have curvature bounded above by $K$. 

Set $K:=2(Q+2)$ and $\rho:=\rho(\sigma,K)$.
If  $k \geq k'(\tau_0-\sigma,\rho)$, then
 $P_{k}$ is scathed. Indeed, if $k \geq k'(\tau_0-\sigma,\rho)$, we have $R \leq Q$ on $P_{\bar g_k}(x_k,0,\rho,-\tau_0+\sigma))$. If 
$P_{k}$ is unscathed, the curvature-time lemma (\ref{lem:courbure temps}) applied between 
$-\tau_0+\sigma$ and $-\tau_0-\sigma$ (since $2\sigma \leq (2C_0(Q+2))^{-1}$) implies that $R \leq 2(Q+2)$ on $P_{k}$, which excludes the second alternative.

Thus there exists $ x'_k \in \bar B_k(x_k, 0, \rho )$, and $t_k \in [-\tau_0-\sigma,-\tau_0+\sigma]$, assumed to be maximal, such that $\bar g_k(t_k)\ne \bar {g_k}_+(t_k )$ at $x'_k$. 
Since $B(x_k, 0, \rho) \times (t_k,0]$ is unscathed, the above argument shows that $R \leq 2(Q+2)$ on this set, for all sufficiently large $k$. This implies an upper bound on the Riemann tensor on this set and hence by Lemma \ref{lem:Pintact} the parabolic neighbourhood $P_{\bar g_k}(x_k, 0, \rho, t_k)$ is unscathed. This contradicts the definition of $t_k$.
\end{proof}

We can now finish the proof of Proposition~B: since 
$\tau_0=+\infty$, the flow $(M_\infty, g_\infty(\cdot))$ is 
defined on $(-\infty,0]$, and has bounded, nonnegative curvature operator. Moreover, the rescaled evolving metric
 $\bar g_{k}(\cdot)$ is
$\kappa$-noncollapsed on scales less than $\sqrt{R(x_{k},b_{k})}r_0$,
so passing to the limit we see that $g_\infty(\cdot)$ is $\kappa$-noncollapsed on all scales.  The metric $g_\infty(\cdot)$ is not flat since it has scalar curvature $1$ at the point $(x_\infty,0)$. 

This shows that $(M_\infty, g_\infty(\cdot))$ is a $\kappa$-solution. By Theorem~\ref{thm:kappa VC} and the choice of constants in Subsection~\ref{sub:constantes}, every point of $M_\infty$ has an $(\frac{\epsi_0}{2},\frac{C_0}{2})$-canonical neighbourhood. Hence for sufficiently large $k$,  $(\bar x_k,0)$ has an $(\epsi_0,C_0)$-canonical neighbourhood. This contractiction finishes the proof of Proposition~B.

\section{Proof of Proposition C}\label{sec:proofC}

We recall the statement:

\begin{propC}
For all $Q_0,\rho_0>0$ and all $0\le T_A< T_\Omega$ there exists $\kappa=\kappa(Q_0,\rho_0,T_A,T_\Omega)$ such that for all $0<r<10^{-3}$ there exists $\bar\delta_C=\bar\delta_C(Q_0,\rho_0,T_A,T_\Omega,r)>0$ such that the following holds.

Let $0<\delta\leqslant\bar\delta_C$ and $b\in (T_A,T_\Omega]$, and $(M(\cdot),g(\cdot))$ be a $(r,\delta)$-surgical solution defined on $[T_A,b)$ such that 
$g(T_A)$ satisfies $| \Rm | \leqslant Q_0$, has injectivity radius at least $\rho_0$,  $\phi_A$-almost nonnegative curvature and satisfies $\Rmin(g_0)\ge -6/(4T_A+1)$. 
Then $g(\cdot)$ satisfies $(NC)_\kappa$.
\end{propC}

Note that by Cheeger's theorem and standard estimates on Ricci flow, there exists a constant $\kappa_\mathrm{norm}$ depending only on the normalisation of the initial condition, i.e.~$Q_0,\rho_0$, such that $(M(\cdot),g(\cdot))$ satisfies $(NC)_{\kappa_\mathrm{norm}}$ on $[T_A,T_A+2^{-4}Q_0^{-1}]$.

We set $\kappa_0 := \min(\kappa_\mathrm{norm},\kappa_\mathrm{sol}/2,\kappa_\mathrm{st}/2)$.

\subsection{Preliminaries}
Let $v_k(\rho )$ denote the volume of a ball of radius $\rho$ in the model space of constant sectional curvature $k$ and dimension $n$.

Let $\kappa>0$. One says that a Riemannian ball $B(x,\rho)$ is $\kappa$-noncollapsed if $\vert \Rm \vert \leqslant \rho^{-2}$ on $B(x,\rho)$ and if $\vol B(x,\rho) \geqslant 
\kappa \rho^{3}$. Similarly, a parabolic ball $P(x,t,\rho,-\rho^{-2})$ is $\kappa$-noncollapsed if $\vert \Rm \vert \leqslant \rho^{-2}$ on $P(x,t,\rho,-\rho^{-2})$ and  $\vol B(x,t,\rho) \geqslant 
\kappa \rho^{3}$.

We recall the following elementary lemma from \cite[Section 8.1]{B3MP}.

\begin{lem}\label{lem-preliminaire} \label{pas-intact}
 \begin{enumerate}
  \item If $B(x,\rho)$  is $\kappa$-noncollapsed, then for every $\rho' \in (0,\rho)$, $B(x,\rho')$ is $C\kappa$-noncollapsed, where $C:={v_0(1)\over v_{-1}(1)}$.
 The same property holds for $P(x,t,\rho',-\rho'^2)\subset P(x,t,\rho,-\rho^2)$.
\item Let $r,\delta$ be surgery parameters and $g(\cdot)$ be an $(r,\delta)$-surgical solution. 
Assume that $P_0=P(x_0,t_0,\rho_0,-\rho_0^2)$ is a scathed parabolic neighbourhood such that 
$|\Rm| \leqslant \rho_0^{-2}$ on $P_0$. Then $P_0$ is $e^{-12}\kappa_{st}/2$-noncollapsed.
\end{enumerate}
\end{lem}

\begin{rem}
\begin{itemize}
 \item From (i), we deduce that in order to establish noncollapsing at some point  $(x,t)$ on all scales $\leq 1$, it suffices to do it on the maximal scale $\rho \leq 1$ such that
$|\Rm| \leq \rho^{-2} $ on $P(x,t,\rho,\rho^{-2})$. This observation will be useful later.
\item If some metric ball $B(y,\rho)$ is contained in a $(\epsi,C_0)$-canonical neighbourhood \emph{ which is not $\epsi_{0}$-round} and 
satisfies $|\Rm| \leqslant \rho^{-2}$, then $B(y,\rho)$ is $C_0^{-1}$-noncollapsed on the scale $\rho$ by~\eqref{eq:vol}.
\end{itemize}
\end{rem}

\subsection{The proof}

We turn to the proof of Proposition C. 

Let $M_{\textrm{reg}}$ be the set of regular points in spacetime. This is an open, arcwise connected $4$-manifold. Likewise we let $M_{\textrm{sing}}$ be the set of singular points in spacetime.
Let $\gamma : [t_{0},t_{1}]\to \bigcup_t M(t)$ be a map such that $\gamma(t)\in M(t)$ for every $t$. Let $\bar t \in [t_0,t_{1}]$. Here we adopt the convention that 
$M_+(t)=M(t)$ if $t$ is regular.

\begin{defi}
One says that $\gamma$ is \bydef{continuous at} $\bar t$ if there is $\sigma>0$ such that\\
 1) $t \to \gamma(t) \in M(\bar t)$ on $[\bar t-\sigma,\bar t)$ and is left continuous at $\bar t$\\
 2) $t \to \gamma(t)\in M_+(\bar t)$ on $]\bar t,\bar t+\sigma]$ and has a right limit at $\bar t$ 
denoted $\gamma_+(\bar t)$\\
 3)Assume $\bar t<t_1$. If $\gamma(\bar t) \in M_\reg(\bar t)$, then $\gamma(\bar t)=\gamma_+(\bar t)$ under the identification of $M_\reg(\bar t)$ and 
$M(\bar t) \cap M_+(\bar t)$; if $\gamma(\bar t) \in S \subset \cals$, then $\gamma(\bar t)=\gamma_+(\bar t)$ under the identification of $S$ and 
the corresponding component of $\partial M\cap M_+(\bar t)$. 
\end{defi}
In particular, if $\gamma(\bar t) \in M_\sing(\bar t) \setminus \cals$ for $\bar t<t_1$, it is not continuous at $\bar t$. Indeed, $\gamma(\bar t)$ disappears at time $\bar t$.

We say that $\gamma$ is \bydef{unscathed} if $\gamma(t) \in M_\reg(t) $ for all 
$t \in [t_0,t_1)$. Otherwise $\gamma$ is \bydef{scathed.} 
  
 We adapt the arguments from the smooth case, replacing Perelman's reduced volume $\widetilde V$ (\cite[Section 7]{Per1}) by: 

$$\widetilde V_{\textrm{reg}}(\tau) := \int_{Y(\tau)} \tau^{-3/2} e^{-\ell(\mathcal{L}\exp(v),\tau)} J(v,\tau) dv,$$
where $\tau = t_{0}-t$ and
$$Y(\tau) := \{v \in T_{x_0}M \mid \mathcal{L}\exp(v):[0,\tau] \rightarrow M \text{ is minimal and unscathed }\}\,.$$

Let $(x_0,t_0)$ be a point. By Lemma~\ref{lem-preliminaire} and the remark following this lemma, we restrict attention to the scale $\rho_0\leq 1$ which is maximal such that $|\Rm|\leq \rho_0^{-2}$ on $P_0:=P(x_0,t_0,\rho_0,-\rho_0^2)$ and assume that $P_0$ is unscathed.

As before we set $B_0:=B(x_0,t_0,\rho_0)$.

\subsection{The case $\rho_0 \geqslant \frac{r}{100}$ }\label{subsec:grand rayon}

\begin{lem}\label{technique1}
Let $\hat r,\Delta,\Lambda$ be positive numbers. Then there exists $\bar \delta=\bar \delta (\hat r, \Delta, \Lambda )>0$ with the following property. Let $(M(\cdot),g(\cdot))$ be an $(r,\delta)$-surgical solution  on an interval $I=[a, a+\Delta ]$ with $\delta \leqslant \bar\delta$  and $r\geqslant \hat r$ on $I$. Let $(x_0,t_0)\in M\times I$ and $\rho_{0}\geqslant \hat r$ be such that $P_0:=P(x_0,t_0,\rho_0,-\rho_0^2)\subset M\times I$ is unscathed and $|\Rm |\leq \rho_0^{-2}$ on $P_0$.

Let $\gamma$ be a continuous spacetime curve defined on $[t_1,t_0]$ with $t_1\in [0,t_0]$ and such that $\gamma (t_0)=x_0$ and $\gamma$ is scathed. Then ${\mathcal L}_{t_0-t_1}(\gamma )\geq \Lambda$.
\end{lem}

Here ${\mathcal L}_{t_0-t_1}$ denotes the ${\mathcal L}$-length based at $(x_{0},t_{0})$, that is 
$$ {\mathcal L}_{t_0-t_1}(\gamma ) =\int_{t_1}^{t_0} \sqrt{t_0-t} \left(R(\gamma(t),t) + |\dot\gamma(t) |_{g(t)}^2\right)\ dt.$$
\begin{proof} To prove the lemma, it suffices to obtain one of the two inequalities:

\begin{equation}\label{ineq1}
\int_{t_1}^{t_0}\sqrt{t_0-t}\ R(\gamma(t),t)\ dt\geq \Lambda\,,
\end{equation}
\begin{equation}\label{ineq2}
\int_{t_1}^{t_0}\sqrt{t_0-t_1}|\dot\gamma(t) |_{g(t)}^2\ dt\geq \Lambda+4\Delta^{3/2}=:\Lambda'\,.
\end{equation} 

Indeed ${\mathcal L}_{t_0-t_1}(\gamma ) \geqslant \int_{t_1}^{t_0}\sqrt{t_0-t}\ R(\gamma(t),t)\ dt$, hence  \eqref{ineq1} implies the lemma.  For  \eqref{ineq2}, this comes from the fact that $R\geq -6$, hence $$\int_{t_1}^{t_0}\sqrt{t_0-t}\ R\ dt\geq -4[\tau^{3/2}]^{t_{0}-t_{1}}_0\geq -4\Delta^{3/2}\,.$$

Intuitively, those two conditions mean that a curve has large $\mathcal L$-length if it has large energy (which is the case if it moves very fast or goes a very long way), or if it stays long enough in an area of large scalar curvature.

Since $\gamma$ is scathed, it cannot remain in $P_0$. We shall make a first dichotomy according to whether $\gamma$ goes out very fast or not.

Set $$\alpha:=\min \left\{ \left({\hat r\over 4\Lambda'}\right)^2, 10^{-2}\right\} \in (0,10^{-2})\,.$$

\textbf{Case 1:} There exists $t'\in[t_0-\alpha \rho_0^2, t_0]$ such that $\gamma (t')\not\in B_0$. 

\vspace{0,5cm}
\centerline{\input collapse4.pstex_t}
\vspace{0,5cm}

Choose $t'$ maximal with this property. We then have
$$\int_0^{t_0-t'}|\dot\gamma |\ d\tau\leq\Big (\int \sqrt{\tau}|\dot\gamma|^2 \Big )^{1/2}\Big ( \int{1\over\sqrt{\tau}}  \Big )^{1/2}\,,$$
so
$$\int_0^{\tau_1}\sqrt{\tau} |\dot\gamma |^2\ d\tau\geq \int_0^{t_0-t'}\sqrt{\tau}|\dot\gamma |^2\ d\tau
\geq \Big (\int_0^{t_0-t'}|\dot\gamma |\ d\tau \Big )^2
\Big (\int_0^{t_0-t'}{1\over \sqrt{\tau}}\ d\tau\Big )^{-1}\,.$$
On $(t',t_0]$, we have $\gamma\subset P_0$. Since
$P_0$ is unscathed, we have 
$$g(t_0)e^{-4\rho_0^{-2}(t_0-t')}\leq g(t)\leq g(t_0)e^{4\rho_0^{-2}(t_0-t')}\,,$$
hence
$${1\over 2}g(t_0)\leq e^{-4\alpha }g(t_0)\leq g(t)\leq  e^{4\alpha }g(t_0)\leq 2 g(t_0)\,.$$

Since $\gamma (t' )\not\in B_0$,
$$\int_0^{t_0-t'}|\dot\gamma|_{g(t_0-\tau )}\geq {1\over \sqrt{2}}\int_0^{t_0-t'}|\dot\gamma |_{g(t_0)}\geq {\rho_0\over \sqrt{2}}\,,$$
so
$$\int_0^{t_0-t'}|\dot\gamma |^2\sqrt{\tau}d\tau \geq {\rho_0^2\over 2}\big ( [2\sqrt{\tau}]_0^{t_0-t'}  \big )^{-1}\geq {\rho_0\over 4\sqrt{\alpha}}\geq {\hat r\over 4\sqrt{\alpha}}\,.$$
By choice of $\alpha$, this last quantity is bounded below by $\Lambda'$. This shows that $\gamma$ satisfies Inequality \eqref{ineq2}.

\begin{rem}
In this case, there is no constraint on $\delta$. 
\end{rem}

\textbf{Case 2:} For all $t\in [t_0-\alpha \rho_0^2, t_0]$, $\gamma (t)\in B_0$.

Since $\gamma$ is scathed, there exists $(\bar x, \bar t)$ such that $\gamma (\bar t) \notin M_\reg(\bar t)$. Since $P_0$ is unscathed, we have $\bar t < 
t_0-\alpha \rho_0^2$. Since $\gamma$ is continuous and defined after 
$\bar t$, we have $\gamma(\bar t) \in \partial M_\sing(t) \subset \cals(t)$. Assume that $\bar t$ is maximal for this property. We have $R(\bar x, \bar t) \approx h^{-2}$, where, for the sake of simplicity we set $h:=h(\bar t)$. We may choose $\bar \delta$ small enough (depending on $\hat r$) so as to make $h$ so small that $R(\bar x, \bar t)$ is strictly greater than $12\hat r^{-2} \geqslant 12\rho_0^{-2}$.

For constants $\theta\in [0,1)$ and $A>>1$ to be chosen later, we set 
$$P:=P_+(\bar x,\bar t,Ah,\theta h^2) $$ and take $\bar \delta\leqslant \bar \delta_{\rm per}(A,\theta,r)$, so that Theorem~\ref{thm:persistence-cap} applies. In particular, we have $R\geq {1\over 2}h^{-2}>6\rho_0^{-2}$ on $P$.
This implies that $P_0\cap P=\emptyset$.

We distinguish two subcases.

\textbf{Subcase i:}  $\gamma ([\bar t, \bar t+\theta h^2])\subset  B(\bar x, \bar t, Ah)$.

Then by Theorem~\ref{thm:persistence-cap},  $P$ is unscathed. Indeed, otherwise  $B_{g_+(\bar t)}(\bar x,Ah) \subset\Sigma_{t'}$ for some $t' \in (\bar t,\bar t+\theta h^2)$. Hence $\gamma(t') \in \Sigma_{t'}$, which contradicts our choice of $\bar t$.

\vspace{0,5cm}
\centerline{\input collapse7.pstex_t}
\vspace{0,5cm}

Moreover, $\bar t+\theta h^2\leq t_0-\alpha \rho_0^{2}$,
so $t_0-t\geq \alpha \rho_0^2$ for all $t\in [\bar t, \bar t+\theta h^2]$. Closeness with the standard solution implies
\begin{eqnarray*}
\int\limits_{\bar t}^{\bar t+\theta h^2}\sqrt{t_0-t}R\ dt & \geqslant &  {\constst\over 2}\int\limits_{\bar t}^{\bar t+\theta h^2}\sqrt{t_0-t}{h^{-2}\over 1-(t-\bar t)h^{-2}}\ dt\\
& \geqslant & \constst {\sqrt{\alpha \rho_0^2}\over 2}\int\limits_{\bar t}^{\bar t+\theta h^2}{h^{-2}\over 1-(t-\bar t)h^{-2}}\ dt\\
& = & \constst{\sqrt{\alpha \rho_0^2}\over 2}\int\limits_0^\theta {1\over 1-u}du\\
& \geqslant & - \constst{\sqrt{\alpha} \hat r \over 2}\ln (1-\theta ) \\
& \geqslant & \Lambda'\,,
\end{eqnarray*}
for $\theta$ close enough to $1$, depending only on $\hat r,\Delta, \Lambda$. 

We deduce
\begin{eqnarray*}
\int_{t_1}^{t_0} \sqrt{t_0-t}R\ dt & = & \int_{\bar t}^{\bar t + \theta h^2} 
\sqrt{t_0-t}R\ dt + \int_{t_1}^{\bar t} \sqrt{t_0-t}R\ dt + 
\int_{\bar t + \theta h^2}^{t_0} \sqrt{t_0-t}R\ dt\\
& \geqslant & \Lambda'- 4\Delta^{3/2} = \Lambda.
\end{eqnarray*}

Hence Inequality \eqref{ineq1} holds. Fix $\theta :=\theta (\hat r, \Delta, \Lambda)$ such that this condition is satisfied.

\textbf{Subcase ii: } There exists $t'\in [\bar t, \bar t+\theta h^2]$ such that $\gamma (t')\not\in B(\bar x, \bar t, Ah)$.

We assume that $t'$ is minimal with this property.
By the same argument as before, $P':=P+(\bar x,\bar t,Ah, t')$ is unscathed, and $A^{-1}$-close to the standard solution by the persistence theorem. As before, this implies that for all $s,s'\in (\bar t,t']$,
$$ g(s)\leq e^{C\over 1-\theta}g(s')\,,$$
where $C$ is a universal constant. Thus we have
$$\int_{\bar t}^{t'}\sqrt{\tau}|\dot\gamma |_{g(t)}^2\geq e^{-C\over 1-\theta}\Big (\int |\dot\gamma |_{g(\bar t)}\Big )^2\Big (\int {1\over \sqrt{\tau}}\Big )^{-1}\,.$$

Since $t'\leq t_0-\alpha \rho_0^2$ we bound $\tau$ from below by $\alpha \rho_0^2$ on $[\bar t, t']$, so
$$\int_{\bar t}^{t'}\sqrt{\tau}|\dot\gamma |^2\geq e^{-C\over 1-\theta }(Ah)^2\sqrt{\alpha \rho_0^2}(t'-\bar t)^{-1}\geq e^{-C\over 1-\theta }(Ah)^2
{\hat r\over 10}\sqrt{\alpha} {1\over \theta h^2}= e^{-C\over 1-\theta }\sqrt{\alpha}{\hat r A^2\over 10\theta}\,.$$
Fixing $A$ large enough, Inequality \eqref{ineq2} holds.
\end{proof}

A consequence of the previous lemma is the following result (see~\cite[Lemmas 78.3 and 78.6]{Kle-Lot:topology} or \cite{B3MP} for more details). 

\begin{lem}\label{existence-geodesique}
Let $\hat r,\Delta,\Lambda$ be positive numbers. There exists $\bar \delta:=\bar\delta (\hat r, \Delta, \Lambda )$ with the following property.
Let $g(\cdot)$ be a $(r,\delta)$-surgical solution defined on $[a,a+\Delta]$ such that $r\geqslant \hat r$ and $\delta\leq \bar\delta$. Let $(x_0,t_0)$ and $\rho_0\geqslant \hat r$ be such that $P_0:=P(x_0,t_0,\rho_0,-\rho_0^2)$ is unscathed and $|\Rm |\leq \rho_0^{-2}$ on $P_0$. Then
\begin{enumerate}
\item $\forall (q,t)\in M\times [a,a+\Delta]$, if $\ell (q,t_0-t)<\Lambda$, then there is an unscathed minimising $\mathcal L$-geodesic $\gamma$ connecting $x_0$ to $q$.
\item $\forall \tau>0,\,\,\min_q\ell (q,\tau )\leq 3/2$ and is attained.
\end{enumerate}
\end{lem}

We come back to the proof of Proposition C.
Recall that $P_0$ is unscathed, and satisfies $|\Rm |\leq \rho_0^{-2}$. 
The arguments of \cite[\S 7]{Per1} apply to unscathed minimising $\mathcal L$-geodesics. In particular, if $\gamma (\tau )=\mathcal L_\tau\exp_{(x_0,t_0)} (v)$ is minimising and unscathed on $[0,\tau_0]$, then
$ \tau^{-3/2}e^{-\ell(v,\tau )}J(v,\tau )$
is nonincreasing on $[0,\tau_0]$.

Define
$$Y(\tau ):=\{ v\in T_{x_0}M\,;\,\, \mathcal L\exp (v) : [0,\tau ]\longrightarrow M\textrm{ is minimising and unscathed }\}\,.$$
It is easy to check that $\tau\leq\tau'\Rightarrow Y(\tau )\supset Y(\tau' )$. Then we set
$${\widetilde V}_{\textrm{reg}}(\tau ):=\int_{Y(\tau )}\tau^{-3/2}e^{-\ell (v,\tau )}J(v,\tau )dv\,.$$
This function is nondecreasing on $[0,t_0]$. We shall adapt the proof of $\kappa$-noncollapsing in the smooth case, replacing $\widetilde V$ by ${\widetilde V}_{\textrm{reg}}$. Set $$\kappa :={\vol_{g(t_0)} B_0 \over \rho_0^3},\quad \tau_0=\kappa^{1/3}\rho_0^2\,.$$

\paragraph{Upper bound on ${\widetilde V}_{\textrm{reg}}(\tau_0)$:}
As in the smooth case, we get
$$\mathcal L_\tau\exp(\{ v\in Y(\tau )\,;\,\, |v|\leq {1\over 10}\kappa^{-1/6}\})
\subset B_0 $$
thus
$$I:=\int_{\{ v\in Y(\tau )\,;\,\, |v|\leq {1\over 10}\kappa^{-1/6}\}}\tau^{-3/2}e^{-\ell(v,\tau )}J(v,\tau )dv\leq e^{C\kappa^{1/6}}\sqrt{\kappa}$$
and
$$
\begin{array}{rl}
I' & :=\int\limits_{\{ v\in Y(\tau )\,;\,\, |v|\geq {1\over 10}\kappa^{-1/6}\}}\tau^{-3/2}e^{-\ell(v,\tau )}J(v,\tau )dv \\
\\
& \leq\int\limits_{\{ v\in Y(\tau )\,;\,\, |v|\geq {1\over 10}\kappa^{-1/6}\}}\lim_{\tau\to 0}(\tau^{-3/2}e^{-\ell(v,\tau )}J(v,\tau ))dv\\
\\
& \leq e^{-{1\over 10}\kappa^{-1/6}}\,.
\end{array}
$$
In conclusion,
$$\begin{array}{rl}
I+I' & \leq  e^{C\kappa^{1/3}}\sqrt{\kappa} + e^{-{1\over 10}\kappa^{-1/6}}\\
\\
& \leq C\sqrt{\kappa}
\end{array}
$$
for $\kappa\leq \kappa (3)$ and some universal $C$.

\paragraph{Lower bound for ${\widetilde V}_{\textrm{reg}}(t_0)$}
Monotonicity of ${\widetilde V}_{\textrm{reg}}(\tau_0 )$ implies
$${\widetilde V}_{\textrm{reg}}(\tau_0 )\geq {\widetilde V}_{\textrm{reg}}(t_0 )\,.$$
We are going to bound  ${\widetilde V}_{\textrm{reg}}(t_0 )$ from below as a function of $\vol B(q_0,0,1)$, where $q_0 \in M$.

Set $\Lambda:=21$ and apply Lemma~\ref{existence-geodesique} with parameter $\bar \delta (\hat r,T,\Lambda )$. There exists $q_0\in M$ such that $\ell (q_0, t_0-{1\over 16})\leq {3\over 2}$ and the first part of Lemma~\ref{existence-geodesique} gives us a minimising curve $\gamma$ connecting $x_0$ to $q_0$ realising the minimum $\ell (q_0,t_0-1/16)$. Let $q\in B(q_0, 0,1)$. Consider a curve $\gamma$ obtained by concatenating some $g(0)$-geodesic from $(q, 0)$ to $ (q_0,1/16)$ with some minimising $\mathcal L$-geodesic between $(q_0,1/16)$ and $(x_0,t_0)$. We have
\begin{multline}
\ell (q,t_0)={L(q,t_0)\over 2\sqrt{t_0}}\leq \\
{1\over2\sqrt{t_0}}L(q_0,t_0-1/16)+{1\over2\sqrt{t_0}}\int_{t_0-1/16}^{t_0}\sqrt{\tau}(R_{g(t_0-\tau )}+|\dot{\gamma} (\tau )|^2_{g(t_0-\tau)})d\tau\,,
\end{multline}
which leads to
$$\ell (q,t_0)\leq {3\over 2}+{1\over 2}\int_{t_0-1/16}^{t_0}(12+e^{1/2}|\dot\gamma (\tau )|^2_{g(0)})d\tau$$
since for $s\in [0,1/16]$ (i.e.~$\tau\in [t_0-1/16,t_0]$) the metrics $g(s)$ satisfy $\Rm \le 2$, hence  $R\le 12$, and are  $1/2$-Lipschitz equivalent. We obtain
$$\ell (q,t_0)\leq {3\over 2}+{1\over 2}({1\over 16}12+16e^{1/2}d^2_{g(0)}(q,q_0))\leq 20\,.$$
Hence $\ell(q,t_0) < \Lambda$, and Lemma~\ref{existence-geodesique} i) re-gives an unscathed minimising $\mathcal L$-geodesic  $\tilde\gamma$ connecting $x_0$ to $q$. Hence $q=\tilde\gamma (t_0)=\mathcal L_{t_0}\exp (\tilde v )$ for $\tilde v\in Y(t_0)$. This shows that $\mathcal L_{t_0}\exp (Y(t_0))\supset B(q,0,1)$.

Moreover, we have $\{ \tilde v\in Y(t_0)\,;\,\, \mathcal L_{t_0}\exp (\tilde v)\in B(q,0,1)\}$, $\ell\leq 10$, which implies
$$\begin{array}{rl}
{\widetilde V}_{\textrm{reg}}(t_0) & =\int_{Y(t_0)}t_0^{-3/2}e^{-\ell(v,t_0)}J(v,t_0)dv\geq
\int_{B(q_0,0,1)}t_0^{-3/2}e^{-20}dv_{g(0)}\\
\\
& \geq T^{-3/2}e^{-20}\vol_{g(0)} B(q_0,0,1)\,.
\end{array}
$$

\subsection{The case $\rho_0\leqslant \frac{r}{100}$} \label{grande-courbure}
In this case $\rho_{0}$ is below the scale of the canonical neighbourhood for all $t$ in the interval and hence any point with scalar curvature greater than $\rho_{0}^{-2}$ has a canonical neighbourhood. At such a point the noncollapsing is given by this neighbourhood if it is not $\epsi$-round.  

Since $\rho_0<1$, there exists $(y,t)\in \bar P_0$ such that $|\Rm (y,t)|=\rho_0^{-2}$. 
Hence we have 
$$|\Rm (y,t)| \geqslant  r^{-2} \geqslant 10^6.$$
Since $\{g(t)\}$ has curvature pinched toward positive,
$$R(y,t) \geqslant |\Rm (y,t)| = \rho_0^{-2} \geqslant 10000r(t)^{-2}.$$ 
Hence $(y,t)$ has an $(\epsi_0,C_0)$-canonical neighbourhood $U$. 

\paragraph{Case 1: $U$ is not $\epsi_0$-round}

Let us show that $B(x_0,t,e^{-2}\rho_0) \subset U$. This is clear if $U$ is closed, so we only have to deal with the cases of necks and caps.

By the curvature bounds on $\bar P_0$ we have $d_t(x_0,y)\leqslant e^2 \rho_0$ and
$B(x_0,t,e^{-2}\rho_0) \subset B(x_0,t_0,\rho_0)$. 
If $U$ is an $\epsi_0$-neck, then  $d_t(y,\bord \overline{U})  \geqslant (2\epsi_0)^{-1}R(y,t)^{-1/2}$. Since $R(y,t) \leqslant 6\rho_0^{-2}$, we get  $$d_t(y,\partial \overline{U}) \geqslant (2\sqrt{6}\epsi_{0})^{-1}\rho_0 \geqslant (e^2+e^{-2})\rho_0\,,$$ hence $B(x_0,t,e^{-2}\rho_0) \subset U$. 

If $U$  is an $(\epsi_0,C_0)$-cap, write it $U=V \cap W$ where $V$ is a core. Let $\gamma : [0,1] \rightarrow 
\bar B_0$ be a minimising $g(t_0)$-geodesic connecting $y$ to $x_0$. If $x_0 \notin V$,  let $s \in [0,1]$ be maximal such that $\gamma(s) \in \partial V$. Since
$\gamma(s) \in B_0$, we have $R(\gamma(s),t) \geqslant 6\rho_0^{-2}$ and we deduce that  $d(\gamma(s),\partial \overline{U}) \ge (\sqrt{6}\epsi_{0})^{-1}\rho_0$.  As $d_t(\gamma(s),x_0) \leqslant e^2\rho_0$ we get $B(x_0,t,e^{-2}\rho_0) \subset U$. 

Comparing this to Equation~\eqref{eq:vol}, we see that
$$\vol_{g(t)} B(x_0,t,e^{-2}\rho_0)\geqslant C_{0}^{-1} (e^{-2}\rho_0)^3\,.$$
By estimates on distortion of distances and volume as in the proof of  Lemma~\ref{lem-preliminaire}, we conclude that
$$\vol_{g(t_0)} B_0\geqslant C_{0}^{-1}e^{-18} \rho_0^3\,.$$

\paragraph{Case 2: $U$ is $\epsi_0$-round}

Note that the method of Case 1 applies equally well if $U$ is homeomorphic to $S^3$ or $RP^3$, so we assume it is not the case.

The only thing we have to do is to prove that there are only finitely many possible topologies for $U$. For simplicity of notation we assume $(x_0,t_0)=(y,t)$, i.e.~the point $(x_0,t_0)$ has an $\epsi_0$-round canonical neighbourhood $U$, and $|\Rm(x_0,t_0)| \ge 1000r^{-2}$.

\begin{lem}
There exists $t'_0<t_0$ such that
\begin{itemize}
\item $U$ is unscathed on $[t'_0,t_0]$;
\item for every $t\in [t'_0,t_0]$ $(U,g(t))$ is $\epsi_0$-round;
\item letting $\rho'_0$ be defined at $(x_0,t'_0)$ in the obvious way, we have $2r\ge \rho'_0\ge \frac{r}{2}$.
\end{itemize}
\end{lem}

\begin{proof}
Let $t''_0<t_0$ be minimal such that $U$ is unscathed and  for every $t\in [t''_0,t_0]$,  $(U,g(t))$ is $\epsi_0$-round and $R\ge r^{-2}$ on $(U,g(t))$. We claim that 
$\Rmin = r^{-2}$ on $(U,g(t''_0))$. Indeed by continuity $R \ge r^{-2}$ on $(U,g(t''_0))$. Hence $(x_{0},t''_0)$ has a canonical neighbourhood $V$. By continuity, $(U,g(t''_0))$ is $2\epsi_0$-round, so $V=U$; since we have excluded $S^3$ and $RP^3$, we deduce that $V$ is in fact $\epsi_0$-round. Since $\epsi_{0}$-roundness is an open property, it follows  that if $\Rmin >r{-2}$ on $(U,g(t''_0))$ then $t"_{0}$ is not minimal. This proves the claim. 

By $\epsi_{0}$-roundness, $R(\cdot,t''_0) \approx r^{-2}$ on $U$ and $|\Rm (\cdot,t''_0)| \approx r^{-2}/6$. Therefore we can find $t'_0\in (t''_0,t_0)$ such that $|\Rm (\cdot,t'_0)| \approx r^{-2}$ on $U$ and $|\Rm| \le r^{-2}$ on $P(x_{0},t'_{0},r,-r^{-2})$ (comparing with the evolving round metric one can find $t'_{0}$ close to $t"_{0}+\frac{5}{4}r^2$). It follows that the maximal radius 
$\rho'_{0}$ such that $|\Rm|\le {\rho'_{0}}^{-2}$ on $P(x_{0},t'_{0}, \rho'_{0}, -{\rho'_{0}}^{2})=:P_{0}$ with  $P_{0}$ unscathed, is close to $r$. 
\end{proof}

Since $\rho'_{0} \ge r/2$ we can argue as in subsection \ref{subsec:grand rayon} to get uniform noncollapsing at $(x_0,t'_0)$ on the unit scale. As $(U,g(t'_0))$ is $\epsi_0$-homothetic to $(U,g(t_0))$ and $\rho'_{0} \le 2r <1$, we also have uniform noncollapsing at $(x_0,t_0)$ on the unit scale.

\section{Generalisations and open questions}\label{sec:general}
\subsection{Consequences and generalisations}

First we state a finiteness result which follows immediately from Theorem~\ref{thm:existence 1 norm} and Corollary~\ref{corol:extinction}.
\begin{corol}\label{corol:finiteness}
Let $R_0,Q,\rho$ be positive numbers. Then the class of prime $3$-manifolds admitting complete riemannian metrics of scalar curvature greater than $R_0$, sectional curvature bounded in absolute value by $Q$, and injectivity radius greater than $\rho$ is finite up to diffeomorphism.
\end{corol}

Remark that the primeness hypothesis is necessary: otherwise, one could have, say, a connected sum of arbitrarily many copies of the same manifold. The key point is that the geometric bounds considered here do \emph{not} imply any diameter bound (nor compactness of the manifold for that matter.) Hence Corollary~\ref{corol:finiteness} is not a purely geometric finiteness theorem, but rather a mixed geometrico-topological finiteness theorem.

Next we discuss an equivariant version of our main technical theorem.

\begin{defi}
Let  $(M(\cdot),g(\cdot))$ be a surgical solution defined on some interval $I$. Let $\Gamma$ be a group endowed with an action on each $M(t)$ for $t\in I$, which is constant in between singular times. We say that $(M(\cdot),g(\cdot))$ is $\Gamma$-\bydef{equivariant} if for each $t$, the action of $\Gamma$ on $M(t)$ is isometric, and for each singular time $t$, the union of all $2$-spheres along which surgery is performed is $\Gamma$-invariant.
\end{defi}

\begin{theo}\label{thm:existence surg equi}
Let $M$ be an orientable $3$-manifold. Let $g_0$ be a complete riemannian metric on $M$
which has bounded geometry. Let $\Gamma$ be a group acting properly discontinuously on $M$ by isometries for $g_0$. Then there exists a complete
surgical solution $(M(\cdot),g(\cdot))$ of bounded geometry defined on $[0,+\infty)$, with initial condition $(M(0),g(0))=(M,g_0)$, and such that there is for each $t$ a properly discontinuous action of $\Gamma$ on $M(t)$ such that $(M(\cdot),g(\cdot))$ is $\Gamma$-equivariant, and such that if $t$ is a singular time and $x$ a point belonging to some disappearing component, then $(x,t)$ has an $(\epsi_0,C_0)$-canonical neighbourhood. Furthermore, if the action of $\Gamma$ on $M$ is free, then one can ensure that for each $t$, the action of $\Gamma$ on $M(t)$ is also free.
\end{theo}

\begin{proof}
We repeat the proof of Theorem~\ref{thm:existence 1 precis}, paying attention to equivariance with respect to the group $\Gamma$.
By the Chen-Zhu uniqueness theorem~\cite{Che-Zhu2}, Ricci flow automatically preserves the symmetries of the original metric, so the only thing to check is that surgery can be done equivariantly. For this we can apply~\cite[Lemma 3.9]{dl:equi}. Note that the constant $\epsilon$ appearing in that paper is \emph{a priori} smaller than our $\epsi_0$. However, it is easy to check that if we replace $\epsi_0$ by some smaller positive number $\epsi'_0$ in the proof of Theorem~\ref{thm:existence 1 precis} and subsequently the constants $\beta_0$ and $C_0$ by the appropriate constants $\beta'_0$ and $C'_0$, then the proof goes through without changes.

For the addendum where it is assumed that the action of $\Gamma$ is free, there is an additional point to check: that surgery can be done so that the action of $\Gamma$ on the post-surgery manifold is still free. For simplicity, we are going to explain this in a riemannian setting, ignoring the issue of strong necks, which is irrelevant here.

Let $(X,\tilde {g})$ be a $3$-manifold with an isometric, free, properly discontinuous action of $\Gamma$ and $\{N_i\}$ be a $\Gamma$-invariant, locally finite collection of pairwise disjoint $\delta$-necks in $X$. Let $(Y,g)$ be the quotient riemannian manifold $X/\Gamma$.

Suppose first that for each $N_i$ and each nontrivial element $\gamma\in\Gamma$ we have $N_i\cap \gamma N_i=\emptyset$. Then the collection $\{N_i\}$ projects to a  locally finite collection of pairwise disjoint $\delta$-necks in $Y$. Hence we can do metric surgery on $Y$, obtaining a riemannian manifold $(Y_+,g_+)$. We then lift the construction, getting a riemannian manifold $(X_+,\tilde{g_+})$ which on the one hand is obtained from $(X,\tilde{g})$ by metric surgery on $\{N_i\}$, and on the other hand inherits a free, properly discontinuous, isometric action of $\Gamma$.

Thus we are done unless there exists $i$ and $\gamma$ such that $N_i\cap \gamma N_i\neq\emptyset$. In this case, $N_i$ is invariant by $\gamma$. Since $\gamma$ acts freely, it must act on $N_i$ by an involution, so that $N_i$ projects to a cap $C\subset Y$ diffeomorphic to a punctured $RP^3$. In this case, $C$ contains, say a $4\delta$-neck whose preimage in $X$ contains two $4\delta$-necks interchanged by $\gamma$. Thus, up to replacing $\delta$ by $4\delta$, we can apply the construction of the first paragraph.
\end{proof}

\begin{corol}\label{corol:positive scalar universal}
Let $M$ be a connected, orientable $3$-manifold which carries a complete metric $g$ of uniformly positive scalar curvature. Assume that the riemannian universal cover of $(M,g)$ has bounded geometry. Then $M$ is a connected sum of spherical manifolds and copies of $S^2\times S^1$.
\end{corol}

\begin{proof}
We apply Theorem~\ref{thm:existence surg equi} to the universal cover of $(M,g)$ endowed with the action of $\Gamma:=\pi_1(M)$.
Let  $(\tilde{M}(\cdot),\tilde{g}(\cdot))$ be a surgical solution satisfying the conclusion of that theorem. By Corollary~\ref{corol:extinction}, this surgical solution must be extinct. Thus $M$ is a connected sum of  metric quotients of the disappearing components of  $(\tilde{M}(\cdot),\tilde{g}(\cdot))$. There remains to check that such quotients are themselves connected sums of spherical manifolds and copies of $S^2\times S^1$.

We use the fact that the disappearing components are covered by canonical neighbourhoods, and the action of $\Gamma$ on them is isometric. Let $X$ be such a component. Remark that $X$ is simply-connected, since the van Kampen theorem implies that  surgery along $2$-spheres on a simply-connected $3$-manifold produces simply-connected $3$-manifolds. If $X$ is compact, then by Perelman's Geometrisation Theorem, $X$ is diffeomorphic to the $3$-sphere, and its quotients are spherical manifolds.

If $X$ is noncompact, then it is diffeomorphic to $S^2\times\Rr$ or $\Rr^3$. In the former case, it is an exercise in topology (cf.~\cite{Sco}) to show that the quotient can only be $S^2\times\Rr$ itself, a punctured $RP^3$, $S^2\times S^1$, or a connected sum of two copies of $RP^3$.

In the latter case, we obviously need to use the geometry. As in~\cite[Section 3]{dl:equi}, we consider the open subset $T$ consisting of all points that are centres of $\epsi_0$-necks. Since $X$ is diffeomorphic to $\Rr^3$, $T$ is an $\epsi_0$-tube, and its complement $C$ is the core of an $\epsi_0$-cap and diffeomorphic to the $3$-ball. By definition, $T$ is automatically invariant by any isometry. Hence $C$ is also invariant by any isometry. Thus by the Brouwer fixed point theorem, $X$ does not admit any nontrivial free isometric group action.
\end{proof}

Here is a more precise theorem which may be useful for subsequent applications:
\begin{theo}\label{thm:existence 2 precis}
There exist sequences $r_k,\delta_k,\kappa_k>0$ such that for any complete normalised riemannian $3$-manifold $(M_0,g_0)$, there exists a surgical solution $(M(\cdot),g(\cdot))$ defined on $[0,+\infty)$, satisfying the initial condition \linebreak $(M(0),g(0))=(M_0,g_0)$, and such that for every nonnegative integer $k$, the restriction of  $(M(\cdot),g(\cdot))$ to $[k,k+1]$ is an $(r_k,\delta_k,\kappa_k)$-surgical solution.

Moreover, if $(M_0,g_0)$ is endowed with a properly discontinuous isometric action of some group $\Gamma$, then the surgical solution can be made $\Gamma$-equivariant. In addition, if the action of $\Gamma$ on $M_0$ is free, then the action on each $M(t)$ can be chosen to be free.
\end{theo}

This follows from iteration of Theorem~\ref{thm:existence 1 precis}.
Indeed, assuming the parameters $r_k,\delta_k,\kappa_k>0$ are known,
we deduce from $\Theta_k:=\Theta(r_k,\delta_k)$ a bound for the sectional curvature $Q_k$. From this and the $\kappa_k$-noncollapsing property, we deduce a lower bound for volumes of balls of radius at most $Q_k^{-1/2}$. This gives a lower bound $\rho_k$ for the injectivity radius of every metric, in particular the metric $g(k+1)$.

The addendum about equivariance follows as explained in the proof of Theorem~\ref{thm:existence surg equi}.

\subsection{Open questions}

The first question asks whether the hypotheses of
Corollary~\ref{corol:positive scalar universal} are necessary.

\begin{question}
Let $M$ be a connected, orientable $3$-manifold which admits a complete riemannian metric of uniformly positive scalar curvature. Is $M$ a connected sum of spherical manifolds and copies of $S^2\times S^1$ ?
\end{question}

Next we consider what happens when we relax the hypothesis on the scalar curvature from uniform positivity to positivity. This class is significantly wider, e.g.~it includes $S^1\times \Rr^2$. One could even relax the condition further to nonnegativity.

\begin{question}[Problem 27 in \cite{Yau}]
Classify $3$-manifolds admitting complete riemannian metrics of positive (resp.~nonnegative) scalar curvature up to diffeomorphism.
\end{question}

\appendix

\section{Hamilton's compactness theorem}
A \bydef{pointed evolving metric} is a triple $(M,\{g(t)\}_{t\in I},(x_0,t_0))$ where $M$ is a manifold, $g(\cdot)$ is an evolving metric on $M$, and $(x_0,t_0)$ belongs to $M\times I$. 
We say that a sequence of pointed evolving metrics $(M_k,\{g_k(t)\}_{t\in I},(x_k,t_0))$ \bydef{converges smoothly} to a pointed evolving metric $(M_\infty,\{g_\infty(t)\}_{t\in I},(x_\infty,t_0))$ if there 
exists an exhaustion of $M$ by open sets $U_{k}$, such that $x \in U_{k}$ for all $k$, and smooth embeddings $\psi_{k} : U_{k} \rightarrow M_{k}$ 
sending $x$ to $x_{k}$ and
such that $\psi_{k}^*g_{k}(\cdot)-g(\cdot)$ and all its derivatives converge to zero uniformly on compact subsets of $M\times I$.

\begin{theo}[Hamilton's compactness]\label{thm:hamilton compactness}
Let $(M_k , \{g_k (t)\}_{t\in (a,b]},(x_k,t_0))$ be a sequence of complete
pointed Ricci flows of the same dimension. Assume:
\begin{enumerate}
\item For all $\rho>0$ ,
$$\sup_{k\in \NN} \sup_{B(x_k,t_0,\rho)\times (a,b]} | \Rm| < + \infty$$
\item $$\inf_{k\in \NN}  inj(M_{k},g_{k}(t_{0}),x_{k}) > 0.$$
\end{enumerate}
Then $(M_k , \{g_k (t)\}_{t\in (a,b]},(x_k,t_0))$ converges smoothly  to a complete Ricci flow of the same
 dimension, defined on $(a,b]$.
\end{theo}

\rem If $g(t)$ is defined on $[a,b]$, one can take  $t_{0}=a$ if one has also uniform bounds on the derivatives of the curvature operator at time  $t_{0}$, i.e.~if for any $\rho >0$, for any integer $p$, $\sup_{k\in \NN} \sup_{B(x_k,t_0,\rho)\times \{t_{0}\}} | \nabla^p \Rm| < + \infty$.\\

\section{Partial Ricci flows}\label{sec:partial Ricci}

\begin{defi} Let $(a,b]$ a time interval. A  \emph{partial Ricci flow}\footnote{This definition differs from that of \emph{local Ricci flow} introduced by 
D. Yang \cite{Yan}.} 
on $U\times (a,b]$ is a pair $(\mathcal P,g(\cdot,\cdot))$, where $\mathcal{P} \subset U\times (a,b]$  is an open subset which contains 
$U\times\{b\}$ and  $(x,t) \mapsto g(x,t)$ is a smooth map defined on $\mathcal{P}$ such that the 
restriction of $g$ to any subset $V\times I \subset \mathcal{P}$ is a Ricci flow on $V\times I$. 
\end{defi}

Inspection of the proof of Theorem \ref{thm:hamilton compactness} shows that the following natural extension holds. 

\begin{theo}[Local compactness for flows]\label{thm:local compactness}
Let $(U_k , \{g_k (t)\}_{t\in (a,0]},(x_k,0))$ be a sequence of pointed Ricci flows of the same dimension. 
Suppose that for some $\rho_{0} \in (0,+\infty]$, all the balls $B(x_{k},0,\rho)$ of radius $\rho < \rho_{0}$ are relatively compact 
in $U_{k}$ and that the following holds : 
\begin{enumerate}
\item For any $\rho\in (0,\rho_{0})$, there exists $\Lambda(\rho) <+\infty$ and $\tau(\rho)>0$ such that  
$\vert \Rm \vert < \Lambda(\rho)$ on all $P(x_{k},0,\rho,-\tau(\rho))$.
\item $$\inf_{k\in \NN}  inj(U_{k},g_{k}(0),x_{k}) > 0.$$
\end{enumerate}
Then there is a riemannian ball $B(x_{\infty},\rho_{0})$ of the same dimension
such that the pointed sequence $(B(x_k,0,\rho_{0}), g_k (\cdot), x_{k})$ subconverges smoothly to a partial Ricci flow $g_{\infty}(\cdot)$ defined on 
$\bigcup_{\rho < \rho_{0}} (B(x_{\infty},\rho) \times (-\tau(\rho),0])$. 
Moreover, if $\rho_{0}=+\infty$ then for any $t\in [\sup_{\rho} -\tau(\rho),0]$, $g_{\infty}(t)$ is complete.
\end{theo}

\bibliographystyle{alpha}
\bibliography{ricci}

\end{document}